\documentclass[12pt, reqno, a4paper,oneside]{amsart}
\usepackage{graphicx, epsfig, psfrag}
\usepackage{amsfonts,amssymb,amsbsy,amsthm,amsmath}
\usepackage[T1]{fontenc}
\usepackage{bm}
\usepackage{color}
\usepackage{float}
\usepackage{hyperref}
\usepackage{mathrsfs}
\usepackage{subfigure}
\usepackage{longtable}
\usepackage{graphicx}
\usepackage{bbm}
\usepackage{mathtools}
\usepackage{stmaryrd}
\usepackage[all,cmtip]{xy}
\usepackage{tikz}
\usetikzlibrary{matrix,arrows,decorations.pathmorphing}

\usepackage{ulem}
\usepackage[top=2cm,bottom=2cm,left=1.5cm,right=1.5cm]{geometry}

\usepackage{environments}
\numberwithin{theorem}{section}
\numberwithin{equation}{section}
\newcommand{\alth}[1]{{\color{blue}#1}}

\renewcommand{\cases}[1]{\left\{ \begin{array}{rl} #1 \end{array} \right.}
\newcommand{\smfrac}[2]{{\textstyle \frac{#1}{#2}}}

\def\L{\mathsf{L}}
\def\CpL{\L_p}
\def\CpLst{\L^*_p}

\def\Tr{\mathsf{Tr}}
\def\Sq{\mathsf{Sq}}
\def\Hx{\mathsf{Hx}}

\def\V{\mathcal{V}}
\def\Num{\mathcal{K}}

\def\Out{\mathrm{Ext}}
\def\Int{\mathrm{Int}}

\def\A{\mathsf{A}}
\def\Dom{\mathsf{D}}
\def\CpD{\Dom_{n,p}}
\def\CpDst{\Dom_{n,p}^*}

\def\df{\mathbf{d}}
\def\codf{\boldsymbol{\delta}}
\def\dfst{\mathbf{d}^*}
\def\codfst{\boldsymbol{\delta}^*}
\def\Lap{\boldsymbol{\Delta}}

\def\dist{\mathrm{dist}}
\def\diam{\mathrm{diam}}

\def\Wsc{\mathscr{W}}
\def\Wscz{\mathscr{W}_0}
\def\Lsc{\mathscr{L}}

\def\Gc{\mathcal{G}}
\def\Gd{G^\L}
\def\Gdst{G^{\L^*}}

\def\Quot{\mathscr{Q}}
\def\llb{\llbracket}
\def\rrb{\rrbracket}

\def\Pos{\mathscr{M}}
\def\Posen{\Pos^\eps_n}
\def\Posc{\Pos^\eps_\infty}

\def\Prob{\mathbb{P}}

\def\I{\mathcal{I}}
\def\Rate{\mathcal{R}}
\def\Ent{\mathcal{A}}
\def\Tran{\mathcal{B}}
\def\tscale{\mathcal{T}}
\def\Ham{\mathcal{H}}
\def\Lag{\mathcal{L}}
\def\Act{\mathcal{J}}
\def\Nhd{\mathcal{N}}

\def\e{\mathrm{e}}

\def\avec{\mathsf{a}}
\def\svec{\mathsf{s}}
\def\evec{\mathsf{e}}
\def\xvec{\mathsf{x}}

\DeclareMathOperator*{\Expo}{{\rm Exp}}

\def\R{\mathbb{R}}
\def\C{\mathbb{C}}
\def\N{\mathbb{N}}
\def\Z{\mathbb{Z}}

\def\WW{\mathrm{W}}
\def\CC{\mathrm{C}}

\def\DD{\mathrm{D}}

\def\<{\langle}
\def\>{\rangle}

\def\quotient#1#2{%
    \raise1ex\hbox{$#1$}\big/\lower1ex\hbox{$#2$}%
}

\def\b{\big}
\def\B{\Big}
\def\bg{\bigg}

\def\Us{\mathscr{W}}

\def\chr{\mathbbm{1}}

\def\mR{{\sf R}}

\def\sep{\,|\,}
\def\bsep{\,\b|\,}
\def\Bsep{\,\B|\,}

\def\del{\delta}

\def\dx{\,{\rm d}x}

\def\dt{\,{\rm d}t}

\def\AXint#1#2#3{{\setbox0=\hbox{$#1{#2#3}{\int}$}
\vcenter{\hbox{$#2#3$}}\kern-.5\wd0}}

\def\E{\mathcal{E}}

\def\eps{\epsilon}

\DeclareMathOperator*{\argmin}{{\rm argmin}}
\DeclareMathOperator*{\argmax}{{\rm argmax}}

\def\supp{{\rm supp}}
\def\div{{\rm div}}
\def\curl{{\rm curl}}

\title[Upscaling thermally--driven dislocation motion]{Upscaling
a model for the thermally--driven motion of screw dislocations}

\author{T. Hudson}
\address{T. Hudson \\ CERMICS, \'Ecole des Ponts ParisTech \\  
6 et 8, Avenue Blaise Pascal \\ 77455 Champs--sur--Marne
\\ France}
\email{hudsont@cermics.enpc.fr}

\date{\today}


\begin{document}

\begin{abstract}
\alth{
We formulate and study a stochastic model for the thermally--driven
motion of interacting straight screw dislocations in a cylindrical
domain with a convex polygonal cross--section.
Motion is modelled as a Markov jump process, where waiting times for
transitions from state to state are assumed to be exponentially
distributed with rates expressed in terms of the potential energy
barrier between the states. Assuming the energy of the system is
described by a discrete lattice model, a precise asymptotic description
of the energy barriers between states is obtained.
Through scaling of the various physical
constants, two dimensionless parameters are identified which govern
the behaviour of the resulting stochastic evolution. In an asymptotic
regime where these parameters remain fixed, the process is found to
satisfy a Large Deviations Principle. A sufficiently explicit
description of the corresponding rate functional is obtained such that
the most probable path of the dislocation configuration may be described
as the solution of Discrete Dislocation Dynamics with an explicit
anisotropic mobility which depends on the underlying lattice structure.
}
\end{abstract}

\maketitle

\section{Introduction}
Dislocations are topological line defects whose motion
is a key factor in the plastic behaviour of crystalline solids.
After their existence was hypothesised in order to explain a
discrepancy between predicted and observed yield stress in metals
\cite{Orowan34,Polanyi34,Taylor34},
they were subsequently experimentally identified in the 1950s
via electron microscopy \cite{HirschHornWhelan56,Bollmann56}.
Dislocations are typically described by a curve in the crystal, called
the \emph{dislocation line}, which is where the resulting distortion is
most concentrated, and their \emph{Burgers vector}, which reflects the
mismatch in the lattice they induce \cite{HirthLothe}.

Although the discovery of dislocations is now over 80
years distant, the study of these objects remains of
significant interest to Materials Scientists and Engineers today.
In particular, a cubic centimetre of a metallic solid may
contain between $10^5$ and $10^9$m of dislocation lines 
\cite{HullBacon11}, leading to a dense networked geometry,
and inducing complex stress fields in the material which 
are relatively poorly understood. Accurately modelling the
behaviour of dislocations therefore remains a major hurdle to
obtaining predictive models of plasticity on a single crystal
scale.

In this work, we propose and study a discrete stochastic model
for the thermally--driven motion of interacting straight screw 
dislocations in a cylindrical crystal of finite diameter. The basic
assumptions of this model are that all
screw dislocations are aligned with the axis of the
cylinder, and that the motion of dislocations
proceeds by random jumps between `adjacent' equilibria, with 
the rate of jumps being governed by the temperature and the
\emph{energy barrier} between states: this is the
minimal additional potential energy which must be gained in order
to pass to from one state to another. To describe the system, we
prescribe a lattice energy functional, variants of which
have been extensively studied in recent literature
\cite{Ponsiglione07,HO14,ADLGP14,HO15,ADLGP16}.

By rescaling the model in space and time, we identify two
dimensionless parameters, and with a specific family of scalings
corresponding to a regime in which dislocations are
dilute relative to the lattice spacing, the time over 
which the system is observed is long and the system 
temperature is low, we find we may apply the theory of Large
Deviations described in \cite{FK06}
to obtain a mesoscopic evolution law for the most probable
trajectory of a dislocation configuration.

\alth{
The major novelties of this work are the demonstration
of uniqueness (up to symmetries of the model) of equilibria containing
dislocations, a precise asymptotic characterisation of the energy
barriers between dislocation configurations, and the rigorous
identification of both a parameter regime in which the two--dimensional 
Discrete Dislocation Dynamics framework
\cite{AG90,vdGN95,BC04,BulatovCai} is valid, as well as a new set of
explicit nonlinear anisotropic mobilities which depend upon the
underlying lattice structure.
The nonlinearity and anisotropy of the mobilities obtained is in
contrast to the linear isotropic mobility often assumed in Discrete
Dislocation Dynamics simulations.
}

\subsection{Kinetic Monte Carlo models}
The stochastic model we formulate is based
on the observation that at low temperatures, thermally--driven 
particle systems spend long periods of time close to local
equilibria, or \emph{metastable states}, before transitioning to
adjacent states, and repeating the same process. It is a
classical assertion that such transitions are approximately 
exponentially distributed at low temperatures, with a rate which
depends upon the temperature and energy barrier which must be
overcome to pass into a new state; the
\emph{transition rate} from state $\mu$ to state $\nu$,
$\Rate(\mu\to\nu)$, is given approximately by the formula
\begin{equation}
  \Rate(\mu\to\nu)= \Ent(\mu\to\nu)\,\e^{-\beta \Tran(\mu\to\nu)},\label{eq:rate}
\end{equation}
where
\begin{itemize}
  \item $\beta := (k_BT)^{-1}$ is the inverse of the thermodynamic
    temperature of the system, with $k_B$ being Boltzmann's
    constant and $T$ being the absolute temperature;
  \item $\Tran(\mu\to\nu)$ is the \emph{energy barrier}, that is, the
    additional potential energy relative
    to the energy at state $\mu$ that the system must acquire in
    order to pass to the state $\nu$; and
  \item $\Ent(\mu\to\nu)$ is the entropic \emph{prefactor} which is
    related to the `width' of the pathway by which the system may pass
    from the state $\mu$ to the state $\nu$ with minimal potential
    energy.
\end{itemize}
\alth{
The discovery and refinement of the rate formula \eqref{eq:rate}
is ascribed to Arrhenius \cite{A89}, Eyring \cite{E35}, and Kramers
\cite{K40}, and a review of the physics literature on this subject may
be found in \cite{HTB90}. For It\^o SDEs with small noise (the usual
mathematical interpretation of the correct low--temperature dynamics of 
a particle system) \eqref{eq:rate} has
recently been rigorously validated in the mathematical literature: for
a review of recent progress on this subject, we refer the reader to
\cite{B13}.
}

We may use the observation above to generate a simple
coarse--grained model for the thermally--driven evolution of a particle 
system. Begin by labelling the local equilibria of the system, $\mu$,
and prescribe a set of neighbouring equilibria $\Nhd_\mu$ which may
be accessed from $\mu$, along with the transition rates
$\Rate(\mu\to\nu)$, for $\nu\in\Nhd_\mu$. Given that the system is in a 
state $\mu$ at time $0$, we model a transition from $\mu$ to a new state
$\nu'\in\Nhd_\mu$ as a jump at a random time $\tau$, where
\begin{gather*}
  \tau \sim \min_{\nu\in\Nhd_\mu}\Expo\b(\Rate(\mu\to\nu)\b)=
  \Expo\B(\sum_{\nu\in\Nhd_\mu}\Rate(\mu\to\nu)
  \B)\\
  \text{and}\quad\Prob[ \mu\to\nu'\sep t=\tau ] = 
  \frac{\Rate(\mu\to\nu')}{\sum_{\nu\in\Nhd_\mu} 
  \Rate(\mu\to\nu)}.
\end{gather*}
This defines a Markov jump process on the set of all states:
such processes are sometimes called Kinetic Monte Carlo (KMC) models,
and are  highly computationally efficient for certain problems in
Materials Science \cite{Voter07}. As an example of
their use, KMC models have recently been particularly
successful in the study of pattern formation during epitaxial
growth \cite{BSS14,SSW03}. Due to the ease with which samples from
exponential random variables may be computed, KMC models
allow attainment of significantly longer timescales than Molecular
Dynamics simulations of a particle system, with the tradeoff being
that fine detail on the precise mechanisms by which phenomena occur may
be lost.

A major hurdle in the prescription of a computational KMC model is
the definition of the rates $\mathscr{R}(\mu\to\nu)$. In practice,
these must be derived or pre--computed by some means, normally via a
costly \emph{ab initio} or Molecular Dynamics computation run on the
underlying particle system to be approximated.
Likewise, a large part of the analysis we undertake here is
devoted to rigorously deriving an asymptotic expression for
energy barrier $\Tran(\mu\to\nu)$, which then informs
our choice of $\Rate(\mu\to\nu)$ using formula \eqref{eq:rate}.

\subsection{Modeling screw dislocations}
In order to use the KMC framework described above to model the motion
of dislocations, we must give an energetic description of the system
which allows us to define both corresponding metastable states $\mu$
and the energy barriers $\Tran(\mu\to\nu)$.
In several recent works \cite{Ponsiglione07,HO14,ADLGP14,HO15,ADLGP16},
variants of an anti--plane lattice model have been studied in
which the notion of the energy of a configuration of straight screw
dislocations can be made mathematically precise, and in which screw
dislocations may be
identified using the topological framework
described in \cite{AO05}. Here, we will follow \cite{HO14,HO15} in
considering the energy difference
\begin{equation*}
  E_n(y;\tilde{y}):= \sum_{e\in \Dom_{n,1}}\b[\psi(\df 
  y(e))-\psi(\df \tilde{y}(e))\b],
\end{equation*}
which compares the energy of deformations $y$ and $\tilde{y}$
of a long cylindrical crystal with cross--section $n\Dom$:
the scaled cross--section
$n\Dom\subset\R^2$ is a convex lattice polygon in either the
square, triangular or hexagonal lattice, $\Dom_{n,1}$ denotes a set
of pairs of interacting columns of atoms,
$\df$ is a finite difference operator, $y$ and $\tilde{y}$
are anti--plane displacement fields, and $\psi$ is a periodic
inter--column interaction potential, here taken to be
$\psi(s):=\smfrac12\lambda\,\dist(s,\Z)^2$.

We define a \emph{locally stable equilibrium} to be a displacement $y$
such that $u=0$ minimises $E_n(y+u;y)$ among all perturbations which are
sufficiently small in the energy norm
\alth{
\begin{equation*}
  \|u\|_{1,2}:=\bg(\sum_{e\in \Dom_{n,1}}|\df u(e)|^2\bg)
  ^{1/2}.
\end{equation*}
}
Configurations containing dislocations are identified by
considering \emph{bond--length 1--forms} associated with
$\df y$, the definition of which is recalled in
\S\ref{sec:disl_configs}. In analogy with the procedure described in
\S1.3 of \cite{HirthLothe}, this construction allows us to define the
Burgers vector in a region of the crystal subject to the deformation $y$
as the integral of the bond--length
1--form around the boundary of the region. This defines a
field $\mu$, which we call the \emph{dislocation configuration},
and we say that the displacement field $y$ \emph{contains the
dislocations} $\mu$.

\alth{
The results of \cite{HO15,ADLGP14,ADLGP16}
demonstrate that there are a large number of locally stable
equilibria in this model which contain dislocations for a range
of underlying lattice structures. Nevertheless, since these
existence results are ultimately all based upon compactness methods,
they do not provide a fine description of the equilibria, nor a
guarantee of uniqueness up to lattice symmetries, for a fixed
choice of the dislocation configuration $\mu$.
The first achievement of this work is therefore
Theorem~\ref{th:equivalence},
which provides a novel
construction of the equilibria corresponding to dislocation
configurations in the particular case where
$\psi(s):=\smfrac12\lambda\,\dist(s,\Z)^2$. This construction
uses a form of lattice duality to show that these minima may be
characterised as the `discrete harmonic
conjugate' (interpreted in an appropriate sense) of
lattice Green's functions satisfying Dirichlet boundary conditions
on a finite lattice domain.
In particular, this representation enables us to show that, given
a dislocation configuration, there exist corresponding
equilibria which are unique up to lattice symmetries.
}

\alth{
\subsection{Energy barriers}
For two dislocation configurations $\mu$ and $\nu$,
we define the \emph{energy barrier} for the transition from $\mu$ to
$\nu$ as
\begin{equation*}
  \Tran_n(\mu\to\nu):=\min_{\gamma\in\Gamma_n(\mu\to\nu)}
  \max_{t\in[0,1]} E_n(\gamma(t);u_\mu),
\end{equation*}
where $u_\mu,u_\nu$ are locally stable equilibria containing 
dislocation configurations $\mu$ and $\nu$ respectively, and
$\Gamma_n(\mu\to\nu)$ is the space of continuous paths 
connecting these equilibria. The second major achievement of this
work is Theorem~\ref{th:energy_barriers}, which gives a
precise asymptotic formula for $\Tran_n(\mu\to\nu)$ as the domain
and dislocation configuration are scaled, in terms of the gradient
of the \emph{renormalised energy} \cite{CG05,SS12,ADLGP14}.
In the course of proving this result, in
\S\ref{sec:barrier_proofs} we constructively demonstrate the
existence of \emph{transition states} $u_\uparrow$, such that
\begin{equation*}
  E_n(u_\uparrow;u_\mu) = \Tran_n(\mu\to\nu).
\end{equation*}
The construction of $u_\uparrow$ again uses the form of lattice duality
we describe and lattice Green's functions on the finite domain.
Moreover, the properties of Green's functions
allow us to compute $\Tran_n(\mu\to\nu)$ explicitly in terms of a
single finite difference of the dual lattice Green's function. In
Theorem~\ref{th:interior_asymptotics}, we obtain a precise asymptotic
description of this finite difference in terms of the gradient of
the continuum renormalised energy as the domain is rescaled, and hence
to provide an asymptotic formula for $\Tran_n(\mu\to\nu)$.
Our strategy for proving Theorem~\ref{th:interior_asymptotics} is t
develop a theory akin to the  classical gradient estimates for
solutions of Poisson's equation
(see \S3.4 of \cite{GilbargTrudinger}) in a discrete setting.
}

\subsection{Upscaling via a Large Deviations Principle}
\alth{
Once we have obtained the asymptotic representation of
$\Tran_n(\mu\to\nu)$ given in Theorem~\ref{th:energy_barriers},
we apply formula \eqref{eq:rate} to
define the rates $\Rate_n(\mu\to\nu)$ and hence the stochastic model
considered. We then seek to understand the behaviour of
this model in the regime where the distance between dislocations is
significantly larger than the lattice spacing.
Scaling the various physical constants inherent in the
model enables us to identify two non--dimensional constants
which govern the evolution.

Fixing these constants leads us to consider the asymptotic regime
in which the temperature is low, the diameter of the cylindrical
domain and the spacing between dislocations is large relative with
the lattice spacing, and the time over which the process is
observed is long. In this regime, we find that the processes satisfy a
\emph{Large Deviations Principle}, which provides a means of
describing the asymptotic probability of rare events in
random processes. A general theoretical framework for proving
such results has been developed over the last 50 years, and
major treatises on the subject describing a variety of approaches
include \cite{FW12,Ellis06,DZ10,FK06}.
}

More precisely, a sequence of random variables $X^n$ taking values
on a metric space $M$ is said to satisfy a Large Deviations Principle
if there exists a lower semicontinuous functional
$\I:M\to[0,+\infty]$ such that for any open set $A\subseteq M$,
\begin{equation*}
  \liminf_{n\to\infty}\smfrac1n\log\Prob[X^n\in A] \geq
  -\inf_{x\in A} \I(x),
\end{equation*}
and for any closed set $B\subseteq M$, we have that
\begin{equation*}
  \limsup_{n\to\infty} \smfrac1n\log\Prob[X^n\in B] \leq
  -\inf_{x\in B} \I(x).
\end{equation*}
The function $\I$ is called the \emph{rate function} of the 
Large Deviations Principle, and is called \emph{good} if each 
of the sub--level sets $\{ x\sep \I(x)\leq a\}$ for $a\in\R$
is compact in $M$ (a property normally referred to as
\emph{coercivity} in the Calculus of Variations literature).
The existence of a Large Deviations Principle may be interpreted
as saying that, for any Borel set $A$,
\begin{equation*}
  \Prob[X^n\in A] \simeq \exp\b(-n \inf_{x\in A} \I(x)\b),
  \quad\text{as }n\to\infty,
\end{equation*}
i.e. the probability of observing events disjoint from $\I^{-1}(0)$
becomes exponentially small as $n\to\infty$.

\alth{
In the setting considered here, the random variables $X^n$ correspond to
trajectories of the dislocation configuration through
an appropriate state space. In order to prove a Large Deviations
Principle, we apply the theory developed in \cite{FK06} and summarise
the main results of this treatise in a form suited to our application in
Theorem~\ref{th:LDP_synthesis}. The existence of a Large
Deviations Principle is then asserted in Theorem~\ref{th:LDP_main},
which also gives an explicit description of the
`most probable' trajectory of the system. This trajectory corresponds
to a solution of the equations usually simulated in the study of
Discrete Dislocation Dynamics \cite{AG90,BC04,BulatovCai}, with
an explicit anisotropic mobility function $\mathcal{M}^\L_{A,B}$ which
depends upon the underlying lattice structure.
}

We conclude our study by discussing the interpretation of this result,
and show that the additional regimes identified in \cite{BP15} also
apply here: in particular, we show it is possible to recover the
linear gradient flow structure normally used in Discrete Dislocation
Dynamics simulations \cite{AG90,BC04,BulatovCai} in a further
parametric limit, but we argue that in the appropriate parameter
regimes, a stochastic evolution problem may be more appropriate to
model dislocation evolution.

\subsection{Structure and notation}
\alth{
In order to give a precise statement
of our main results, \S\ref{sec:preliminaries} is devoted to describing
the geometric framework which is both used to describe the Burgers
vector of a lattice deformation in our model and the notion of duality
which we use in the subsequent analysis.

In \S\ref{sec:barrier_results}, we state and discuss our main results.
These are Theorem~\ref{th:equivalence}, which
characterises equilibria containing dislocations,
Theorem~\ref{th:energy_barriers}, which provides a precise asymptotic
formula for the energy barrier between equilibria, and
Theorem~\ref{th:LDP_main}, which asserts the existence of a Large
Deviations Principle for the Markov processes and asymptotic regime we
consider. The proofs of these
results are given in \S\ref{sec:elliptic_estimates},
\S\ref{sec:barrier_proofs} and \S\ref{sec:LDP_proofs} respectively.

Since we introduce significant amounts of notation in order to
concisely state our results, Table~\ref{tab:notation} is provided for
convenience.
}

\begin{table}[p]
\centering
\caption{Notation conventions.}
\begin{tabular}{| c | p{14cm} |}
\hline
\textbf{Symbol} & \textbf{Description} \\
\hline
 $\L$ &$m$--dimensional multilattice identified with
  a lattice complex\\[1pt]
  $\Tr,\Sq,\Hx$ &Triangular, square
  and hexagonal lattices in $\R^2$\\[1pt]
  $\Num,\Num^*,\V,\V^*$&Constants depending on $\L$\\[1pt]
  $\Dom$ &Convex lattice polygon\\[1pt]
  $c_l, \varphi_l$ &Position and interior angle of corner $l$ of
                     $\Dom$\\[1pt]
  \hline
  $\CpL$,$\CpLst$ &Set of $p$--cells in the  primal and dual 
  lattice complexes induced by $\L$\\[1pt]
  $\CpD$,$\CpD^*$ &Set of $p$--cells in the primal and dual
  lattice subcomplexes induced by $n\Dom$\\[1pt]
  $\Out(\CpD)$&Set of $p$--cells in $\CpD$ 
  at the `edge' of the complex induced by $n\Dom$\\[1pt]
  $\Int(\CpD)$&Set of $p$--cells in $\CpD$ 
  which lie `away from the edge of $n\Dom$'\\[1pt]
  $e$&A $p$--cell\\[1pt]
  $[e_0,e_1]$&$1$--cell $e$ such that $\partial e = e_1\cup-e_0$.
  \\[1pt]
  $e+\avec$&$p$--cell obtained by translating $e$ by the vector
             $\avec$\\[1pt]
  $\partial,\delta$  &Boundary and coboundary operators\\[1pt]
  $\df,\codf$ &Differential and codifferential
  on forms defined on the lattice complex\\[1pt]
  $\Lap$ &Hodge Laplacian on forms\\[1pt]
  $\Wsc(\CpL)$,$\Wsc(\CpD)$&Set of $p$--forms on $\L$ and 
  $\Dom$\\[1pt]
  $\Lsc^2(\CpL)$&Hilbert space of square--integrable
  $p$--forms on $\L$\\[1pt]
  $\Wscz(\CpD)$&Set of $p$--forms on $\Dom$ which
  vanish on $\Out(\CpD)$\\[1pt]
  $(\cdot)^*$&Duality mapping on $p$--cells and
  $p$--forms\\[1pt]
  \hline
  $E_n(y;\tilde{y})$ & Energy difference between deformations $y$
  and $\tilde{y}$\\[1pt]
  $\psi$ & Potential giving energy per unit length of interaction 
  between columns of atoms\\[1pt]
  $[\df u]$ & Set of bond--length 1--forms corresponding to 
  $\df u$\\[1pt]
  $u_\mu$ & Locally stable equilibrium containing dislocations 
  $\mu$\\[1pt]
  $\Rate_n(\mu\to\nu)$ & Exponential transition rate to pass from 
                         $\mu$ to $\nu$\\[1pt]
  $\Ent_n(\mu\to\nu)$ & Entropic prefactor for transition from $\mu$ to $\nu$\\[1pt]
  $\Tran_n(\mu\to\nu)$ & Potential energy barrier to transition from
                         $\mu$ to $\nu$\\[1pt]
  $\Gamma_n(\mu\to\nu)$ & Space of paths in deformation space
  connecting $u_\mu$ and $u_\nu$\\[1pt]
  $u_\uparrow$ & Transition state, i.e. deformation where
                 $\Tran_n(\mu\to\nu)$ is attained\\[1pt]
  $\alpha_\uparrow,\alpha_\downarrow$ & Bond--length 1--forms corresponding
  to the transition state\\[1pt]
  \hline
  $\chr_e$ & $p$--form which is $\pm1$ on $\pm e$ and $0$
  otherwise\\[1pt]
  $\Gd$ &Green's function for the full lattice $\L$\\[1pt]
  $Q^r$&Polygonal set of radius $r$ in the
  lattice\\[1pt]
  $\omega^r_e$&Harmonic measure for $Q^r$ evaluated
  at $e\in\Out(Q^r_0)$\\[1pt]
  $G_{\mu^*}$ & Solution to $\Lap^* G_{\mu^*}=\mu^*$ in 
  $\Wscz(\Dom_{n,0}^*)$\\[1pt]
  $\Gc_y$ & Continuum Green's function, solving $-\Delta\Gc_y=\smfrac{1}{\V}\delta_y$
            in $\Dom$, $\Gc_y=0$ on $\partial\Dom$\\[1pt]
  $\Posen$ & Set of `well--separated'
  dislocation positions\\[1pt]
  $\Posc$ & Set of macroscale 
  `well--separated' dislocation positions\\[1pt]
  \hline
  $\beta$ & Inverse thermodynamic temperature\\[1pt]
  $\tscale_n$ & Characteristic timescale of observation\\[1pt]
  $\DD([0,T];M)$ & Skorokhod space of c\`adl\`ag functions from
  $[0,T]$ to a metric space $M$\\[1pt]
  $\Omega_n$, $H_n$ & Infinitesimal and nonlinear generators of the KMC process\\[1pt]
  $\Ham^\L_{A,B},\Lag^\L_{A,B}$ & The Hamiltonian and Lagrangian for KMC process
  \\[1pt]
  $\E$ & Renormalised energy\\[1pt]
  $\Psi^\L_{A,B}$ & Dissipation potential\\[1pt]
  $\Act^\L_{A,B}$ & Large Deviations rate functional\\[1pt]
  \hline
\end{tabular}
\label{tab:notation}
\end{table}

\section{Preliminaries}
\label{sec:preliminaries}
\alth{
As stated in the introduction,
the construction of the local minima corresponding to dislocation
configurations we give below relies upon a particular
dual construction which corresponds in some sense to the
construction of a `discrete harmonic conjugate'. This construction is
most conveniently expressed using a discrete theory of differential
forms, which also provides the basis for a definition of the Burgers
vector of a deformation. The reader already familiar with this theory
may wish to refer to Table~\ref{tab:notation} for our choice of
notation and skip to \S\ref{sec:dual_examples}, where the particular
examples necessary for the subsequent analysis are given.
}

\alth{
\subsection{Lattice complex}
\label{sec:lattice_complex}
We begin by recalling some facts about \emph{lattice complexes},
which provide the correct tools to study dislocations in the 
model we consider. Lattice complexes are a particular class of
\emph{CW complex}, which are objects usually studied in algebraic
topology, and were defined with a particular view to applications in
the modelling of dislocations in crystals in \cite{AO05}:
we follow the same basic definitions and terminology here. For further
details on the definitions below, we refer the reader to
Section~2 of \cite{AO05}, and for background on such constructions in a
general setting, see either the Appendix of \cite{Hatcher},
or \cite{Munk84}.

To provide some intuition to those less familiar with the
notions described here, we remark at the outset that a lattice complex
may be thought of as a `skeleton' of sets of
increasing dimension which is built on the lattice points and fills
$\R^d$. The elements of this skeleton are \emph{$p$--cells}, where $p$
refers to the `dimension' of the particular element.
The key idea behind the definition of a lattice complex is that it
provides a means by which to make rigorous sense of
\begin{itemize}
\item the \emph{boundaries} of sets;
\item operators analogous to the \emph{gradient}, \emph{divergence},
  and \emph{curl}, and
\item versions of the \emph{Divergence} and \emph{Stokes'
  theorems} which relate the above notions.
\end{itemize}
Since these are likely to be familiar, we will point out some
analogies with these more familiar calculus
concepts along the way. The reader is invited to refer to
Figure~\ref{fig:lattices} for an illustration of the particular
lattice complexes used in the subsequent analysis.
}

\subsubsection{Construction of a lattice complex}
% \alth{
% We recall that a \emph{$d$--dimensional Bravais lattice},
% $\L\subset\R^d$, is a subset of $\R^d$ of the form
% \begin{equation*}
%   \L := \mathsf{A}\Z^d+c,
% \end{equation*}
% where $\mathsf{A}$ is a fixed linear transformation $\mathsf{A}$ with
% $\det(\mathsf{A})>0$, and $c$ is a fixed vector in $\R^d$.
% More generally, a \emph{$d$--dimensional multilattice}, $\L$, is a set
% of the form
% \begin{equation*}
%   \L := \bigcup_{i=1}^n \b(\mathsf{A}\Z^d+c_i\b),
% \end{equation*}
% where again $\det(\mathsf{A})>0$ and a set of vectors
% $\{c_1,\ldots,c_n\}\subset\R^d$.
% }

Given a Hausdorff topological space $S$, a $0$--cell is simply a
member of some fixed subset of points in $S$. Higher--dimensional cells
are then defined iteratively: for
$p\geq1$, a $p$--dimensional cell (or $p$--cell) is $e\subset S$
for which there exists a homeomorphism mapping the interior of the
$p$--dimensional closed ball in $\R^p$ onto $e$,
and mapping the boundary of the ball onto a finite union of cells
of dimension less than $p$.

A CW complex is a Hausdorff topological space along with a 
collection of cells as defined above, such that $S$ is the 
disjoint union of all cells. The CW complex is
\emph{$d$--dimensional} if the maximum dimension of any cell is $d$,
and $S$ is referred to as the \emph{underlying space} of the complex:
$S_p$ will denote the set of all $p$--cells in the complex.

Each $p$--cell may be assigned an orientation consistent with the 
usual notion for set in $\R^d$, and we write $-e$ to mean the
$p$--cell with opposite orientation to that of $e$.
We may define an operator $\partial$, called the boundary 
operator, which maps oriented $p$--cells to consistently oriented
$(p-1)$--cells, which intuitively are `the boundary' of the
original cell. Similarly, the coboundary operator $\delta$ may be
defined, mapping an oriented $p$--cell, $e$, to all
consistently oriented $p+1$--cells which have $e$ as part of their
boundary.

We now recall from \cite{AO05} that a lattice complex is a CW
complex such that:
\begin{itemize}
  \item the underlying space is all of $\R^d$,
  \item the set of $0$--cells forms an $d$--dimensional lattice,
    and
  \item the cell set is translation and symmetry invariant.
\end{itemize}
Throughout, we will denote such a lattice complex $\L$, and the
set of $p$--cells of the corresponding complex will be $\CpL$.
Due to the translation invariance of $\L$, it will be particularly
convenient to consider translations of lattice $p$--cells,
so for $e\in\CpL$ and a vector $\avec\in\R^d$, we define
\begin{equation*}
  e+\avec:=\b\{x\in\R^d\bsep  x= y+\avec, y\in e\b\}.
\end{equation*}
We will always assume that we have chosen coordinates such that
$\{0\}\in\L_0$ and, abusing notation, we will write $0$ to refer
to this $0$--cell.

A second convenient notational convention we will occasionally use is
the representation of a $1$--cell through its boundary; we write
\begin{equation*}
  e = [e_0,e_1]\quad\text{to mean}\quad e\in\L_1\quad
  \text{such that}\quad\partial e=e_1\cup -e_0.
\end{equation*}

\subsubsection{Spaces of $p$--forms and calculus on lattices}
\label{sec:lattice_calculus}
\alth{
For the application considered here, we wish to describe
deformations of a crystal. These are appropriately described in the
lattice complex framework as \emph{$p$--forms}, which are real--valued
functions on $p$--cells which change sign if the orientation of the
cell on which they are evaluated is reversed.
}
We define $\Wsc(\CpL)$ to be the space of all $p$--forms, that is
\begin{equation*}
  \Wsc(\CpL):=\b\{f:\CpL\to\R\bsep f(e)=-f(-e),\text{ for any }
  e\in\CpL\b\}.
\end{equation*}
It is straightforward to check that this is a vector space under
pointwise addition. We also define the set of
\emph{compactly--supported} $p$--forms,
\begin{equation*}
  \Wsc_c(\CpL):=\B\{f\in\Wsc(\CpL)\Bsep {\textstyle \overline{\bigcup\{e\sep 
  f(e)\neq0}\} }\text{ is compact in }\R^d\B\},
\end{equation*}
where here and throughout, $\overline{A}$ denotes the closure of
$A\subset\R^d$.

Let $A\subset \CpL$ be finite; then for $f\in \Wsc(\CpL)$, we
define the integral
\begin{equation*}
  \int_A f  := \sum_{e\in A} f(e).
\end{equation*}
The \emph{differential} and \emph{codifferential} are respectively
the linear operators $\df:\Wsc(\CpL)\to\Wsc(\L_{p-1})$
and $\codf:\Wsc(\CpL)\to
\Wsc(\L_{p+1})$, defined to be
\begin{equation*}
  \df f(e) := \int_{\partial e} f\qquad\text{and}\qquad
  \codf f(e):= \int_{\delta e} f.
\end{equation*}
\alth{
For a $0$--form on a lattice complex, the differential is simply the
finite difference operator defined for a pair of nearest neighbours, and
in a continuous setting the same operator is the gradient.
Similary, $\codf$ acting on $1$--forms is either (the negative of)
the discrete or continuum divergence operator. In a three--dimensional
complex, both $\df$ acting on $1$--forms and $\codf$ acting on
$2$--forms may be thought of as the curl operator.
}

The bilinear form
\begin{equation*}
  (f, g) := \int_{\CpL} fg
\end{equation*}
is well--defined whenever $f\in\Wsc_c(\L_p)$ or $g\in\Wsc_c(\L_p)$. Moreover, if $f\in\Wsc_c(\L_p)$
and $g\in\Wsc_c(\L_{p+1})$, we have the integration by parts
formula
\begin{equation}
  (\df f, g) = (f, \codf g);\label{eq:IBP}
\end{equation}
\alth{
this statement should be compared with that of the
Divergence Theorem and variants, using the vector calculus
interpretation of $\df$ and $\codf$ given above.
}
Furthermore, by defining the space
$\Lsc^2(\CpL):=\b\{f\in\Wsc(\CpL)\bsep(f,f)<+\infty\b\}$,
this bilinear form defines an inner product. It is
straightforward to show that this is then a Hilbert space with the
induced norm, which we denote $\|u\|_2:=(u,u)^{1/2}$.

We recall the definition of the Hodge Laplacian as the operator
\begin{equation}
  \Lap:\Wsc(\CpL)\to\Wsc(\CpL)\quad\text{with}\quad
  \Lap f := (\codf\df +\df\codf)f\label{eq:Lap_defn}
\end{equation}
when $p\neq0$ and $p\neq m$, and in the cases where $p=0$ and
$p=m$,  $\Lap = \codf\df$ and $\Lap=\df\codf$ respectively.
\alth{
Note that, in a continuum setting, this
definition of the Laplacian agrees with the interpretation of
$\df$ as the gradient on $0$--forms and $\codf$ as the negative of the
divergence on $1$--forms.
}
Any function satisfying $\Lap f = 0$ on $A\subset\CpL$ is said
to be \emph{harmonic on} $A$.

Finally, $\chr_e$ will always denote the $p$--form
\begin{equation*}
  \chr_e(e'):=\cases{\pm1 & e'=\pm e,\\
  0&\text{otherwise.}}
\end{equation*}

\subsection{Dual complex}
\label{sec:dual_construction}
The common notion of duality which occurs in algebraic
topology relating to CW complexes is that of the \emph{cohomology}.
This is usually presented as
an abstract algebraic structure, since it is only this structure
which is needed to deduce topological information
about a CW complex. In some cases it may also be given a more
concrete identification, which will be particularly
important for the subsequent analysis.

Given an $m$--dimensional lattice complex, when possible, we
define the dual complex as follows:
\begin{itemize}
  \item For any $e\in \L_m$, let $e^*:=\int_e x\dx$, the 
   barycentre of set $e$ in $\R^d$, and let 
   \begin{equation*}
     \L^*_0:=\{ e^* \sep e\in \L_m\}.
   \end{equation*}
 \item For a collection of elementary $m$--cells $A\in \L_m$, let
   \begin{equation}
     A^*:= \bigcup_{e\in A} e^*.\label{eq:duality_sums}
   \end{equation}
 \item \alth{
   Now, iterate over $p=m-1,m-2,\ldots,0$: for each $p$,
   let $e\in \L_p$, and consider $\del e\in \L_{p+1}$ as a sum
   of elementary $p$--cells. Find the corresponding cells in 
   $\L^*_{m-p-1}$. Define $e^*\in \L^*_{m-p}$ to be the
   convex hull of $(\del e)^*$ with $(\del e)^*$ removed,
   assigning $e^*$ the same orientation as $e$. For $A$, a sum of
   elementary $p$--cells, we again define $A^*$ via
   \eqref{eq:duality_sums}.
 }
\end{itemize}
We define boundary and coboundary operators on the dual
lattice complex, $\partial^*$ and $\delta^*$, so that
\begin{equation}
  \partial^* e^* = (\delta e)^*,\quad\text{and}\quad
  \delta^* e^* = (\partial e)^*.\label{eq:df+codf_duality}
\end{equation}
By construction, $*:\CpL\to \L^*_{m-p}$
defines an isomorphism of the additive group structure usually
defined on lattice complexes (see \S2.2 of \cite{AO05}).
The equalities stated in \eqref{eq:df+codf_duality} may then be
interpreted as the statement of the Poincar\'e duality theorem
(see for example Section~3.3 of \cite{Hatcher}), and
the construction described above is succinctly represented in 
the following commutation diagram.
\begin{equation*}
\begin{tikzpicture}[baseline=(current bounding box.center)]
  \node (l0) at (-2,0) {};  
  \node (a) at (0,0) {$\L_{p+1}$}; 
  \node (b) at (2.5,0) {$\CpL$};
  \node (c) at (5,0) {$\L_{p-1}$};
  \node (r0) at (7,0) {};
  \node (l1) at (-2,-2) {};
  \node (d) at (0,-2) {$\L^*_{m-p-1}$}; 
  \node (e) at (2.5,-2) {$\L^*_{m-p}$};
  \node (f) at (5,-2) {$\L^*_{m-p+1}$};
  \node (r1) at (7,-2) {};
  \path[dotted]
    (l0) edge (a)
    (c) edge (r0)
    (l1) edge (d)
    (f) edge (r1);
  \path[<->,font=\scriptsize]
    (a) edge node[right] {*} (d)
    (b) edge node[right] {*} (e)
    (c) edge node[right] {*} (f);
  \path[->,font=\scriptsize]
    ([yshift= 2pt]a.east) edge node[above] {$\partial$} ([yshift= 2pt]b.west)
    ([yshift= 2pt]b.east) edge node[above] {$\partial$} ([yshift= 2pt]c.west)
    ([yshift= 2pt]d.east) edge node[above] {$\delta^*$} ([yshift= 2pt]e.west)
    ([yshift= 2pt]e.east) edge node[above] {$\delta^*$} ([yshift= 2pt]f.west)
    ([yshift=-2pt]b.west) edge node[below] {$\delta$} ([yshift= -2pt]a.east)
    ([yshift=-2pt]c.west) edge node[below] {$\delta$} ([yshift= -2pt]b.east)
    ([yshift=-2pt]e.west) edge node[below] {$\partial^*$} ([yshift= -2pt]d.east)
    ([yshift=-2pt]f.west) edge node[below] {$\partial^*$} ([yshift= -2pt]e.east);
\end{tikzpicture}
\end{equation*}

Since the differential and codifferential operators inherit features
from the structure of the CW complex on which $p$--forms are defined,
we now show that similar duality properties hold for the differential
complexes on $\L$ and $\L^*$. For any $f\in\Wsc(\CpL)$, we define
$f^*\in\Wsc(\L^*_{m-p})$ via
\begin{equation*}
  f^*(e^*):= f(e).
\end{equation*}
Again, it may be checked that $*:\Wsc(\CpL)\to\Wsc(\L^*_{m-p})$ is
an isomorphism; in fact, $*$ defines an
isometry of the spaces $\Lsc^2(\CpL)$ and $\Lsc^2(\L^*_{m-p})$. 
The differential, denoted $\dfst:\Wsc(\CpLst)\to\Wsc(\L^*_{p-1}(\L^*))$, and codifferential, denoted $\codfst:\Wsc(\L^*_{p-1})\to\Wsc(\CpLst)$, are then
\begin{equation*}
  \dfst f^*(e^*) := \int_{\partial^* e^*} f^* = \int_{\delta e} f =\codf f(e),\quad\text{and}\quad \codfst f^*(e^*) :=
  \int_{\delta^* e^*} f^* = \int_{\partial e} f =\df f(e).
\end{equation*}
Again, this relationship is concisely expressed in the following
diagram.
\begin{equation*}
\begin{tikzpicture}[baseline=(current bounding box.center)]
  \node (l0) at (-2,0) {};  
  \node (a) at (0,0) {$\Us(\L_{p+1})$}; 
  \node (b) at (3,0) {$\Us(\CpL)$};
  \node (c) at (6,0) {$\Us(\L_{p-1})$};
  \node (r0) at (8,0) {};
  \node (l1) at (-2,-2) {};
  \node (d) at (0,-2) {$\Us(\L^*_{m-p-1})$}; 
  \node (e) at (3,-2) {$\Us(\L^*_{m-p})$};
  \node (f) at (6,-2) {$\Us(\L^*_{m-p+1})$};
  \node (r1) at (8,-2) {};
  \path[dotted]
    (l0) edge (a)
    (c) edge (r0)
    (l1) edge (d)
    (f) edge (r1);
  \path[<->,font=\scriptsize]
    (a) edge node[right] {*} (d)
    (b) edge node[right] {*} (e)
    (c) edge node[right] {*} (f);
  \path[->,font=\scriptsize]
    ([yshift= 2pt]a.east) edge node[above] {$\df$} ([yshift= 2pt]b.west)
    ([yshift= 2pt]b.east) edge node[above] {$\df$} ([yshift= 2pt]c.west)
    ([yshift= 2pt]d.east) edge node[above] {$\codfst$} ([yshift= 2pt]e.west)
    ([yshift= 2pt]e.east) edge node[above] {$\codfst$} ([yshift= 2pt]f.west)
    ([yshift=-2pt]b.west) edge node[below] {$\codf$} ([yshift= -2pt]a.east)
    ([yshift=-2pt]c.west) edge node[below] {$\codf$} ([yshift= -2pt]b.east)
    ([yshift=-2pt]e.west) edge node[below] {$\dfst$} ([yshift= -2pt]d.east)
    ([yshift=-2pt]f.west) edge node[below] {$\dfst$} ([yshift= -2pt]e.east);
\end{tikzpicture}
\end{equation*}

\subsection{Examples: the square, triangular and hexagonal 
lattices}
\label{sec:dual_examples}
In the analysis which follows, we focus exclusively on
2--dimensional lattice complexes, and in particular the 
triangular, square and hexagonal lattices denoted $\Tr$, $\Sq$
and $\Hx$ respectively. Let $\mR_4$ and $\mR_6$ be
the rotation matrices
\begin{equation*}
  \mR_4:=\begin{pmatrix}
    0 & -1 \\
    1 & 0
  \end{pmatrix}\quad\text{and}\quad
  \mR_6:=\begin{pmatrix}
    \smfrac12 & -\smfrac{\sqrt{3}}{2} \\[1mm]
    \smfrac{\sqrt{3}}{2} & \smfrac12
  \end{pmatrix}.
\end{equation*}
For convenience, we define $\evec_1:=\avec_1:=(1,0)^T$,
and
\begin{equation*}
  \evec_i:=\mR_4^{i-1}\evec_1\quad\text{for }i\in\{1,2,3,4\},\quad
  \text{and}\quad\avec_j:=
  \mR_6^{j-1}\avec_1\quad\text{for }j\in\{1,\ldots,6\}.
\end{equation*}
\alth{
The triangular, square and hexagonal lattices are defined to be
\begin{align*}
  \Tr:=[\avec_1,\avec_2]\cdot \Z^2,\qquad\Sq:=\Z^2,
  \quad\text{and}\quad\Hx:= \sqrt{3}\,\mR_4\Tr \cup
  \b[\sqrt{3}\,\mR_4\Tr+\evec_1\b];
\end{align*}
the nearest neighbour directions in $\Sq$ are therefore $\evec_i$, and
$\avec_i$ in $\Tr$ or $\Hx$.
We may define lattice complexes based on these sets (see \S2.3.2 
and \S2.3.3 of \cite{AO05} and \cite{AO10}), and moreover
\begin{equation*}
  \Tr^* = \smfrac{\sqrt{3}}3\mR_4\Hx+\smfrac13(\avec_2+\avec_3),
  \quad\Sq^* = \Sq+\smfrac12(\evec_1+\evec_2),\quad\text{and}
  \quad\Hx^* = \sqrt{3}\mR_4\Tr+\smfrac{\sqrt{3}}3(\avec_1+\avec_2).
\end{equation*}
}
Figure~\ref{fig:lattices} illustrates the three lattices and the
duality mapping between $\L$ and $\L^*$.

At this point, we give the definitions of some lattice--dependent
constants which will arise during our analysis:
\begin{equation}
  \Num := \cases{ 3 &\text{if }\L=\Hx,\\
  4 &\text{if }\L=\Sq,\\
  6 &\text{if }\L=\Tr.}
  \quad\text{and}\qquad
  \V := \cases{ 2 &\text{if }\L=\Hx,\\
  4 &\text{if }\L=\Sq,\\
  6 &\text{if }\L=\Tr,}
  \label{eq:constants_defn}
\end{equation}
For convenience, we will write $\V^*$ and $\Num^*$ to mean the relevant
constants for the dual lattice.
Note that $\Num$ is the number of nearest 
neighbours in the lattice.
% and $\V$ is the constant
% such that, if $\avec_i$ are the nearest neighbour directions in the
% lattice and $\mathsf{I}$ is the identity matrix,
% \begin{equation*}
%   \sum_{i=1}^N\avec_i\otimes\avec_i = \V\,\mathsf{I}.
% \end{equation*}

\begin{figure}
\begin{center}
\begin{minipage}{0.26\textwidth}
\includegraphics[width=\textwidth]{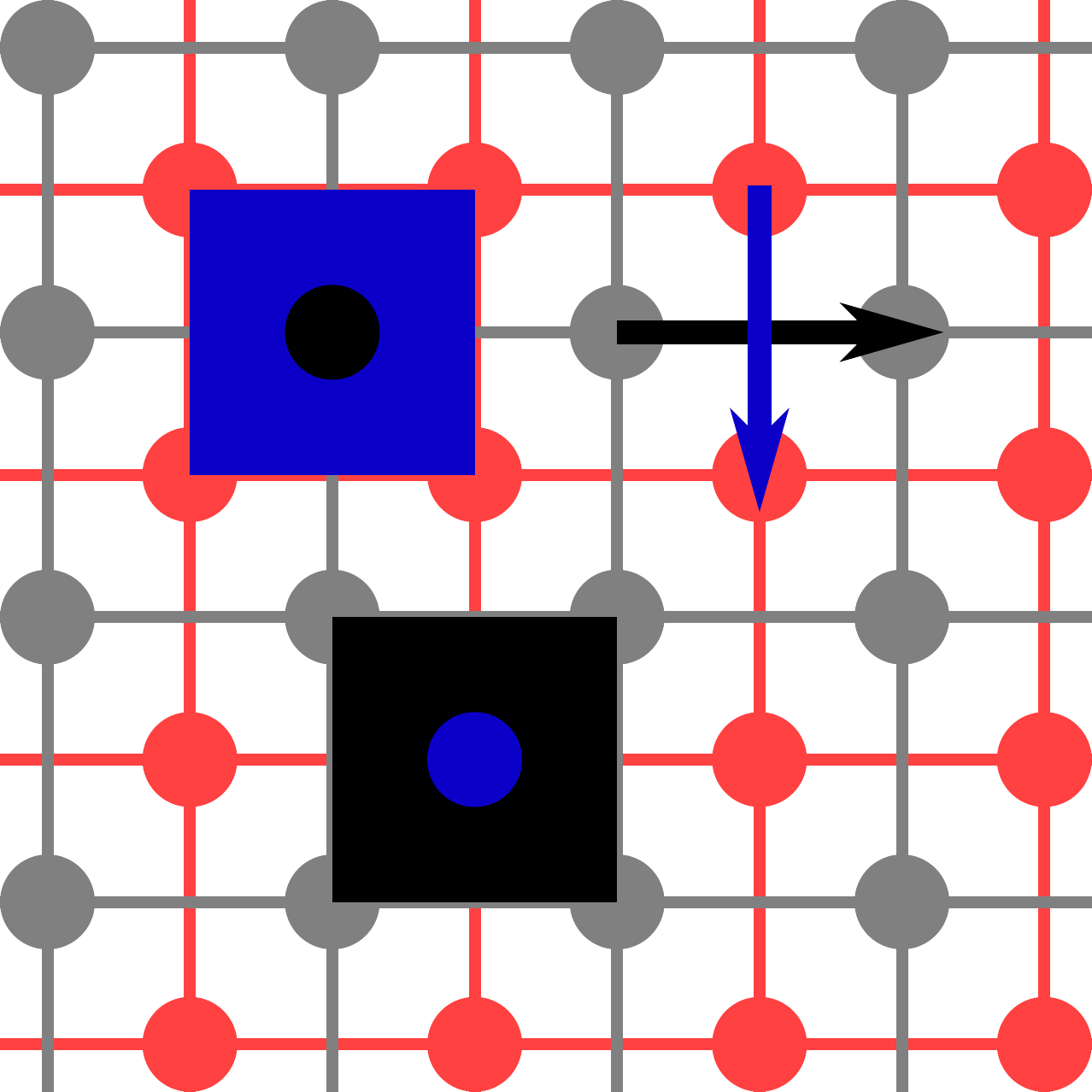}
\end{minipage}
\hfill
\begin{minipage}{0.33\textwidth}
\includegraphics[width=\textwidth]{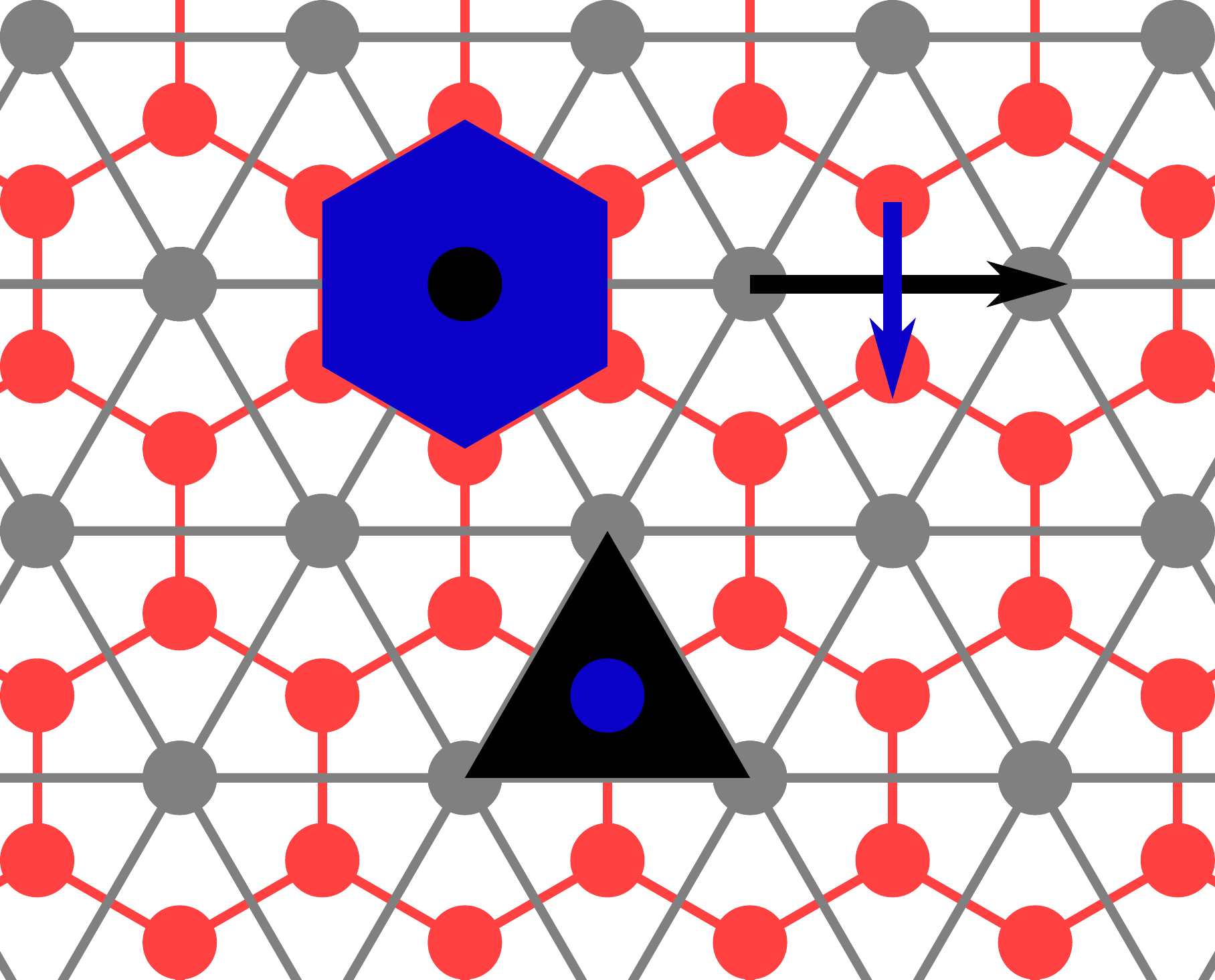}
\end{minipage}
\hfill
\begin{minipage}{0.33\textwidth}
\includegraphics[width=\textwidth]{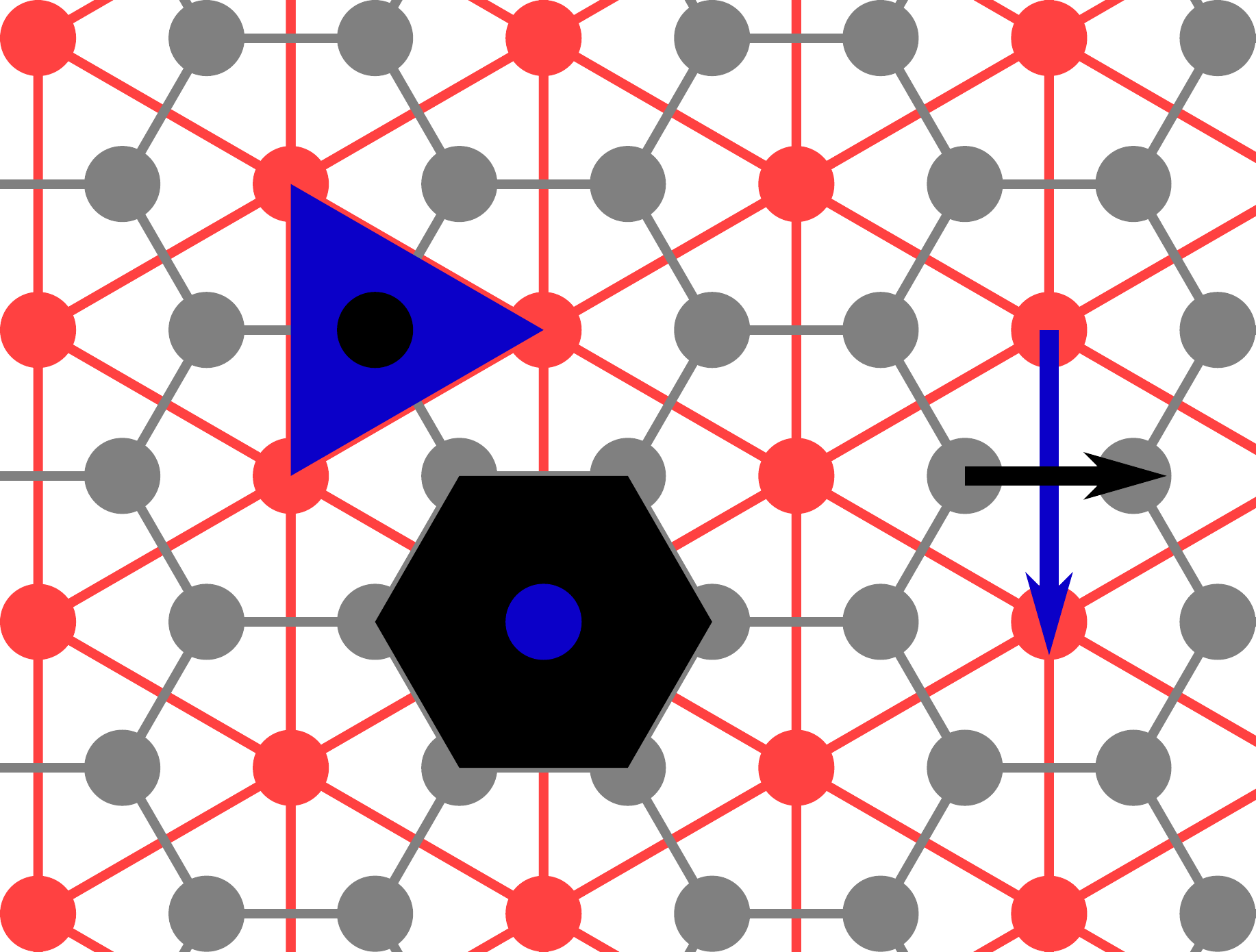}
\end{minipage}
\end{center}

\caption{The square, triangular and hexagonal lattices 
respectively, and their duals. Primal lattices are shown
in grey, dual lattices in red; 2--cells are left uncoloured.
Particular primal cells are highlighted in black, and their
respective dual cells are given in blue.}
\label{fig:lattices}
\end{figure}

\subsection{Finite lattice subcomplexes}
\label{sec:finite_doms}
For the particular application we will consider, we will make use
of finite subcomplexes of the full lattice complex, and so we now
make precise the notation we use as well as the particular
assumptions made throughout our analysis.
\alth{
The reader may
find it useful to refer to Figure~\ref{fig:subcomplexes}, which
illustrates the construction in a couple of simple cases.
}

\subsubsection{Induced subcomplexes}
\label{sec:induced_subcomplexes}
Given a finite subset $\A_0\subset\L_0$, we define 
the \emph{induced lattice subcomplex} by inductively defining
\begin{equation*}
  \A_p:=\b\{e\in\L_p\bsep \partial e\subset 
  \A_{p-1}\b\}.
\end{equation*}
This is a well--defined CW complex when the corresponding 
boundary $\partial^\A$ and coboundary 
$\del^\A$ operators are defined by restriction, i.e.
\begin{equation*}
  \partial^\A e:= \partial e\cap\A_{p-1},
  \quad\text{and}\quad\del^\A e =(\del e)\cap
  \A_{p+1}\quad\text{for all }e\in\A_p.
\end{equation*}
The induced differential and codifferential operators $\df^\A$
and $\codf^\A$ are then defined in the same way as 
$\df$ and  $\codf$, using $\partial^\A$ and 
$\del^\A$ in place of  $\partial$ and $\del$, and we 
may define the spaces $\Wsc(\A_p)$ and 
$\Lsc^2(\A_p)$.

It will be convenient to distinguish what we term the
\emph{exterior} and \emph{interior} $p$--cells of the CW complex $\A$, respectively defined to be
\begin{equation*}
  \Out(\A_p):=\{e\in\A_p \sep \del e\neq 
  \del^\A e\},\quad
  \text{and}\quad \Int(\A_p):=\A_p\setminus 
  \Out(\A_p).
\end{equation*}
The former set may be thought of as the `edge' of the lattice
subcomplex, and the latter as the `interior' of the lattice 
subcomplex.

We now define a subcomplex of the dual lattice complex which
we call the \emph{dual subcomplex induced by} $\A_0$. Let
$\A^*_m:=\{ e^*\in \L^*_m \sep e\in \A_0\}$, 
and inductively define
\begin{equation*}
  \A^*_{m-p}:=\b\{ e^*\in \L^*_{m-p} \bsep e^*\in
  \partial^*a^*\text{ for some }a^*\in \A^*_{m-p+1}\}
\end{equation*}
for $p\geq 1$.
We remark that this definition is \emph{not} equivalent to 
defining sets of sets of dual $p$--cells by directly taking the 
dual of the primal $p$--cells; however, we do have the inclusion
\begin{equation*}
  \b[\A_p\b]^*\subseteq \A^*_{m-
  p}\quad\text{for each } p,
\end{equation*}
where equality always holds when $p=m$ by definition. The other
inclusions follow by induction on $p$: note that $e\in\A_p$
with $p\geq1$ implies that $e\in\del^\A a$ for some
$a\in \A_{p-1}$, $\del^\A a\subseteq\del a$, 
and hence $e^*\in \partial^*a^*$ for some $a^*\in 
\A^*_{m-p+1}$.
As before, we may define $\partial^{\A^*}$ and $\delta^{\A^*}$
by restriction, which in turn leads us to define operators 
$\df^{\A^*}$ and $\codf^{\A^*}$ analogously.

Similarly, let
\begin{equation*}
  \Out(\A_p^*):=\b\{e^*\in\A_p^* \bsep \del^* 
  e^*\neq \del^{\Dom^*} e^*\b\},\quad
  \text{and}\quad 
  \Int(\A_p^*):=\A_p^*\setminus 
  \Out(\A_p^*).
\end{equation*}
By construction, $\Out(\A^*_2)=\emptyset$, and 
$e\in\Out(\A^*_{n-p})$ if and only if there exists no 
$a\in\A_p$ with $e=a^*$ (see
Figure~\ref{fig:subcomplexes} for an illustration).

From now on, it will always be clear from the context whether 
we are referring to the relevant operators on $\L$ and $\L^*$,
or on $\A$ and $\A^*$, so for the sake of
concision, we will suppress $\A$ from our notation.

\begin{figure}
  \includegraphics[width=0.47\textwidth]
  {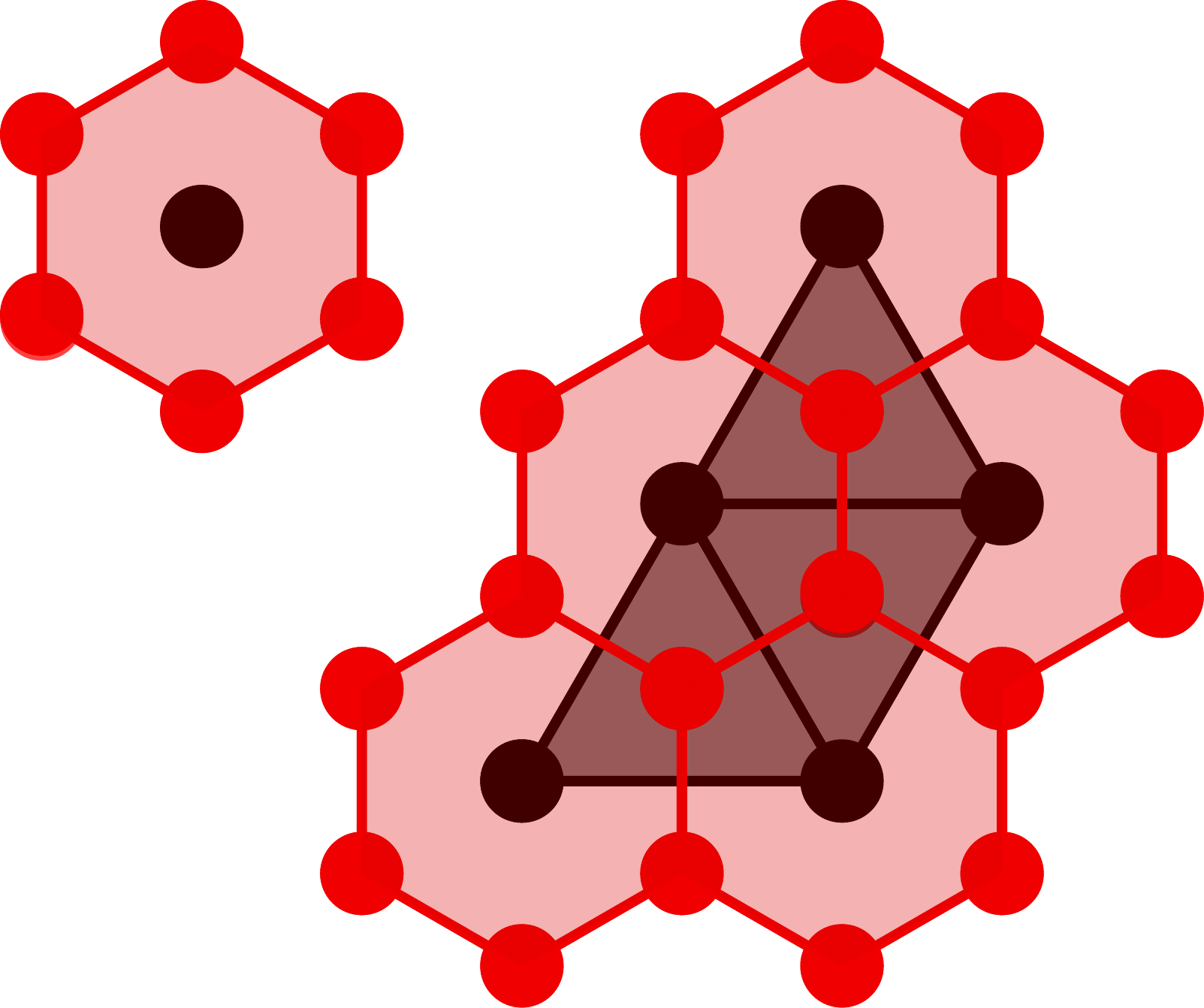}\hfill
  \includegraphics[width=0.47\textwidth]
  {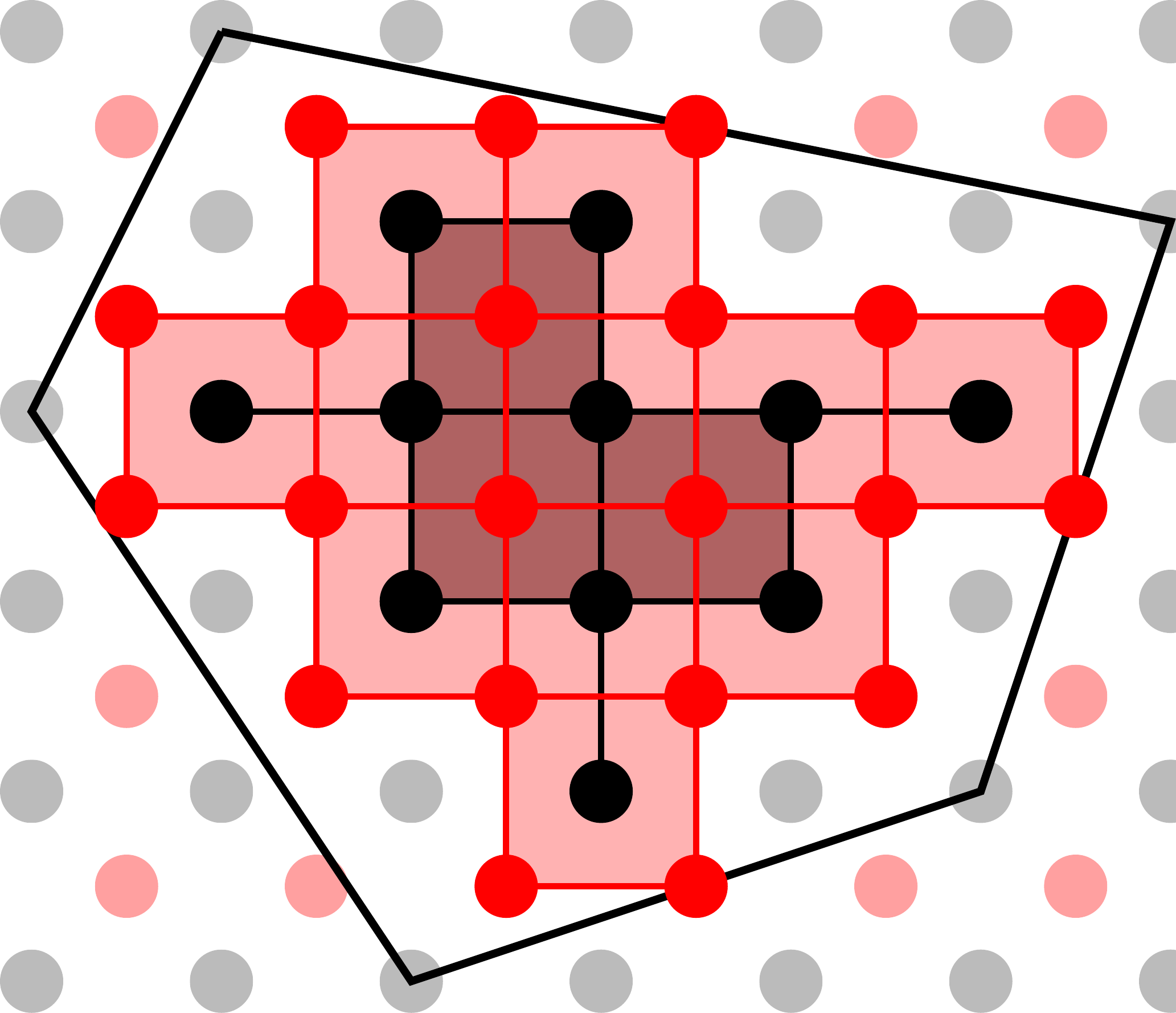}
  \caption{On the left, an example of primal (in black) and
  dual (in red) induced subcomplexes for a general 
  subset of the triangular lattice: $\A_0$ is the
  set of black points.
  On the right, a lattice polygon $\Dom$ in the square lattice, and
  the corresponding primal and dual subcomplexes, which are
  both path-- and simply--connected.}
  \label{fig:subcomplexes}

\end{figure}

\subsubsection{Subcomplexes induced by a domain}
We will say that an induced lattice subcomplex is
path--connected if for any $e,e'\in \A_0$, there 
exists  $\gamma\subset\A_1$ such that
\begin{equation*}
  \partial\gamma = e \cup - e',
\end{equation*}
and call such $\gamma\subset\A_1$ a \emph{path} which
connects $e$ and $e'$.
We will say a lattice subcomplex is \emph{simply--connected}
if for any $\gamma'\subset \A_1$ such that
$\partial\gamma = \emptyset$, $\gamma = \partial A$ for some 
$A\in\A_2$.

Throughout our analysis, $\Dom$ will always denote a closed
convex lattice polygon, i.e. a non--empty compact convex
subset of $\R^2$ which has corners $c_l\in\L$ and internal
angles $\varphi_l$ where $l=1,\ldots,L$ indexes the corners,
following \cite{Grisvard}. We consider the scaled domains $n\Dom$, 
for $n\in\N$, noting that $n\Dom$ remains a lattice polygon,
and denote $\Dom_n$ to be the largest induced lattice
subcomplex with respect to inclusion such that
\begin{itemize}
  \item $\Dom_{n,p}\subset n\Dom$ for all $p$,
  \item $\Dom_{n,p}^*\subset n\Dom$  for all $p$,
  \item $\Dom_n$ and $\Dom^*_n$ are both path connected and
    simply connected.
\end{itemize} 
It can be shown that such a complex always exists as long as
$n$ is sufficiently large, since $\Dom$ is convex: we give an
example on the right--hand side of Figure~\ref{fig:subcomplexes}.

\subsubsection{Counting and distances}
For a collection of $p$--cells $A\subset \Dom_{n,p}$, we write $\#A$
to designate the smallest number of elementary $p$--cells $e_i$
such that $A = \bigcup_{i=1}^{\#A} e_i$.

We define $\diam(\Dom)$ to be
\begin{equation*}
  \diam(\Dom):=\max\b\{ |x-y| \bsep x,y\in\Dom\b\},
\end{equation*}
and we note that there exists a constant $C^\L>0$ which
depends only on the underlying lattice $\L$ such that
\begin{equation*}
  \max\B\{\min\b\{\#\gamma \bsep \gamma\subset 
  \Dom_{n,1},\partial\gamma = e-e'\b\}\Bsep e',e\in\Out(\Dom_{n,0})
  \B\}\leq C^\L n\,\diam(\Dom).
\end{equation*}
We write $\dist(A,B)$ to mean the shortest
distance between two sets $A,B\subset\R^d$, i.e.
\begin{equation*}
  \dist(A,B):=\inf\b\{ |x-y|\bsep x\in A, y\in B\b\}.
\end{equation*}

\subsubsection{Spaces of $p$--forms on lattice subcomplexes}
\label{sec:p_forms}
The space of $p$--forms on the lattice subcomplex induced by 
$n\Dom$ is denoted
\begin{equation*}
  \Wsc(\Dom_{n,p}):=\b\{ u:\Dom_{n,p}\to\R\bsep u(e) =-
  u(-e)\b\}.
\end{equation*}
As for the space of forms defined on $\L$, we define the inner product
and induced norm
\begin{equation*}
  (u,v):=\int_{\Dom_{n,p}} u\,v,\quad\text{and}\quad\|u\|_2:=(u,u)^{1/2}.
\end{equation*}
Since $\Dom_{n,p}$ is finite, these are always well--defined; we will
also make occasional use of the norm 
\begin{equation*}
  \|u\|_\infty:= \max_{e\in \Dom_{n,p}} |u(e)|.
\end{equation*}
We denote the subspace of $p$--forms vanishing on
$\Out(\Dom_{n,p})$
\begin{equation*}
  \Wsc_0(\Dom_{n,p}):=\b\{u\in\Wsc(\Dom_{n,p})\bsep
  u = 0\text{ on }\Out(\Dom_{n,p})\b\},
\end{equation*}
which is clearly a vector space, and the bilinear form
\begin{equation*}
  (\!(u,v)\!):= \int_{\Dom_{n,1}}\hspace{-2mm}\df u\,\df v
\end{equation*}
is a well--defined inner product on $\Wsc_0(\Dom_{n,0})$.
$\Wsc_0(\Dom_{n,0})$ is thus a Hilbert space with the
corresponding norm, denoted $\|u\|_{1,2}:=(\!(u,u)\!)^{1/2}$.
We now demonstrate positive--definiteness of the inner product,
since we will use the resulting version of Poincar\'e 
inequality below.

Since $\Dom_n$ is path--connected, for any
$e\in \Int(\Dom_{n,0})$, there exists
$\gamma\subset \Dom_{n,1}$ such that $\partial \gamma = e\cup-e'$, 
with $e'\in\Out(\Dom_{n,0})$ and $\#\gamma\leq C^\L_0 n\,\diam(\Dom)$.
For any $u\in\Wsc_0(\Dom_{n,0})$, we then have
$u(e)= \int_\gamma \df u$,
so applying the Cauchy--Schwarz inequality, we have
\begin{equation*}
  |u(e)|^2=\bg|\int_\gamma \df u\bg|^2\leq \#\gamma\int_{\gamma}|\df u|^2
  \leq \#\gamma\int_{\Dom_{n,1}} |\df u|^2.
\end{equation*}
Integrating over $\Dom_{n,0}$, and noting that there exists a
constant $C^\L_1>0$ which depends only on the underlying lattice
$\L$ such that $\#\Dom_{n,0}\leq C^\L_1 n^2\diam(\Dom)^2$, we have
\begin{equation}
  \int_{\Dom_{n,0}}|u|^2 \leq C^\L_2n^3\diam(\Dom)^3 \int_{\Dom_{n,1}}
  |\df u|^2,\label{eq:Poincare}
\end{equation}
where $C^\L_2 = C^\L_0 C^\L_1$. We note that the same inequality also
holds for $u\in\Wsc_0(\Dom_{n,0}^*)$ by a similar argument.

\subsubsection{Duality for $p$--forms on lattice subcomplexes}
We define the duality mapping
$*:\Wsc(n\CpD)\to\Wsc_0(\Dom^*_{n,2-p})$ as follows:
\begin{equation*}
  u^*(a) = \cases{
    u(e) & a=e^*\in\Int(\Dom^*_{n,2-p}),\\
    0    & a\in\Out(\Dom^*_{n,2-p}).
  }
\end{equation*}
We note that this mapping is well--defined since as noted in
\S\ref{sec:induced_subcomplexes}, $a\in\Out(\Dom^*_{n,2-p})$ if
and only if there exists no $e\in \Dom_p$ with $a=e^*$.
This duality mapping defines an isomorphism from $\Wsc(\CpD)$ to
$\Wsc_0(\Dom^*_{n-p})$ as vector spaces; as, in addition
\begin{equation*}
  (u,u) =\int_{n\CpD} |u|^2 = \int_{\Int(n\CpDst)}|u^*|^2
  = \int_{n\CpDst}|u^*|^2 = (u^*,u^*),
\end{equation*}
it follows that $*$ defines an isometry of the spaces
$\Lsc^2(n\CpD)$ to $\Lsc^2(\Dom^*_{n,2-p})$.
Moreover, for any $e\in n\CpD$, we verify that
\begin{equation}
\label{eq:duality_Dom}
\begin{aligned}
  \df u(e)&= \int_{\partial e}u=
  \int_{(\partial e)^*}
  \hspace{-3mm}u^*=\int_{\del^* e^*} \hspace{-3mm}u^* = 
  \codf^* u^*(e^*),\\
  \text{and}\qquad
  \codf u(e)&= \int_{\del e} u = \int_{(\del e)^*}\hspace{-3mm}
   u^*= \int_{\partial^* e^*}\hspace{-3mm}u^* = \df^* u^*(e^*).
\end{aligned}
\end{equation}

\subsection{Dislocation configurations}
\label{sec:disl_configs}
We now recall some definitions from \cite{HO15} which will permit
us to give a kinematic description of screw dislocations in the
setting of our model. Given 
$u\in\Wsc(\Dom_{n,0})$, we define the associated set of
\emph{bond--length 1--forms}
\begin{equation*}
  [\df u]:=\b\{\alpha\in\Wsc(\Dom_{n,1})\bsep \|\alpha\|_\infty
  \leq \smfrac12, \alpha-\df u\in\Z\b\}.
\end{equation*}
A \emph{dislocation core} is any positively--oriented 2--cell
$e\in \Dom_2$ such that
\begin{equation*}
  \df \alpha(e) =\int_{\partial e}\alpha\neq 0.
\end{equation*}
Let $\mu\in\Wsc(\Dom_2)$, with $\mu:\Dom_2\to\{-1,0,+1\}$.
We will say that $u$ is a deformation \emph{containing the dislocation configuration} $\mu$ if
\begin{equation*}
  \exists\,\alpha\in[\df u]\quad\text{such that}
  \quad\df\alpha = \mu.
\end{equation*}
The 2--form $\mu$ represents the \emph{Burgers vectors} of the
dislocations in the configuration, which are the topological `charge' of
dislocations; see \cite{HirthLothe,HullBacon11} for general discussion
of the notion of the Burgers vector and its importance in the study of
dislocations, and \cite{AO05,HO15} for further discussion of the
physical interpretation of this specific definition.

For the purposes of our analysis, we define sets of admissible
dislocation configurations. For $\eps>0$, $n\in\N$, and 
$b_i\in\{\pm1\}$ for $i=1,\ldots,m$, we define the set
$\Posen(b_1,\ldots,b_m)$ of $2$--forms
\begin{multline*}
  \Posen(b_1,\ldots,b_m):=\B\{\mu=\sum_{i=1}^m b_i 
  \chr_{e_i}\Bsep e_i\in \Dom_2\text{ positively oriented},
  \dist(e_i,\Out(\Dom_{n,0}))\geq n\eps,\\\dist(e_i,e_j)\geq 
  \epsilon n,\text{for all }i,j\in\{1,\ldots,m\},i\neq j\B\}.
\end{multline*}
Each 2--form in this set represents a collection of $m$ 
dislocations with respective Burgers vectors $b_1,\ldots,b_m$
and cores $e_1,\ldots,e_m$: these dislocations are separated from
each other and from the boundary by a distance of at least
$\eps n$.
\alth{
Since we will assume that the number of dislocations
$m$, and the Burgers vectors $b_1,\ldots,b_m$ are fixed throughout, we
will suppress the dependence on $(b_1,\ldots,b_m)$ from now on.
}

\section{Main Results}
\label{sec:barrier_results}
\subsection{Energy and equilibria}
\label{sec:equivalence}
As stated in the introduction, we follow
\cite{AO05,Ponsiglione07,ADLGP14,HO14,ADLGP16,HO15} and consider a
nearest--neighbour anti--plane lattice model for the cylinder of
crystal. Let $\psi:\R\to\R$ be given by
$\psi(x):=\smfrac12\lambda\dist(x,\Z)^2$; we consider the energy
difference functional
\begin{equation*}
  E_n(y;\tilde{y}):=\int_{\Dom_{n,1}} \b[\psi(\df y)-\psi(\df 
  \tilde{y})\b].
\end{equation*}
This functional is a model for potential energy per unit length
of a long cylindrical crystal, and points $\Dom_{n,0}$ correspond
to columns of atoms which are assumed to be periodic in the
direction perpendicular to the plane considered. For further
motivation of this model, we refer the reader to \S1 of
\cite{ADLGP16}.

Following Definition~1 of \cite{HO15}, we will say that
$y\in\Wsc(\Dom_{n,0})$ is a \emph{locally stable equilibrium} if
there exists $\epsilon>0$ such that
\begin{equation*}
  E_n(y+u;y)\geq 0\quad\text{whenever }\|u\|_{1,2}\leq \eps.
\end{equation*}
Due to the periodicity of $\psi$, we note that any locally stable
equilibrium generates an entire family of equilibria: letting 
$z\in\Wsc(\Dom_{n,0})$ taking values in $H+\Z$ for some $H\in\R$, if 
$y$ is a locally stable equilibrium, then so is $y+z$. These
equilibria are physically indistinguishable, since they correspond 
to a vertical `shifts' of columns by an integer number of lattice 
spacings, and a rigid vertical translation of the entire crystal
by $H$. We therefore define the equivalence relation
\begin{equation}
  u\sim v\quad\text{if and only if}\quad u=v+z,
  \quad\text{where}\quad z:\Dom_{n,0}\to \Z+H\quad
  \text{for some }H\in\R,
  \label{eq:equivalence_defn}
\end{equation}
and denote the equivalence classes of this relation as
$\llb y\rrb$.

We recall that Theorem~3.3 in \cite{HO15} gives sufficient
conditions such that locally stable equilibra containing
dislocations exist in the case of a more general choice of 
$\psi$ than that chosen here. Our first main result is similar, 
but in addition provides a very precise representation of the
corresponding bond--length 1--form in the case considered here,
and asserts the uniqueness (up to lattice symmetries) of local
equilibria containing a given dislocation configuration.

\begin{theorem}
\label{th:equivalence}
Fix $\eps>0$ and $\Dom$ a convex lattice polygon; then for all
$n$ sufficiently large, the following statements hold:
\begin{enumerate}
\item For every $2$--form $\mu\in\Posen$, there exists a
  corresponding locally stable equilibrium $u_\mu$ which contains
  the dislocation configuration $\mu$;
\item Each such equilibrium $u_\mu$ is unique up to the equivalence
  relation defined in \eqref{eq:equivalence_defn}; and
\item For any $u\in\llb u_\mu\rrb$, there is a unique bond--length
  $1$--form $\alpha\in[\df u]$ satisfying
  $\alpha^* =\df^*G_{\mu^*}$, where $\mu^*$ is the $0$--form dual to
  $\mu$, and $G_{\mu^*}\in\Wsc_0(\Dom_{n,0})$ is the solution to
  \begin{equation}
    \Lap^* G_{\mu^*} = \mu^*\text{ in }\Int(\Dom^*_{n,0}),
    \quad\text{with}\quad G_{\mu^*} = 0\text{ on }
    \Out(\Dom^*_{n,0}).\label{eq:latt_grn_func_Dom}
  \end{equation}
\end{enumerate}
\end{theorem}

\subsubsection*{Strategy of proof}
The proof of this theorem is the main focus of
\S\ref{sec:elliptic_estimates}. We begin by showing that if $u$ is
a locally stable equilibrium containing dislocations $\mu$,
then $\alpha\in[\df u]$ must necessarily satisfy
\begin{equation}
  \|\alpha\|_\infty<\smfrac12,\qquad\df \alpha =\mu\text{ on }
  \Dom_2,
  \qquad\text{and}\qquad \codf\alpha=0\text{ on }\Dom_{n,0}.
  \label{eq:equilibrium_conds}
\end{equation}
We show that these conditions are satisfied by at most one 
$\alpha\in\Wsc(\Dom_{n,1})$, and using the duality transformation
described in \S\ref{sec:dual_construction}, we verify that
$\alpha\in\Wsc(\Dom_{n,1})$ satisfying $\alpha^*=\df^*G_{\mu^*}$
verifies the latter two conditions. Showing that
$\|\alpha\|_\infty=\|\df^*G_{\mu^*}\|_\infty<\smfrac12$ is the
most technical aspect of the proof, and requires us to develop a
theory which is analogous to obtaining interior estimates for
solutions of a
boundary value problem for Poisson's equation in the continuum
setting. To conclude, we obtain the class $\llb u_\mu\rrb$ by 
`integrating' $\alpha$.
\medskip

\subsection{Energy barriers}
Let $\CC\b([0,1];\Wsc(\Dom_{n,0})\b)$ denote the space of 
continuous paths from $[0,1]$ to $\Wsc(\Dom_{n,0})$.
For $\mu$ and $\nu\in\Posen$, we define the
set of continuous paths which move any local equilibrium in 
$\llb u_\mu\rrb$ to any other local equilibrium in
$\llb u_\nu\rrb$ to be
\begin{multline*}
  \Gamma_n(\mu\to\nu):=\b\{ \gamma\in\CC\b([0,1];\Wsc(\Dom_{n,0})\b)\bsep
  \gamma(0)\in\llb u_\mu\rrb,\gamma(1)\in\llb u_\nu\rrb,\\
  \forall t\in[0,1], \alpha\in[\df\gamma(t)] \text{ implies }\df\alpha =\mu\text{ or }\df\alpha=\nu\b\}.
\end{multline*}
\alth{
In the case where we will apply this definition, i.e. where
$\nu-\mu=b_i[\mathbbm{1}_q-\mathbbm{1}_p]$ with $q^*=p^*+\avec^*$ for
some nearest--neighbour direction $\avec^*$ in the dual lattice, 
corresponding to a single dislocation `hopping' to an adjacent site, the
final condition on
the paths in the above definition ensures that the Burgers vectors of
the configurations
along the path vary only on the 2--cells $p$ and $q$. In other
words, we make the modelling assumption that dislocations move strictly 
from one site to an adjacent site, and not via a more complicated
route.
}

We define the \emph{energy barrier for the transition from
$\mu$ to $\nu$} for $\mu,\nu\in\Posen$ to be
\begin{equation}
  \Tran_n(\mu\to\nu):=\min_{\gamma\in\Gamma_n(\mu\to\nu)}
  \max_{t\in[0,1]}E_n(\gamma(t);u_\mu).
  \label{eq:energy_barrier_defn}
\end{equation}
Our second main result concerns an asymptotic representation of this
quantity.

\begin{theorem}
\label{th:energy_barriers}
Suppose that $\mu,\nu\in\Posen$ are $2$--forms such
that $\nu-\mu = b_i[\chr_q-\chr_p]$ for some $i$, where 
$q^*=p^*+\avec^*$ for some nearest neighbour direction $\avec^*$
in $\L^*$.
For $i=1,\ldots,m$, let $x_i\in\Dom$ be
such that $\dist(x_i,\smfrac1ne_i^*)\leq\smfrac1n$. 
Then there exist a constant $c_0$ which depends only 
on the underlying lattice complex $\L$ such that
\begin{equation*}
  \Tran_n(\mu\to\nu) = \lambda c_0+\smfrac12\lambda n^{-1}
  \B[b_i^2\nabla \bar{y}_j(x_j)\cdot \avec^*+\sum_{i\sep i\neq j}
  b_jb_i\nabla \Gc_{x_i}(x_j)\cdot \avec^*\B]+o(n^{-1}),
\end{equation*}
where
\begin{enumerate}
  \item $\lambda$ is given in the definition of $\psi$,
  \item $\bar{y}_j$ solves the boundary value problem
  \begin{equation*}
    \Delta \bar{y}_j = 0\text{ in }\Dom,\qquad
    \bar{y}_j(\cdot) =  \smfrac{1}{\V\pi}\log(|
    \cdot-x_j|)\text{ on }
    \partial\Dom,
  \end{equation*}
  \item $\Gc_{y}$ is the solution to
    \begin{equation*}
      \Delta\Gc_y = \smfrac{\V}{2}\delta_y\text{ in }\Dom,\quad\text{with}\quad
      \Gc_y = 0\text{ on }\partial\Dom,
    \end{equation*}
    where we recall the definition of $\V$ from
    \eqref{eq:constants_defn}, and 
  \item $o(n^{-1})$ satisfies $no(n^{-1})\to0$ 
    as $n\to\infty$, uniformly for all $\mu\in\Posen$.
\end{enumerate}
\end{theorem}

\subsubsection*{Strategy of proof}
The proof of this result is the main focus of 
\S\ref{sec:barrier_proofs}.
Our main task is the explicit construction of a \emph{transition
state}, i.e. $u_\downarrow\in\Wsc(\Dom_{n,0})$ such that
\begin{equation*}
  E_n(u_\downarrow;u_\mu)= \min_{\gamma\in\Gamma_n(\mu\to\nu)}
  \max_{t\in[0,1]}E_n\b(\gamma(t);u_\mu\b).
\end{equation*}
This may be seen as a generalisation of the notion of a critical
point, but is not a true critical point, since $E_n$
is not differentiable at $u_\downarrow$. Nevertheless, we show that
$\alpha\in[\df u_\downarrow]$ has a dual which is closely related to the
interpolation of $\df^*G_{\mu^*}$ and $\df^*G_{\nu^*}$ which are solutions
of \eqref{eq:latt_grn_func_Dom}.
This dual representation, combined with the precise
asymptotics obtained for $\df^* G_{\mu^*}$ in order to prove 
Theorem~\ref{th:equivalence}, allow us to derive the 
expression of $\Tran_n(\mu\to\nu)$.

\subsection{Remarks on the model}
\alth{
Here, we collect a few remarks concerning the choice of model, the
notion of duality we use, and some
further links between the results above and the way in which
dislocations are modelled in continuum elastoplasticity.

\subsubsection*{More general potentials}
The derivation of the energy we consider as given in \S2.2 of
\cite{ADLGP14} suggests that potential $\psi$ should be chosen to be
smooth, in keeping with the usual assumptions on interatomic
potentials. On the other hand, our results rely heavily on the
definition of $\psi$, since the structure of the potential chosen
permits us both to prove the characterisation and uniqueness of $\alpha$
given in Theorem~\ref{th:equivalence}, and to be precise
about the set on which $\Tran(\mu\to\nu)$ is attained. This ultimately
provides us with a means by which to prove
Theorem~\ref{th:energy_barriers}.

In spite of this, a result similar to
Theorem~\ref{th:energy_barriers} may hold in cases where $\psi$ is
more general, but is sufficiently `close' to the choice made here (see
for example the structural assumptions made in \S5 of \cite{ADLGP14}).
Since the interatomic distances rapidly approach those predicted by
linear elasticity as one moves away from a dislocation core (see
Theorem~3.5 in \cite{EOS15}), and much of the potential 
energy is carried by the elastic field at significant distances from the
dislocation core where a harmonic approximation of the energy is
valid, heuristically one might expect that the energy barrier should be 
similar to that given in Theorem~\ref{th:energy_barriers}. However, due
to the complexity of possible transitions in a more general case, such
a result does not seem tractable without very strong assumptions on the
potential, and significant additional technicalities: we therefore do
not pursue such results here.

\subsubsection*{Dynamics in the infinite lattice}
We remark that a significant amount of our analysis is devoted to
verifying the first condition in \eqref{eq:equilibrium_conds} holds.
This aspect of the proof of Theorem~\ref{th:equivalence} would be
significantly simplified if we were to consider the problem in 
an infinite domain, since in this case integral representations of the
lattice Green's function are available via Fourier--analysis.
Nevertheless, we pursue the evolution on a finite domain here, both
because this is a case of physical relevance, and because
we are able to demonstrate that the boundary affects
the evolution of the configuration in exactly the manner described in
\S2.1 of \cite{vdGN95}.

\subsubsection*{Equilibrium conditions and geometry}
Finally, we remark that the two latter conditions in
\eqref{eq:equilibrium_conds} are analogous to the requirement that
a continuum strain field $\varepsilon$ satisfies
\begin{equation*}
  \curl(\varepsilon) = \mu\quad\text{and}\quad\div(\C:\varepsilon)=0.
\end{equation*}
These are the conditions usually prescribed
on a strain field $\varepsilon$ which contain dislocations described by 
a measure $\mu$ in a linear elastic setting (see for example (1.1) in
\cite{CG05}).

We also note that the precise notion of duality which we use is specific
to two--dimensional modelling of dislocations, as it is only in this
case that $\L_1$ and $\L_1^*$ are related by duality. The fact that dual
1--cells are orthogonal segments suggest that one should view the
construction of $\alpha$ by duality as a version of the Cauchy--Riemann
equations for harmonic conjugate functions.
}

\subsection{KMC model for dislocation motion}
\label{sec:modeling}
With the asymptotic expression for $\Tran_n(\mu\to\nu)$ given by
Theorem~\ref{th:energy_barriers}, we are now in a position
to apply \eqref{eq:rate} and formulate the KMC model for
dislocation motion we wish to study. In doing so, we make several
modeling assumptions, which we now discuss in detail.

Our first assumption is that the only possible 
transitions are from $\mu\in\Posen$
to $\nu\in\Posen$ satisfying
\begin{gather*}
  \nu-\mu = b_i[\chr_q-\chr_p]\quad
  \text{for some }i\in\{1,\ldots,m\},\\
  \text{with}\quad p^* = q^*+\avec^*\quad\text{for some dual lattice
    nearest--neighbour direction }\avec^*.
\end{gather*}
\alth{
This requirement prevents the following possible situations from
arising:
\begin{enumerate}
\item Multiple dislocations cannot move together in a coherent way: it
  seems reasonable to dismiss this possibility since we consider
  a regime where dislocations are far apart.
\item Single dislocations cannot make successive correlated
  jumps over several lattice sites. Since we consider a low
  temperature regime, we expect the probability of multiple correlated
  jumps to be negligible.
\item Dislocations cannot be spontaneously
  generated in the material during the course of the evolution. In this
  case, we expect the energy barrier for dipole creation to be
  higher than that for the motion of single dislocations, so once again,
  we expect such events to be of very small probability and we therefore
  neglect them.
\end{enumerate}
}
We therefore assume that the
transition time for a dislocation $\mu$ to $\nu$ is 
exponentially distributed with rate
\begin{equation}
  \Rate_n(\mu\to\nu):=\Ent_n(\mu\to\nu)
  \exp\b(-\beta\Tran_n(\mu\to\nu)\b),\tag{\ref{eq:rate}}
\end{equation}
where:
\begin{enumerate}
  \item $\Tran_n(\mu\to\nu)$ is the energy barrier for the 
    transition from $\mu$ to $\nu$ defined by 
    \eqref{eq:energy_barrier_defn},
  \item $\beta=(k_B T)^{-1}$ is the inverse thermodynamic 
    temperature, and
  \item $\Ent_n(\mu\to\nu)$ is the pre--exponential rate 
    factor which is related to the entropic `width' of the
    pathways connecting $\mu$ and $\nu$, and hence also
    depends on the inverse temperature $\beta$.
\end{enumerate}
\alth{
Formula \eqref{eq:rate} may be interpreted as follows: the exponential
factor encodes the probability that thermal fluctuations will result
in the system achieving the potential energy necessary for a transition 
to happen. The prefactor then determines how often such energy levels
will lead to a transition: if the passage between states in the energy
landscape is very `narrow', then even if the system achieves sufficient
energy to exit, it may only rarely find the the pathway to achieve such
a transition.

Our second main assumption will be that
$\Ent_n(\mu\to\nu)=\Ent_0+o(1)$, as $\beta\to\infty$ and $n\to\infty$,
where $\Ent_0$ is independent of $\mu$ and $\nu$.
In the case of a finite--dimensional system with a smooth
potential energy $V$, having local minima at $x$ and $y$, and a
saddle point at $z$ with a single unstable direction where the minimal
energy barrier between $x$ and $y$ is achieved, the form of the
prefactor is (see formula (25) in \cite{K40} for the original
one--dimensional derivation, or \cite{HTB90} for an overview of
variants derived in a variety of situations)
\begin{equation}
  \Ent(\mu\to\nu) = \frac{\sqrt{\gamma^2+4|\lambda_1(z)|}-\gamma}{2\pi}
  \sqrt{\frac{\det\nabla^2V(x)}{|\det\nabla^2V(z)|}}+o(1).
  \label{eq:prefactor}
\end{equation}
Here $\gamma$ is a friction coefficient, with units of
$\text{time}^{-1}$, and
$\lambda_1(z)$ is the eigenvalue of the Hessian at $z$ which
corresponds to the unstable direction. The rate can be reduced if
either the eigenvalues of 
$\nabla^2V(x)$ are made smaller, reducing its determinant, or if the
positive eigenvalues of $\nabla^2V(z)$ are increased. The former
means the potential energy `basin' around $x$ is wider, and the latter
means that the `mountain pass' in the energy landscape through which
the system can travel most easily to arrive at state $\nu$ is narrower.
This coefficient therefore encodes entropic effects related to the shape
of the energy landscape.

In our model, we have shown that
there is a discontinuity in the first derivative at the energy barrier
between states, so the exact expression \eqref{eq:prefactor} cannot be
valid; however, in directions for which second
derivatives exist, the Hessian of the energy at the transition
state and at equilibria are identical, motivating the assumption that
$\Ent_n$ is constant as $n\to\infty$ and $\beta\to\infty$.
We remark that it is usual in practice (except in symmetric situations
where multiple transition pathways with the same energy barrier exist)
to choose a constant prefactor in KMC simulations, since eigenvalue
decompositions of the Hessian of the energy are often be unavailable,
and transition events may be too rare to obtain a sufficiently accurate
numerical estimate of the rate.
}
In order to describe the limit, we define the set of admissible
(macroscale) dislocation positions to be
\begin{equation*}
  \Posc:=\b\{(x_1,\ldots,
  x_m)\in\Dom^m\bsep x_i\in\Dom,
  |x_i-x_j|\geq \eps,
  \dist(x_i,\partial\Dom)\geq \eps,\forall i,j
  \text{ with }i\neq j\b\},
\end{equation*}
and identify $\Posen$ with a subset of this 
space by the embedding
\begin{equation}
  \iota_n:\Posen\to\Posc,
  \quad\text{where}\quad\iota_n\bg(\sum_{i=1}^m b_i\chr_{e_i}\bg)
   = \b(\smfrac1ne_1^*,\ldots,\smfrac1n e_m^*\b).
   \label{eq:iotan_defn}
\end{equation}
It is clear that this map is well--defined, and by endowing
$\Posen$ with the metric
\begin{equation*}
  r_n(\mu,\nu) = \sum_{i=1}^m \smfrac 1n\dist\b(e^*_i,(e'_i)^*\b)
  \quad\text{where}\quad \mu = \sum_{i=1}^m 
  b_i\chr_{e_i}\text{ and }\nu=\sum_{i=1}^m b_i 
  \chr_{e'_i},
\end{equation*}
and $\Posc$ with the metric
\begin{equation*}
  r_\infty(\mu,\nu) = \sum_{i=1}^m \dist(x_i,x'_i)
  \quad\text{where}\quad \mu = (x_1,\ldots,x_m)
  \text{ and }\nu=(x'_1,\ldots,x'_m),
\end{equation*}
$\iota_n$ is an isometric embedding. It is straightforward to see
that each of these spaces is compact.

Given a differentiable function $f:\Posc\to\R$, we will write
$\partial_i f(x)$ to mean the $\R^2$--valued function such that
\begin{equation*}
  \partial_i f(x)\cdot \avec = f(x_1,\ldots,x_i+\avec,\ldots,x_m)-f(x_1,\ldots,x_m) + o(|\avec|) \quad\text{for all }\avec\in\R^2.
\end{equation*}

Let $\DD([0,T];\Posen)$ denote the Skorokhod
space of c\`adl\`ag maps from $[0,T]\subset\R$ with values in
$\Posen$, and denote the space of continuous real--valued functions
defined on $\Posen$ to be $\CC(\Posen;\R)$: this is in fact
the space of all real--valued functions on $\Posen$,
since the metric $r_n$ induces the discrete topology.
Define
\begin{equation*}
  \Nhd_\mu:= \b\{\nu\in \Posen\bsep r_n(\mu,\nu) = d^\L\b\},
  \quad\text{where}\quad d^\L = \cases{
    \frac{\sqrt{3}}{3} & \L=\Tr,\\
    1 & \L=\Sq,\\
    \sqrt{3} & \L=\Hx.
  }
\end{equation*}
Since we expect our modelling assumptions to break down as 
dislocations either approach one another or the domain boundary,
we stop the evolution in such an event. We therefore denote what we
term the \emph{boundary of} $\Posen$, defined to be
\begin{equation*}
  \partial\Posen:=\bg\{ \mu=\sum_{i=1}^m b_i \mathbbm{1}_{e_i}\in\Posen
  \Bsep\exists\nu\notin\Posen\text{ such that }
  r_n(\mu,\nu)=d^\L\bg\}.
\end{equation*}
We consider the sequence of Markov processes $Y^n\in 
D\b([0,T];\Posen\b)$ which are killed on the
boundary $\partial\Posen$, having infinitesimal
generator $\Omega_n:\CC\b(\Posen;\R\b)\to
\CC\b(\Posen;\R\b)$ where
\begin{equation*}
  [\Omega_n f](\mu) := \cases{\displaystyle
  \sum_{\nu\in\Nhd_\mu}\tscale_n\Rate_n(\mu\to\nu)[f(\nu)-f(\mu)],  
  &\mu \in\Posen\setminus\partial\Posen,\\
  0&\mu\in\partial\Posen,
  }
\end{equation*}
and $\Rate_n(\mu\to\nu)$ is defined in \eqref{eq:rate}.
Since $\Rate_n(\mu\to\nu)$ is strictly positive and
bounded for all $\mu,\nu\in\Posen$ and $n\in\N$,
$\Omega_n$ is a bounded linear operator.
Defining $X^n_t:=\iota_n(Y^n_t)$, it follows that $X^n_t$
is a Markov process on the space $\Posc$.

\alth{
\subsection{The Feng--Kurtz approach to Large Deviations 
Principles}
The last of our main results will be to show that in a specific
asymptotic regime, the Markov processes $X^n$ satisfy a Large
Deviations Principle. To do so, we apply the general theory
developed in \cite{FK06}, which provides an approach to
proving such results by demonstrating the convergence of a sequence 
of nonlinear semigroups. For convenience, we provide the
following theorem as a synthesis of the results of Theorem~6.14 and
Corollary~8.29 in \cite{FK06}, adapted to our application.
}

\alth{
\begin{theorem}
\label{th:LDP_synthesis}
Suppose that the following conditions hold:
\begin{enumerate}
\item $M$ is a compact subset of $\R^N$, viewed a metric space with the
  usual metric induced by the Euclidean norm.
\item For all $n\in\N$, $(M_n,r_n)$ is a complete separable metric
space and there exists a sequence $\iota_n:M_n\to M$ of Borel measurable
maps such that for any $x\in M$, there exists $z_n\in M_n$ satisfying
$\iota_n(z_n)\to x$.
\item For each $n\in\N$, $\Omega_n:\CC(M_n;\R)\to\CC(M_n;\R)$ is the
  infinitesimal generator of a Markov process
  on $M_n$. Suppose the martingale problem is well--posed, i.e. for any
  initial distribution $\mu_0$ on $M_n$, the distribution of the Markov
  process at all later times is uniquely determined, and the
  mapping from $y\in M_n$ to trajectories with initial distribution 
  $\delta_y$ is Borel measurable under the weak topology on 
  the space of probability measures defined on $\DD([0,+\infty);M_n)$.
\item For any $n\in\N$, and any $f\in\CC(M_n;\R)$, define the
  \emph{nonlinear generator}
  \begin{equation}
      H_nf(x):=\smfrac{1}{n}\e^{-nf(x)}\b[\Omega_n\e^{nf}\b](x).
      \label{eq:nonlinear_generator}
  \end{equation}
  Let $H$ be an operator mapping $\CC^1(M;\R)$ to the space of bounded
  measurable functions on $M$, which is represented as
  \begin{equation*}
    Hf(x) = \Ham\b(x,\nabla f(x)\b),
  \end{equation*}
  where $\Ham:M\times\R^N\to\R$ satisfies the following
  conditions:
  \begin{itemize}
  \item $\Ham$ is uniformly continuous on the interior of
    $M\times B_r(0)$ for all $r>0$,
  \item $\Ham$ is differentiable in $p$ on the interior of
    $M\times\R^N$,
  \item $\Ham(x,p)=0$ for all $p\in\R^N$ when
    $x\in\partial M$, and
  \item For all $x\in M$, $p\mapsto\Ham(x,p)$ is a convex
    function.
  \end{itemize}
  For each pair $(f,g)$ such that $g=Hf$, there
  exists a sequence $(f_n,g_n)$ such that $g_n=H_nf_n$,
  $\|f\circ \iota_n-f_n\|\to0$, $g_n$ is uniformly bounded, and for any
  sequence $z_n\in M_n$ satisfying $\iota_n(z_n)\to x$, we have
  \begin{equation}
    g^l(x)\leq \liminf_{n\to\infty} g_n(z_n)\leq \limsup_{n\to\infty}g_n(z_n) 
    \leq g^u(x),\label{eq:u+l_convergence}
  \end{equation}
  where $g^l$ and $g^u$ are respectively the lower and
  upper--semicontinuous regularizations of $g$,
  \begin{equation*}
    g^l(x) := \lim_{r\to 0}\inf_{y\in B_r(x)}g(y)\quad\text{and}\quad
    g^u(x) := \lim_{r\to 0}\sup_{y\in B_r(x)}g(y).
  \end{equation*}
\item There exists $\Lag:M\times\R^N\to[0,+\infty]$ such that
\begin{gather}
  \Lag(x,\xi) =\sup_{p\in\R^N}\b\{ \xi\cdot p-\Ham(x,p)
  \b\},\notag\\
  \lim_{|\xi|\to\infty}\frac{\Lag(x,\xi)}{|\xi|}=+\infty\quad
  \text{for all }x\in M\text{ and }\xi\in\R^N,
  \label{eq:growth_condition}
\end{gather}
and for each $x_0\in M$, there exists $x\in\WW^{1,1}([0,T];\R^N)$
satisfying $x(0)=x_0$ and
\begin{equation}
  \int_0^T \Lag\b(x(t),\dot{x}(t)\b)\dt=0.
  \label{eq:gradient_flow_existence}
\end{equation}
\end{enumerate}
Then the sequence of $M$--valued processes $X_n:=\iota_n(Y_n)$ with
$X_n(0)=\iota_n(y_n)$, where $y_n\in M_n$ and
$\iota_n(y_n)\to x_0$ as $n\to\infty$, satisfy a Large Deviations
Principle with rate functional
\begin{equation}
  \Act(x):=\cases{\displaystyle
    \int_0^\infty \Lag(x,\dot{x})\dt &
    x\in\WW^{1,1}\b([0,+\infty);\R^N\b)\text{ with }x(0)=x_0,\\
    +\infty &\text{otherwise.}
  }
\end{equation}
\end{theorem}

\noindent
\S\ref{sec:LDP_proofs} contains the proof of this result, which amounts
to checking that the assumptions above correspond to a series of
conditions in \cite{FK06}.
}

\subsection{Asymptotics for the KMC model}
\label{sec:scaling}
An important condition of Theorem~\ref{th:LDP_synthesis} is the
verification of the convergence of the nonlinear generator, $H_n$.
It will be this which motivates our particular choice of regime after
we have non--dimensionalised the model. Since we are
interested in the physically--relevant case of observing a large system
over a long timescale, we let $\tscale_n\gg1$ be the timescale
of observation, which will be taken relative to the typical
timescale on which a dislocation configuration changes. We then multiply
all rates by this timescale, which we view as corresponding to
observing the process over a long timescale.

Now, recalling the definition of the nonlinear generator given in
\eqref{eq:nonlinear_generator}, suppose that $f\in\CC^1(\Posc;\R)$, and
and let $x_n=(\smfrac1n e_1^*,\ldots,\smfrac1n e_m^*)$.
By Taylor expanding $f$, we find that
\begin{align*}
  H_n(f\circ\iota_n)(x_n) &= 
  \sum_{i=1}^m\sum_{j=1}^{\Num^*}\frac{\tscale_n\Rate_n(\mu\to\nu)}
  {n}\b[\exp\b(\partial_i f(x_n)\cdot 
  \svec_{i,j}+o(1)\b)-1\b]\quad\text{as }n\to\infty,
\end{align*}
where $\svec_{i,j}$ are the nearest neighbour directions in
$\L^*$ at $e_i^*$, and $\Num^*$ is the number of nearest 
neighbours in $\L^*$.
Now, by applying Theorem~\ref{th:energy_barriers} and the
assumption that $\Ent_n(\mu\to\nu) = \Ent_0+o(1)$, we
have that
\alth{
\begin{gather*}
  \frac{\tscale_n\Rate_n(\mu\to\nu)}{n} = 
  \frac{\tscale_n\Ent_0\e^{-\beta\lambda c_0}}{n}\exp\bg[
  -\frac{\beta\lambda}{2n}
  \partial_i\E(x_n)\cdot \svec_{i,j}\bg]+o\b(\smfrac{\tscale_n}{n}\b),\\
  \text{where}\quad \E(x):=\sum_{i=j}^mb_j^2\bar{y}_j(x_j)
  -\sum_{\substack{i,j=1 \\i<j}}^m \smfrac12 b_ib_j\Gc_{x_i}(x_j).
\end{gather*}
Here, following \cite{ADLGP14} we have defined the
\emph{renormalised energy}, $\E$. $-\partial_i\E(x)$ is the \emph{Peach--K\"ohler} force on the dislocation at $x_i$, and hence the gradient
flow dynamics of $\E$ corresponds to Discrete Dislocation Dynamics.
We identify two parameters in this expression,
\begin{equation*}
  A:= \frac{\tscale_n\Ent_0\e^{-\beta\lambda c_0}}{n }
  \quad\text{and}\quad B:= \frac{\beta\lambda}{2n},
\end{equation*}
which are dimensionless, upon recalling that:
\begin{enumerate}
  \item $\tscale_n$ has units of time,
  \item $n$ is the diameter of the domain relative to a fixed
    reference domain, and hence is dimensionless,
  \item $\beta=(k_BT)^{-1}$ is the inverse thermodynamic temperature
    of the system per particle,
  \item $\lambda$ has units of energy per particle, and
  \item $\Ent_0$ is the rate of successful exits from
      $\mu\to\nu$, and has units of $\text{time}^{-1}$.
\end{enumerate}
We may think of $\Ent_0 \e^{-\beta\lambda c_0}$ as being the number of
times a dislocation hops a single spacing in the full lattice per unit
time, when subject to zero stress. Dividing by
$n$ and multiplying by $\tscale_n$, this becomes the
proportion of the domain crossed per proportion of time over which
the system is observed.
The product $\beta\lambda$ is the ratio between the potential energy
required to  allow transitions to occur relative to the available
thermal energy; dividing by $n$ gives this quantity relative
to the ratio between the lattice spacing and the domain diameter.

We therefore consider the asymptotic 
regime where $n\to\infty$ with $A$ and $B$ are held constant: 
assuming that $\lambda$ and $\Ent_0$ remain constant as $n$, $\beta$ and
$\tscale_n$ vary, this entails that $\beta$ and $\tscale_n$ tend to
infinity, and hence we consider a regime in which a large system is
observed at low temperature for a long time. In this regime, we obtain
the following result, which is proved in \S\ref{sec:LDP_proofs} as an
application of Theorem~\ref{th:LDP_synthesis}.
It corresponds to a rigorous validation of the equations of
two--dimensional Discrete Dislocation Dynamics
\cite{AG90,BulatovCai,BC04} for screw dislocations
in the given physical parameter regime.

\begin{theorem}
\label{th:LDP_main}
Suppose that $\L=\Hx$, $\L=\Sq$ or $\L=\Tr$ and $X^n_0=\iota_n(x^n)$
where $x^n\to x_0\in\Posc$ as $n\to\infty$. Then the sequence of
processes $X^n_t$ satisfies a Large Deviation Principle with a good 
rate function as $n\to\infty$ with $A$ and $B$ fixed.

Moreover, in each case, the rate function is minimised by the unique
solution of the ODE
\begin{equation}
  \dot{x} = \mathcal{M}^\L_{A,B}\b[-\nabla \mathcal{E}(x)\b],
  \quad\text{with}\quad x(0) = x_0,
  \label{eq:DDD}
\end{equation}
where $\mathcal{E}:\Posc\to\R$ is the renormalised energy, and
$\mathcal{M}^\L_{A,B}$ is the mobility function
\begin{equation}
  \mathcal{M}^\L_{A,B}[\xi] :=\cases{\displaystyle
  \sum_{i=1}^m\sum_{j=1}^6A\sinh(B\xi_i\cdot\avec_j)\avec_j & \L=\Hx,
  \\[1mm]\displaystyle
  \sum_{i=1}^m\sum_{j=1}^4A\sinh(B\xi_i\cdot\evec_j)\evec_j & \L=\Sq,
  \\[1mm]\displaystyle
  \sum_{i=1}^m\frac{\sum_{j=1}^6A\sinh(B\xi_i\cdot\avec_j)\avec_j}{
  \sum_{j=1}^32\cosh(B\xi_i\cdot \smfrac13\b[\avec_{2j}+\avec_{2j-1}\b]\b)}
  &\L=\Tr,}
\end{equation}
where $\avec_j$ and $\evec_j$ are as defined in
\S\ref{sec:dual_examples}.
\end{theorem}
}

\subsection{Generalised gradient flows and mobility
functions}
\label{sec:interpretation}
As has been noted in \cite{MPR14,BP15}, there is a close link between
minimisers of Large Deviations rate functionals and gradient flows: we
also observe this phenomenon here in the cases where $\L=\Hx$ and
$\L=\Sq$. In those cases, it is shown in \S\ref{sec:LDP_Hx+Sq} that
the rate functional takes the form
\begin{equation*}
  \Act^\L_{A,B}(x) =
  B\hspace{-1mm}\int_0^T \hspace{-2mm}\Phi^\L_{A,B}(\dot{x})
  +\Psi^\L_{A,B}(-\nabla\mathcal{E}\b(x)\b)
  +\<\nabla\mathcal{E}(x),\dot{x}\>\dt
\end{equation*}
where $\mathcal{M}^\L_{A,B}=\nabla \Psi^\L_{A,B}$, and $\Phi^\L_{A,B}$ is
the Legendre--Fenchel transform of $\Psi^\L_{A,B}$. This entails that the
minimiser of the rate functional is a solution of a generalised gradient
flow in the sense described in \cite{M16}.
Furthermore, as in Theorem~3.1 of \cite{BP15}, we may recover a
quadratic dissipation in the limit where $A\to\infty$ and $B\to0$.

\begin{proposition}
Suppose that $A\to\infty$ and $B\to0$ with 
$AB\to\omega$. Then
\begin{equation*}
  \mathcal{M}^\L_{A,B}[\xi]\to\smfrac12\omega\V^* \xi
\end{equation*}
uniformly on compact subsets of $\R^{2m}$, where
$\V^*$ is the constant $\V$ for $\L^*$. Consequently, for sufficently
small $T$, solutions $x:[0,T]\to\Dom^m$ of \eqref{eq:DDD} converge
uniformly converge to the solution of
\begin{equation*}
  \dot{x} = -\smfrac12\omega\V^*\nabla\mathcal{E}(x),\quad\text{with}
  \quad x(0) = x_0
\end{equation*}
as $A\to+\infty$ and $B\to0$ with $AB\to\omega$.
\end{proposition}
\medskip

\noindent
The proof of this result follows directly from representing
$\mathcal{M}^\L_{A,B}$ via series expansion, and we omit it.
Recalling the interpretation of $A$ and $B$ given in 
\S\ref{sec:scaling}, this could be viewed as suggesting a Large
Deviations Principle in the regime where the thermal energy is much
larger than the energy barrier to dislocation motion, but where the
proportion of the cylinder crossed by a dislocation during the
observed time is small.
However, recalling the definition of $A$ and $B$ from
\S\ref{sec:modeling}, we note that
\begin{equation*}
  AB = \frac{\tscale_n}{n^2}\Ent_0\beta\lambda\exp\b(c_0\beta\lambda\b).
\end{equation*}
If $\beta$, $\lambda$ and $\Ent_0$ are fixed as $n\to\infty$,
choosing $AB\to\omega$ corresponds to a diffusive scaling of the Markov
process. We would therefore expect that randomness would persists on a
macroscopic scale in such an asymptotic regime, a connection which
should be explored in future work.

Finally, we remark that is also possible to consider the other
scaling regime analysed in \cite{BP15}, in which $B\to\infty$ with
$\log(A) = -c_1 B$ for some $c_1$. In terms of the parameters
described in \S\ref{sec:modeling}, this entails that
\begin{equation*}
  \log\frac{\tscale_n\Ent_0}{n} = \beta\lambda \B(c_0
  -\frac{c_1}{2n}\B).
\end{equation*}
Assuming that $\lambda$ and $\Ent_0$ remain fixed, the only way in 
which this scaling regime can be attained is if $n$ remains small and
fixed, with $\beta\to\infty$ and $\tscale_n\to\infty$. Since our
analysis relies upon the fact that $n\to\infty$ to ensure that
lower--order terms vanish in $\Rate_n(\mu\to\nu)$, we cannot be
certain that this limit corresponds to a physically--relevant
limit, and thus we do not study it here.

\section{Proof of Theorem~\ref{th:equivalence}}
\label{sec:elliptic_estimates}
In this section, we develop discrete elliptic estimates which
will allow us to prove Theorem~\ref{th:equivalence}; many of the tools
used are analogous to those used in the regularity theory of
scalar elliptic partial differential equations.
\alth{To motivate our approach, and to provide the reader with some
intuition, we recall the following result, proved in \S3.4 of
\cite{GilbargTrudinger}: given $Q=\{x\in\R^2\sep
|x\cdot\evec_1|,|x\cdot\evec_2|\leq d\}$,
$f\in\CC(\overline{Q})$ and $u\in\CC^2(Q)\cap \CC(\overline{Q})$
satisfying $\Delta u=f$, then
\begin{equation*}
  \b|\nabla u(0)\cdot \evec_i\b| \leq \frac{2}{d}\sup_{\partial Q}|u|
  +\frac{d}{2}\sup_{Q}|f|.
\end{equation*}
Our approach will be to apply the discrete analogue of the techniques
used to prove this bound, i.e. the maximum principle and elementary
potential theory. The application of these techniques in combination
with fine residual estimates, will then allow us to conclude the proof.
}

\subsection{The discrete Poisson boundary value problem}
We begin by proving existence of solutions to the Poisson boundary value problem in a general path--connected
subcomplex $\Dom_{n,0}$.

\begin{lemma}
\label{th:Pois_BVP_existence}
Suppose $\Dom_{n,0}$ is a path--connected lattice subcomplex;
let $g:\Out(\Dom_{n,0})\to\R$, and $f:\Int(\Dom_{n,0})\to\R$ then there
exists a unique solution $u\in\Wsc(\Dom_{n,0})$ to the problem
\begin{equation*}
  \Lap u = f\text{ in }\Int(\Dom_{n,0})
  \quad\text{with}\quad u=g \text{ on }\Out(\Dom_{n,0}).
\end{equation*}
\end{lemma}

\begin{proof}
We employ a discrete version of the Dirichlet principle:
extend $g$ to a $0$--form by defining $g(e):=0$ for all
$e\in \Int(\Dom_{n,0})$, and let $I:\Wsc_0(\Dom_{n,0})\to\R$ be given
by
\begin{equation*}
  I(v) := \smfrac12(\df (v+g),\df (v+g)) -\int_{\Int(\Dom_{n,0})} \hspace{-3mm} fv.
\end{equation*}
It is straightforward to verify that this functional is twice
Gateau--differentiable, with
\begin{equation*}
  \<DI(v),u\> = (\df(v+g),\df u)
  -\int_{\Int(\Dom_{n,0})}fu,\quad\text{and}\quad
  \<D^2I(v)u,u\>= (\!(u,u)\!).
\end{equation*}
It follows that $I$ is strictly convex, so has a unique
minimiser.
By setting $u=\chr_e$ for any $e\in\Int(\Dom_{n,0})$, this 
minimiser $v$ satisfies
\begin{equation*}
  \Lap (v+g) = f \text{ in }\Int(\Dom_{n,0}),
\end{equation*}
and $v+g=g$ on $\Out(\Dom_{n,0})$ by definition.
\end{proof}

Our next auxiliary result is to prove the following
discrete maximum principle.

\begin{lemma}
\label{th:maxmin_prin}
Suppose that $u\in\Wsc(\Dom_{n,0})$. Then
\begin{align*}
  \Lap u \geq 0 \text{ on } \Int(\Dom_{n,0})
  \quad&\text{implies}\quad\min_{e\in \Dom_{n,0}} u(e) = 
  \min_{e\in \Out(\Dom_{n,0})} u(e),\quad\text{and}\\
  \Lap u \leq 0 \text{ on } \Int(\Dom_{n,0})
  \quad&\text{implies}\quad\max_{e\in \Dom_{n,0}} u(e) = 
  \max_{e\in \Out(\Dom_{n,0})} u(e).
\end{align*}
\end{lemma}

\begin{proof}
We prove only the former statement, the proof of the latter 
being almost identical. Suppose that $u$ satisfies
$\Lap u\geq 0$ on $\Int(\Dom_{n,0})$, and that there exists
$e\in\Int(\Dom_{n,0})$ such that
\begin{equation*}
  u(e) = \min_{e'\in \Dom_{n,0}} u(e').
\end{equation*}
Either $e\in\Out(\Dom_{n,0})$, so there is nothing to prove, or
else $e\in\Int(\Dom_{n,0})$. Since
\begin{equation*}
   \Lap u(e) =\hspace{-3mm}\sum_{\substack{e'\in 
   \Dom_{n,0}\\\dist(e',e) = 1}}\hspace{-3mm}[u(e)-u(e')]\geq 0,
\end{equation*}
it follows that $u(e) = u(e')$ for all $e'$ with 
$\dist(e,e')=1$. Iterating, and using the fact that
$\Dom_{n,0}$ is finite, we find that $\min_{e\in \Dom_{n,0}} u(e) = 
\min_{e\in\Out(\Dom_{n,0})}u(e)$, as required.
\end{proof}

\subsection{Green's functions in the full lattice}
We next assert the following lemma, concerning the existence 
of a full lattice Green's function $\Gd$.

\begin{lemma}
\label{th:full_latt_Gfunc}
Suppose that $\L = \Sq$, $\Tr$, or $\Hx$.
Then there exists a lattice Green's function 
$\Gd\in\Wsc(\L_0)$ such that
\begin{equation*}
  \Gd(0) = 0,\qquad\Lap \Gd = \chr_0.
\end{equation*}
In addition:
\begin{enumerate}
  \item $\Gd$ is invariant under the group of lattice point
symmetries, i.e. if $\mR:\R^2\to\R^2$ is an orthogonal linear
transformation such that $\mR\L_0 = \L_0$, then
\begin{equation*}
  \Gd(\mR e) = \Gd(e).
\end{equation*}
  \item $\displaystyle\sup_{e\in\L_1}|\df \Gd(e)|= \Num^{-1}$, 
    where $\Num$ is defined in \eqref{eq:constants_defn}.
  \item There exists a constant $C^\L\in\R$ such that if $u(e):= 
    \Gd(e)+C^\L+\smfrac1{\V\pi}\log|\dist(0,e)|$ for
    $e\in\L\setminus\{0\}$, then
  \begin{align}
    |u(e)|&\lesssim |\dist(e,0)|^{-1}\log|\dist(e,0)|
    \label{eq:corr_bnd_func}\\
    \text{and}\quad |\df u(e)|&\lesssim |\dist(e,0)|^{-2}
    \log|\dist(e,0)|.\label{eq:corr_bnd_deriv}
  \end{align}
\end{enumerate}
\end{lemma}
\medskip

The usual method of constructing $\Gd$ is via the Fourier transform, and
the existence of such a Green's function in the case of Bravais
lattices is a classical assertion, as is the symmetry asserted in (1).
The  bounds \eqref{eq:corr_bnd_func} and
\eqref{eq:corr_bnd_deriv} are proved in Theorem~3.5 of
\cite{EOS15} for all Bravais lattices, thus covering the cases where
$\L=\Sq$ and $\L=\Tr$. It therefore remains to prove (2)
and the other results in the $\L=\Hx$ case: the main observation used
here is that $\Hx$ may be viewed as a subset of $\Tr$.

\begin{proof}
We first prove (2) for $\L=\Sq$ and $\L=\Tr$. Fix $\avec$ to be
a nearest neighbour direction in the lattice. By the symmetry
of $\Gd$ from (1), we have
\begin{equation*}
  1=\Lap \Gd(0) = \Num\Gd(0)-\Num\Gd(0+\avec) = -\Num\Gd(0+\avec).
\end{equation*}
Hence $\df \Gd([0,0+\avec])=\Num^{-1}$. Now consider $v\in\Wsc(\L_0)$
defined to be
\begin{equation*}
  v(e):= \Gd(e+\avec)-\Gd(e).
\end{equation*}
It follows that $\Lap v=\chr_{0-\avec}-\chr_0$.  Applying
Lemma~\ref{th:maxmin_prin} on the lattice subcomplex induced by the set
\begin{equation*}
  B'_r:=\b\{e\in\L_0\bsep\dist(e,0)\leq r, e\neq 
  0,0-\avec\b\},
\end{equation*}
we note that the maximum and minimum of $v$ are attained on
$\Out(B'_{r,0})$, since $\Gd$ is harmonic on $\Int(B'_{r,0})$.
Now, applying (3) and letting $r$ tend to infinity implies the
desired result, noting that $v(0) = -v(0-\avec)= \Num^{-1}$.

It remains to prove the theorem for the case where $\L=\Hx$.
Recall from \S\ref{sec:dual_examples} that $\Hx$ may be written as
\begin{equation*}
  \Hx = \sqrt{3}\mR_4\Tr \cup  \b(\sqrt{3}\mR_4\Tr + \evec_1\b),
\end{equation*}
and define $G^\Hx\in\Wsc(\Hx_0)$ to be
\begin{equation*}
  G^\Hx(e) := \cases{\displaystyle
     3\,G^\Tr(\mR_4^Te/\sqrt{3}) & e\in\sqrt{3}\mR_4\Tr,\\[1mm]
     \displaystyle
     \sum_{e'|\dist(e',e)=1}\hspace{-3mm} 
     G^\Hx(e') &e\in\sqrt{3}\mR_4\Tr+\evec_1,
  }
\end{equation*}
where $G^\Tr$ is the lattice Green's function for $\L=\Tr$.
We note that $\Lap G^\Hx(e) = 0$
by definition for $e\in\sqrt{3}\mR_4\Tr+\evec_1$, and for 
$e\in\sqrt{3}\mR_4\Tr$,
\begin{align*}
  \Lap G^\Hx(e)&=9\,G^\Tr(e/\sqrt{3}) - \sum_{e'|\dist(e',e)=1}
  \bg[\sum_{e''\sep\dist(e'',e')=1} 
  \hspace{-3mm}G^\Hx(e'')\bg],\\
  & = 6\,G^\Tr(e/\sqrt{3}) - \sum_{e'|\dist(e',e)=\sqrt{3}}
  G^\Tr(e'/\sqrt{3}),\\
  & = \chr_0(e).
\end{align*}
Moreover $G^\Hx(0)=G^\Tr(0)=0$, and the symmetry of $G^\Tr$ also 
implies (1) for $G^\Hx$.

Let $C^\Tr$ be the constant in statement (3) for
the case where $\L=\Tr$, and for $e\in\Hx\setminus\{0\}$, define
\begin{equation*}
  u^\Hx(e):=G^\Hx(e)+\smfrac1{2\pi}\log\b|
  \dist\b(e/\sqrt{3},0\b)\b|+3C^\Tr=G^\Hx(e)+\smfrac1{2\pi}
  \log|\dist(e,0)|+3C^\Tr-\smfrac1{4\pi}\log(3).
\end{equation*}
we see that for $e\in\sqrt{3}\mR_4\Tr$, $u^\Hx$ satisfies 
\eqref{eq:corr_bnd_func} by the assertion for the case where 
$\L=\Tr$. For $e\in\sqrt{3}\mR_4\Tr\setminus\{0\}$, define 
$v^\Tr(e):=G^\Tr(e/\sqrt{3})
+\smfrac1{6\pi}\log\b|\dist\b(e/\sqrt{3},0\b)\b|$; then for
$e\in \sqrt{3}\mR_4\Tr+\evec_1$, we have
\begin{equation*}
  u^\Hx(e) = \smfrac1{2\pi}\log\b|\dist\b(e/\sqrt{3},0\b)\b|+
  \sum_{e'|\dist(e,e')=1}\B(
  v^\Tr(e')-\smfrac{1}{6\pi}\log\b|\dist\b(e'/\sqrt{3},0\b)\b|
  \B).
\end{equation*}
Since $\log|x|$ is harmonic away from $0$, Taylor expanding to
third--order about the point $e$ and using the symmetry of $\Hx$ implies
that
\begin{equation}
  \smfrac1{2\pi}\log\b|\dist\b(e/\sqrt{3},0\b)\b|
  -\sum_{e'|\dist(e,e')=1}\smfrac{1}{6\pi}
  \log\b|\dist\b(e'/\sqrt{3},0\b)\b|\lesssim |\dist(e,0)|^{-3}.
  \label{eq:Hx_log_expansion}
\end{equation}
Applying this estimate and \eqref{eq:corr_bnd_func} for $\L=\Tr$,
we obtain that
\begin{equation*}
  |u(e)|\lesssim |\dist(e,0)|^{-1}\log|\dist(e,0)|
\end{equation*}
for all $e\in\Hx\setminus\{0\}$.

To demonstrate \eqref{eq:corr_bnd_deriv}, suppose
without loss of generality that
$e\in\sqrt{3}\mR_4\Tr$ and $e+\avec\in\sqrt{3}\mR_4\Tr+\evec_1$
for some nearest neighbour direction
$\avec$. Recalling the definition of $\avec_i$ from
\S\ref{sec:dual_examples}, for some $i$, we have
\begin{multline*}
  |\df u^\Hx([e,e+\avec])| =\b|v^\Tr\b(e+\sqrt{3}\avec_{i+1}\b)
  +v^\Tr\b(e+\sqrt{3}\avec_i\b)-2v^\Tr(e)\b|\\
  + \smfrac1{2\pi}\log\b|\dist(e+\avec,0\b)/\sqrt{3}\b|
  -\hspace{-4mm}\sum_{e'|\dist(e+\avec,e')=1}\hspace{-4mm}\smfrac{1}{6\pi}
  \log\b|\dist(e',0)/\sqrt{3}\b|.
\end{multline*}
Using the definition of $v^\Tr$, and then applying statement (3) in the 
case $\L=\Tr$ as well as \eqref{eq:Hx_log_expansion} gives the result.
\end{proof}

\subsection{The harmonic measure and interior differential
estimates}
We now define the \emph{harmonic measure}, which allows
us to express functions which are harmonic in a region in terms
of their boundary values. In order to do so, we introduce
$Q^r$, which should be thought of as `balls of radius $r$' in the
lattice, and are defined to be:
\begin{equation}
  Q^r:=\cases{
  [-r,r]^2\cap\L & \L=\Sq,\\[1pt]
  \B\{x\in\R^2\Bsep \b|x\cdot(\avec_1+\avec_2)\b|,
  |x\cdot(\avec_2+\avec_3)|,|x\cdot(\avec_3+\avec_4)|\leq \smfrac12r
  \B\}\cap \L
  &\L= \Tr,\Hx.
  }
  \label{eq:lattice_cube}
\end{equation}

\begin{lemma}
\label{th:H_mes}
Let $Q^r$ be as defined in \eqref{eq:lattice_cube}.
Then for each  $e\in\Out(Q^r_0)$, there exists 
$\omega^r_e\in\Wsc(Q^r_0))$ satisfying
\begin{equation*}
  \Lap \omega^r_e = 0\text{ in }\Int(Q^r_0),
  \quad\text{with}\quad \omega^r_e =
  \frac{\chr_e}{\#\Out(Q^r_0)}
  \text{ on }\Out(Q^r_0).
\end{equation*}
In addition, $\omega^r_e$ satisfies the following properties:
\begin{enumerate}
  \item If $u\in\Wsc(Q^r_0)$ is harmonic in $Q^r_0$,
  then for any $e'\in\Dom_{n,0}$,
  \begin{equation*}
    u(e') = \sum_{e\in\Out(Q^r_0)} \omega^r_e(e')u(e).
  \end{equation*}
  \item There exists a constant $C^\L>0$ 
    depending only on $\L$ such that
   \begin{equation}
     |\df \omega^r_e([0,0+\avec])|\leq C^\L \log(r)r^{-2}
     \label{eq:H_mes_deriv_est}
   \end{equation}
   for any nearest--neighbour direction $\avec$.
\end{enumerate}
\end{lemma}
\medskip

The function $\omega^r_e$ is called the harmonic measure, and
enjoys a variety of interpretations, both probabilistic and
and analytic: for further detail, we refer the reader to 
\cite{GM05}. Its principal use will be as a tool
by which we can estimate the effect of the boundary conditions
on the solution in the domain interior.

The existence of $\omega^r_e$ and statement (1) follow 
directly from Lemma~\ref{th:Pois_BVP_existence}. In the case 
where $\L=\Sq$, a proof of \eqref{eq:H_mes_deriv_est} with the 
improved upper bound $C^\L r^{-2}$ is given in Lemma~3 of
\cite{Guadie} using an explicit construction of $\omega^r_e$.
Further results on the harmonic measure in the square lattice
may also be found in Chapter~8 of \cite{LL10}.

\begin{proof}
It remains to prove (2). We use the discrete analogue of Green's
formula:
\begin{equation*}
  \int_{\Int(Q^r_0)}\hspace{-6mm} u\Lap v-v\Lap u =
  \int_{\Out(Q^r_0)}\hspace{-6mm} u\Lap v-v\Lap u =
  \sum_{e\in\Out(Q^r_0)} \bg[u(e)\bg(\int_{\del e}\df v\bg)-
  v(e)\bg(\int_{\del e}\df v\bg)\bg],
\end{equation*}
which follows by applying \eqref{eq:IBP} to the extension of $u,v,\df u$ and $\df v$ by $0$ to the full lattice complex.
Now, consider $v$ which is the solution to
\begin{equation*}
  \Lap v= 0,\qquad v(e) = \Gd(e+\avec)-\Gd(e)\text{ on }
  \Out(Q^r_0).
\end{equation*}
Such $v$ clearly exists by Lemma~\ref{th:Pois_BVP_existence},
and by applying Lemma~\ref{th:maxmin_prin} and then
Lemma~\ref{th:full_latt_Gfunc}, we obtain that
\begin{equation}
  \sup_{e'\in Q^r_1}|\df v(e')|\leq
  2\sup_{e'\in Q^r_0}|v(e')|\leq
  2\sup_{e'\in \Out(Q^r_0)}|v(e')|
  \leq C^\L r^{-1}\log(r).\label{eq:h_mes_proof_1}
\end{equation}
Defining $u\in\Wsc(Q^r_0)$ to be $u(e) := \Gd(e+\avec)-\Gd(e)-
v(e)$, $u$ vanishes on $\Out(Q^r_0)$, and we have that
$\Lap u = \chr_{e+\avec}-\chr_e$, and hence
\begin{align*}
  \df\omega^r_e([0,0+\avec])
  &=\int_{\Int(Q^r_0)}\hspace{-5mm}
  \omega^r_e\Lap u-u\Lap \omega^r_e
  =\int_{\Out(Q^r_0)}\omega^r_e\Lap u-u\Lap \omega^r_e
  =\frac{\Lap u(e)}{\#\Out(Q^r_0)}.
\end{align*}
Now, applying \eqref{eq:corr_bnd_deriv} and \eqref{eq:h_mes_proof_1},
we obtain
\begin{equation*}
  |\Lap u(e)|\leq C^\L r^{-2}\log(r)
\end{equation*}
which completes the proof.
\end{proof}

The harmonic measure now allows us to obtain the following interior
bound on the differential of a harmonic lattice form $u$.

\begin{lemma}
\label{th:interior_estimates}
Suppose that $u\in\Wsc(\Dom_{n,0})$ satisfies $\Lap u = 0$ and
and $u=g$ on $\Out(\Dom_{n,0})$ with $g:\Out(\Dom_{n,0})\to\R$.
Then there exists a constant $C^\L>0$ depending only on $\L$ 
such that
\begin{equation*}
  |\df u(e)| \leq C^\L 
  \frac{\log(\dist(e,\Out(\Dom_{n,0}))}{\dist(e,\Out(\Dom_{n,0}))}\sup_{e'\in\Out(\Dom_{n,0})}|g(e')|\qquad\text{for any }e\in 
  \Dom_{n,1}.
\end{equation*}
\end{lemma}

\begin{proof}
Suppose that $e=[e_0,e_1]\in \Dom_{n,1}$, and let $\xvec\in\Dom$ be the
vector corresponding to $e_0$. Let $Q^r$ be as 
defined in \eqref{eq:lattice_cube}, where $r =\lfloor
\dist(e_0,\Out(\Dom_{n,0}))\rfloor$; then
$\xvec+Q^r_0\subset \Dom_{n,0}$, and statement (1) in
Lemma~\ref{th:H_mes} implies that
\begin{equation*}
  w(e) = \sum_{e'\in\Out(\xvec+Q^r_0)} 
  \hspace{-5mm}\omega^r_{e'}(e)
  w(e'),\quad\text{so}\quad
  \df w(e) = \sum_{e'\in\Out(\avec+Q^r_0)} 
  \hspace{-5mm}\df \omega^r_{e'}(e)w(e').
\end{equation*}
Applying statement (2) of Lemma~\ref{th:H_mes} and 
Lemma~\ref{th:maxmin_prin}, it follows that
\begin{equation*}
  |\df w(e)|\leq \sup_{e'\in\Out(\xvec+Q^r_0)}\hspace{-3mm}
  |w(e')|
  \sum_{e'\in\Out(\xvec+Q^r_0)}\hspace{-3mm}|\df \omega^r_{e'}(e)|\leq C^\L 
  \frac{\log(r)}{r}\sup_{e'\in\Out(\Dom_{n,0})} |g(e')|.\qedhere
\end{equation*}
\end{proof}

\subsection{Asymptotics for Green's functions on finite
subcomplexes}
We have now collected the necessary analytical tools with which
we will prove Theorem~\ref{th:equivalence}: our final auxiliary 
result is the following precise description of 
the differential of solutions to \eqref{eq:latt_grn_func_Dom}.

\begin{theorem}
\label{th:interior_asymptotics}
Suppose that $\mu=\sum_{i=1}^m b_i\chr_{e_i}\in
\Posen$, and let $G_{\mu^*}\in\Wsc(\Dom^*_{n,0})$ be the solution to 
\eqref{eq:latt_grn_func_Dom}. Let $e\in \Dom_{n,0}^*$ with 
$[e,e+\avec]\in \Dom^*_1$, and let $x\in\R^2$ correspond to
the dual 0--cell $e$; then we have
  \begin{multline*}
  \df^* G_{\mu^*}([e,e+\avec])=b_i\df^*\Gdst([e-\xvec_i,e+\avec-
  \xvec_i])+
  n^{-1}\B[b_i\nabla\bar{y}_i(\smfrac1n x)\cdot \avec+\sum_{j\neq i}
  b_j\nabla \Gc_{\xvec_j}(\smfrac1n x)\cdot \avec\B]\\
  + O\b(n^{-1-\delta}\log(n)\b),
  \end{multline*}
where:
\begin{enumerate}
  \item $\Gdst$ is the full lattice Green's function for 
    $\L^*$, whose existence was asserted in
    Theorem~\ref{th:full_latt_Gfunc},
  \item $e_i^*$ minimises $\dist(x,e^*_i)$ over all $i=1,
    \ldots,m$,
  \item for each $i$, $\xvec_i\in\Dom$ satisfies 
    $\dist(\xvec_i,\smfrac1ne_i^*)\leq \smfrac1n$,
  \item $\Gc_{y}$ is the continuum Dirichlet Green's
    function on $\Dom$ corresponding to the point $y$, 
    i.e. the solution to 
    \begin{equation*}
       -\Delta \Gc_{y}(\cdot) = \smfrac{\V}{2}\delta(\cdot-y)
       \text{ in }\Dom,\quad\text{with}\quad \Gc_{y}
       = 0\text{ on }\partial\Dom,
    \end{equation*}
  \item $\bar{y}_i$ solves
    \begin{equation*}
      -\Delta \bar{y}_i=0\text{ in }\Dom,\quad\text{with}\quad
      \bar{y}_i(s) = \smfrac1{\V\pi}\log(|s-x_i|)\text{ 
      on }\partial\Dom,
    \end{equation*}
  \item $\delta>0$ is an exponent which depends only on 
    $\varphi_l$, the interior angles at the corners of the lattice
    polygon $\Dom$, and
  \item $O(n^{-1-\delta}\log(n))$ denotes an error term which is
    uniform for all $\mu\in\Posen$.
\end{enumerate}
\end{theorem}
\medskip

The proof of this result is technical, so we first outline the main
strategy, which is similar in spirit to the approach taken in the
proof of  Theorem~3.3 in \cite{HO15}. We decompose $G_{\mu^*}$ as a 
sum of
\begin{enumerate}
\item full lattice Green's functions restricted to $\Dom^*_{n,0}$,
\item continuum boundary correctors $\bar{y}_i$, and
\item discrete correctors.
\end{enumerate}
Each of these components are treated separately, applying
Lemma~\ref{th:full_latt_Gfunc}, the regularity theory of
\cite{Grisvard}, the maximum principle proved in 
Lemma~\ref{th:maxmin_prin} and the interior estimate of
Lemma~\ref{th:interior_estimates} to analyse each piece.
Since the entire proof
takes place in the dual complex $\Dom^*_n$, for brevity we drop
$*$ from our notation throughout.

\subsubsection*{Decomposition of $G_\mu$.}
For $i=1,\ldots,m$, let $\xvec_i\in\R^2$ be the vector corresponding to 
the point $e_i$. We begin by decomposing
\begin{equation}
  G_\mu(e) = \sum_{i=1}^m b_i\Gd(e-\xvec_i) +
  \sum_{i=1}^m \bar{y}^n_i(e)+ u(e)+v(e),
  \label{eq:G_func_dom_decomp}
\end{equation}
where:
\begin{enumerate}
  \item $\Gd$ is the full (dual) lattice Green's function,
  \item $\bar{y}^n_i(x):=\bar{y}_i(\smfrac1n x)$, where
  $\bar{y}_i$ solves $\Delta\bar{y}_i=0$ on $\Dom$ with
  boundary values
  \begin{equation*}
    \bar{y}_i(x) = \smfrac1{\V\pi}b_i
    \log(|x-\xvec_i|),
  \end{equation*}
  \item $u$ is the solution to the discrete Poisson problem
  \begin{equation*}
    \Lap u = 0\quad\text{with}\quad u(e) = -\sum_{i=1}^m
    \b[b_i\Gd(e-\xvec_i)+\bar{y}_i^n(e)\b]\text{ on}\quad
    \Out(\Dom_{n,0}),\text{ and}
  \end{equation*}
  \item $v$ is the solution to the discrete Poisson problem
  \begin{equation*}
    \Lap v = -\sum_{i=1}^m\Lap \bar{y}_i^n\quad
    \text{with}\quad v(e) = 0\text{ on }
    \Out(\Dom_{n,0}).
  \end{equation*}
\end{enumerate}
In combination, Lemma~\ref{th:full_latt_Gfunc}, the theory of
boundary value problems on polygons in \cite{Grisvard}, and 
Lemma~\ref{th:Pois_BVP_existence} allow us to conclude that each
of the terms in this decomposition is well--defined. Furthermore,
equality follows since solutions to the Poisson problem are
unique by Lemma~\ref{th:Pois_BVP_existence}.

\subsubsection*{Regularity of $\bar{y}_i$.}
\label{sec:ybar_regularity}
We now recall some facts concerning the regularity of 
$\bar{y}_j$ from \cite{Grisvard}. Applying Theorem~6.4.2.6 in
\cite{Grisvard}, there exists $\sigma\in(0,1)$ such that
$\bar{y}_i$ lies in the space
\begin{equation*}
  \mathcal{X}:=\CC^{4,\sigma}(\Dom)+\mathrm{span}\b\{
  \mathfrak{S}_{l,m}\bsep m\in\N,0<m<(4+\sigma)\varphi_l/\pi\b\},
\end{equation*}
where $\mathfrak{S}_{l,m}$ is given in polar coordinates
$(r_l,\theta_l)$ about the $c_l$, the $l$th corner of $\Dom$ as
\begin{equation*}
  \mathfrak{S}_{l,m}(r_l,\theta_l) := \cases{r_l^{m\pi/\varphi_l}
  \sin\b(\smfrac{m\pi}{\varphi_l}\theta_l\b)\eta(r_l), 
  &m\pi/\varphi_l\notin\N,\\
  r_l^{m\pi/\varphi_l}\log(r_l)\b[
  \sin\b(\smfrac{m\pi}{\varphi_l}\theta_l\b)+\theta_l\cos\b(\smfrac{m\pi}{\varphi_l}\theta_l\b)\b]\eta(r_l), 
  &m\pi/\varphi_l\in\N.
  }
\end{equation*}
We recall that $\varphi_l$ is the interior angle at $c_l$, and
we set $\eta\in\CC^\infty_0(\R)$ to be a cutoff
function so that $\eta(x) = 1$ for $|x|$ sufficiently small,
and $\supp\{\mathfrak{S}_{l,m}\}\cap\supp\{\mathfrak{S}_{l',n}\}
=\emptyset$ for any $n,m\in\N$ and any $l\neq l'$.

We note that $\bar{y}_i$ only fails to be $\CC^{4,\sigma}$ at the
corners of the domain $\Dom$, and since $\Dom$ is convex,
$\pi/\varphi_l>1$, which implies that $\mathfrak{S}_{l,m}\in\CC^{1,
\delta}(\Dom)$. Hence $\bar{y}_i\in\CC^{1,\delta}
(\Dom)$ with $\delta:=\min_l\{\pi/\varphi_l-1\}\in(0,
\smfrac12]$.
$\mathcal{X}$ is
a Banach space when endowed with the norm
\begin{equation}
  \bg\| v+\hspace{-4mm}\sum_{0<m<(4+\sigma)\varphi_l/\pi} 
  \hspace{-4mm}C_{l,m}\mathfrak{S}_{l,m}\bg\|_\mathcal{X} :=
  \|v\|_{\CC^{4,\sigma}
  (\Dom)}+\hspace{-6mm}\sum_{0<m<(4+\sigma)\varphi_l/\pi}
  \hspace{-6mm} |C_{l,m}|.\label{eq:ybar_norm}
\end{equation}
Furthermore, it can be checked that the mapping
\begin{equation*}
  \mathcal{S}:\b\{x\in\Dom\sep\dist(x,\partial\Dom)\geq\eps/2\b\}
  \to\mathcal{X}\quad\text{where}\quad \mathcal{S}
  (x_i):=\bar{y}_i
\end{equation*}
is continuous, and is hence bounded, since the domain of
$\mathcal{S}$ is compact.

\subsubsection*{Estimating $\df \Gd$ and $\df \bar{y}^n_j$.}
Applying \eqref{eq:corr_bnd_deriv} for any $e_j$ which is not
the closest point to $e$ in the support of $\mu$, we have that
\begin{equation*}
  \bg|\df \Gd(e-\xvec_j,e+\avec-\xvec_j) - 
  \int_0^1 \frac1{\V\pi}\frac{e+t\avec-\xvec_j}
  {|e+t\avec-\xvec_j|^2}\cdot \avec\dt\bg|\lesssim
  \frac{\log|\dist(e-\xvec_j,0)|}
  {\dist(e-\xvec_j,0)^2}=O\b(n^{-2}\log(n)\b),
\end{equation*}
which holds uniformly for $\mu\in\Posen$ since
$\dist(e,e_j)\geq\smfrac12\dist(e_i,e_j)\geq\smfrac12\eps n$.
Furthermore, using the homogeneity and regularity of
$(x,y)\mapsto\frac{x-y}{|x-y|^2}$ to Taylor expand under
the integral, we have
\begin{equation}
  \df \Gd([e-\xvec_j,e+\avec-\xvec_j]) = 
  n^{-1}\frac{1}{\Num\pi}\frac{\smfrac1n e-\xvec_j}
  {|\smfrac1n e-\xvec_j|^2}\cdot \avec
  + O(n^{-2}\log(n)).\label{eq:grad_asymptotics_1}
\end{equation}
Using the representation of $\bar{y}_j\in\mathcal{X}$ and Taylor
expanding, we have
\begin{equation}
  \b|\df \bar{y}^n_j([e,e+\avec]) -
  n^{-1}\nabla \bar{y}_j(\smfrac1n e)
  \cdot \avec\b| \leq 
  n^{-2}\sum_{l=1}^L|\dist(\smfrac1ne,c_l)|^{\pi/\varphi_l-2}\|
  \bar{y}_j\|_\mathcal{X}.\label{eq:grad_asymptotics_2}
\end{equation}

\subsubsection*{Estimating $\df u$.}
We now use Lemma~\ref{th:interior_estimates} to estimate $\df u$.
Defining $g:\Out(\Dom_{n,0})\to\R$ to be
\begin{equation}
  g(e):=\sum_{j=1}^m b_j \Gd(e-
  \xvec_j)+\bar{y}^n_j(e).\label{eq:disc_corr_bc}
\end{equation}
By applying \eqref{eq:corr_bnd_func} and again invoking the 
definition and regularity of $\bar{y}^n_j$ to Taylor expand
near the boundary, we have that
\begin{equation*}
  \|g\|_{\ell^\infty(\Out(\Dom_{n,0}))}\lesssim
  n^{-1}\|\bar{y}\|_\mathcal{X}+O\b(n^{-1}\log(n)\b),
\end{equation*}
where the latter term is uniform in $n$ for fixed $\eps$. 
Lemma~\ref{th:interior_estimates} now implies that
\begin{equation}
  |\df u(e)|\lesssim n^{-1}\log(n)\b|\dist\b(e,\Out(\Dom_{n,0})\b)
  \b|^{-1}.\label{eq:du_est}
\end{equation}

\subsubsection*{Estimating $\Lap\bar{y}^n_j$.}
For the purpose of estimating $\df v$, we first obtain bounds on
$\Lap \bar{y}^n_j$.
Let $e\in \Dom_{n,0}\setminus\Out(\Dom_{n,0})$, and $x\in\Dom$ be
the corresponding vector. We use
the regularity of $\bar{y}_j$ to Taylor expand, obtaining
\begin{align}
  \Lap \bar{y}^n_j(e) &= \sum_{j=1}^m\sum_{i=1}^\Num \int_0^1 \nabla 
  \bar{y}_j(\smfrac1n (x+t\svec_i))\cdot \smfrac1n \svec_i\dt,
  \notag\\
  &=\sum_{j=1}^m\sum_{i=1}^\Num\int_0^1 \smfrac12n^{-3}
  \nabla^3 \bar{y}_j(\smfrac1n x)[\svec_i,\svec_i,\svec_i]+
  \smfrac16n^{-4}(1-t)^3\nabla^4 \bar{y}_j(\smfrac1n (x+t\svec_i))
  [\svec_i,\svec_i,\svec_i,\svec_i]\dt,\label{eq:Taylor_Lap_ybar}
\end{align}
where $\svec_i$ are nearest neighbour directions in the dual 
lattice, and the terms involving $\nabla \bar{y}_j$ and 
$\nabla^2\bar{y}_j$ cancel respectively by lattice symmetry and the
fact that $\bar{y}_j$ is harmonic. If the dual lattice is $\Sq$ or
$\Tr$, then the terms involving $\nabla^3\bar{y}_j$ also cancel,
which entails that
\begin{equation*}
  \b|\Lap\bar{y}^n_j(e)\b| \leq \smfrac 16 n^{-4}\sum_{i=1}^m
  \int_0^1(1-t)^3\b|\nabla^4 \bar{y}_j\b(
  \smfrac1n (x+t\svec_i)\b)\b|\dt.
\end{equation*}
By using the description of $\bar{y}_j$ as a sum of $v\in\CC^{4,
\sigma}(\Dom)$ and $\mathfrak{S}_{j,m}$, it can be seen
that each of the integrands in the estimate above is bounded
any $e\in \Dom_{n,0}$ and $\svec_i$, and moreover
\begin{equation}
  \b|\Lap\bar{y}^n_j(e)\b| \leq \smfrac16\Num n^{-4}
  \|\bar{y}\|_\mathcal{X}\sum_l \b|\dist\b(\smfrac1n e,c_l\b)\b|
  ^{\pi/\varphi_l-4}.\label{eq:TrSq_divergence_est}
\end{equation}

Returning to the case where the dual lattice is $\Hx$, we first
Taylor expand to third--order to obtain that
\begin{equation}
  |\Lap\bar{y}^n_j(e)|\leq \smfrac12 n^{-3}\sum_l \dist(\smfrac1n e,c_l)^{\pi/\varphi_l-3}\|\bar{y}\|_\mathcal{X}.
  \label{eq:Hx_div_est_3}
\end{equation}
Define
\begin{equation}
  A:=\b\{e\in \Int(\Dom_{n,0})\bsep e,e+\evec_1\in \Int(\Dom_{n,0})\b\};
  \label{eq:A_defn}
\end{equation}
then for all $e\in A$, we have
\begin{multline*}
  \bg|\smfrac12n^{-3}\sum_{i=1}^3
  \B(\nabla^3\bar{y}_j(\smfrac1ne)[\avec_{2i},\avec_{2i},\avec_{2i}]
  -\nabla^3\bar{y}_j\b(\smfrac1n(e+\evec_1)\b)[\avec_{2i},\avec_{2i},
  \avec_{2i}]\B)\bg|\\=
   \bg|\smfrac12n^{-4}\sum_{i=1}^3\int_0^1\nabla^4\bar{y}_j
  (\smfrac1n(e+t\evec_1))[\evec_1,\avec_i,\avec_i,\avec_i]\bg|,\\
  \leq \smfrac32n^{-4}\sum_l\dist(\smfrac1n e,c_l)
  ^{\pi/\varphi_l-4}\|\bar{y}\|_\mathcal{X}.
\end{multline*}
Using this estimate, and the argument used above in the case where
the dual lattice was $\Tr$, for any $e\in A$, we deduce that
\begin{equation}
  \b|\Lap\bar{y}^n_j(e)+\Lap\bar{y}^n_j(e+\evec_1)\b|\leq
  2n^{-4}\sum_l \dist(\smfrac1n e,c_l)^{\pi/\varphi_l-4}\|
  \bar{y}\|_\mathcal{X}.\label{eq:Hx_div_est_4}
\end{equation}

\subsubsection*{Estimating $\df v$.}
It remains to bound $\df v$. We proceed by constructing upper and
lower bounds on $v$ by using estimates
\eqref{eq:TrSq_divergence_est}, \eqref{eq:Hx_div_est_3} and
\eqref{eq:Hx_div_est_4}
and the full lattice Green's function.
Recalling the result of Lemma~\ref{th:full_latt_Gfunc}, for any
$\xvec\in\Dom$, we note that
\begin{gather*}
  \Lap\b[\Gd(\cdot-x)+\smfrac1{\V\pi}\log|n\,\diam(\Dom)|\b] = 
  \chr_{e}\quad\text{in }\Int(\Dom_{n,0}),\text{ and}\\
  \Gd(\cdot-x)+\smfrac1{\V\pi}\log|n\,\diam(\Dom)|\geq 0
  \qquad\text{on }\Out(\Dom_{n,0}).
\end{gather*}
Next, we define neighbourhoods of each corner of the domain
\begin{equation*}
  B_{l,\eps}:=\b\{ e\in \Int(\Dom_{n,0})
  \bsep \dist(\smfrac1ne,c_l)\leq \eps\b\}.
\end{equation*}
Recalling that $\delta:=\min\{\smfrac\pi{\varphi_l}-1\}\in(0,
\smfrac12]$, estimate \eqref{eq:TrSq_divergence_est} implies
that
\begin{equation}
  |\Lap\bar{y}^n_j(e)|\lesssim n^{-4}\eps^{\delta-3}
  \|\bar{y}_j\|_\mathcal{X}\quad\text{on } \Int(\Dom_{n,0})\setminus
  \bigcup_l B_{l,\eps}.\label{eq:dom_sep_div_ests}
\end{equation}
We now define
\begin{equation*}
  v^\pm(e):= -\bg[\sum_{e'\in \Dom_{n,0}}\Lap \bar{y}^n_j(e')
  \,\Gd(e-e')\bg]\pm C_n,
\end{equation*}
where $C_n$ is a small constant depending upon $n$ that we will 
choose later.
We note that $\Lap [v-v^\pm]=0$, so choosing $C_n$ such that
$v^+\geq0$ and $v^-\leq0$ on $\Out(\Dom_{n,0})$, 
Lemma~\ref{th:maxmin_prin} would imply that
\begin{equation*}
  v^-(e)\leq v(e) \leq v^+(e)\quad\text{for all }e\in
  \Int(\Dom_{n,0}).
\end{equation*}
When the dual lattice is either $\Tr$ or $\Sq$, applying estimate 
\eqref{eq:dom_sep_div_ests}, and summing,
\begin{equation*}
  |v^\pm(e)|\lesssim \|\bar{y}\|_\mathcal{X}\bg[
  \sum_{\substack{e'\in \Int(\Dom_{n,0})\\e'\notin \bigcup_l B_{l,\eps}}}n^{-4}\eps^{\delta-3}\b|\Gd\b(e-e'\b)\b|
  +n^{-1-\delta}\sum_{e\in \bigcup_lB_{l,\eps}}
  \dist(e',nc_l)^{\delta-3}\b|\Gd\b(e-e'\b)\b|\bg]+ C_nn^2.
\end{equation*}
Treating each sum separately, we see that
\begin{gather*}
  \sum_{\substack{e'\in \Int(\Dom_{n,0})\\e'\notin \bigcup_l B_{l,\eps}}}\b|\Gd\b(e-e'\b)\b|
  \lesssim \bg[\sum_{\substack{e'\in \Int(\Dom_{n,0})\\e'\notin \bigcup_l B_{l,\eps}}}
  \log\b|\dist(e,e')+1\b|\bg]\lesssim n^2\log(n),\\
  \sum_{e'\in \bigcup_l B_{l,\eps}}|\dist(e',nc_l)|^{\delta-3}\b|\Gd\b(e-e'\b)\b|
  \lesssim \log(n)\sum_{e\in \bigcup_l B_{l,r}}
  |\dist(e',nc_l)|^{\delta-3}\lesssim \log(n),
\end{gather*}
recalling that statement (3) of Theorem~\ref{th:full_latt_Gfunc}
implies that $|\Gd(e)|\lesssim \log|\dist(e,0)|$,
$\diam(n\Dom)=O(n)$, and the sum on the second line converges since
$\delta\leq \smfrac12<1$.

These estimates imply that
\begin{equation*}
  |v^\pm(e)|\lesssim \|\bar{y}\|_\mathcal{X} 
  n^{-1-\delta}\log(n)+C_n n^2,
\end{equation*}
so choosing $C_n=O(n^{-3-\delta}\log(n))$ gives
\begin{equation}
  |v(e)|=O\b(n^{-1-\delta}\log(n)\b),\quad\text{and hence}\quad
  |\df v(e)|=O\b(n^{-1-\delta}\log(n)\b)\label{eq:dv_est}
\end{equation}
for all $e\in \Dom_{n,1}$.

When the dual lattice is $\Hx$, recall the definition of $A$ from
\eqref{eq:A_defn}, and set
\begin{equation*}
  A':= \{ e\in \Int(\Dom_{n,0})\sep e-\evec_1\notin \Int(\Dom_{n,0})\}.
\end{equation*}
For any $e'\in A$ let $\xvec'\in\Dom$ be the corresponding vector.
We apply \eqref{eq:Hx_div_est_3}, 
\eqref{eq:Hx_div_est_4}, and the conclusions of 
Theorem~\ref{th:full_latt_Gfunc} to deduce that
\begin{multline*}
  \b|G^\Hx(e-\xvec')\Lap\bar{y}^n_j(e)
  +G^\Hx(e+\evec_1-\xvec')\Lap\bar{y}^n_j(e+\evec_1)
  \b|\\\leq |G^\Hx(e-\xvec')| n^{-4}|\dist(\smfrac1ne,c_l)|^{\pi/\varphi_l-4}
  +|\df G^\Hx(e-\xvec',e+\evec_1-\xvec')| n^{-3}
  |\dist(\smfrac1ne,c_l)|^{\pi/\varphi_l-3}\\
  \leq \log|\dist(e,\xvec')| 
  n^{-4}|\dist(\smfrac1ne,c_l)|^{\pi/\varphi_l-4}
  +\frac{\log|\dist(e,\xvec')|}{\dist(e,\xvec')} n^{-3}
  |\dist(\smfrac1ne,c_l)|^{\pi/\varphi_l-3}.
\end{multline*}
By summing over $e'\in A$, we obtain
\begin{equation}  
  \sum_{e'\in A}\b|G^\Hx(e-\xvec')\Lap\bar{y}^n_j(e')+G^\Hx(e+\evec_1-\xvec')
  \Lap\bar{y}^n_j(e'+\evec_1)\b|\lesssim \|\bar{y}_j\|_\mathcal{X}
  \b[\log(n)n^{-2}\eps^{\delta-3}+n^{-1-\delta}\log(n)\b].
  \label{eq:Hx_est1}
\end{equation}
Next, we sum \eqref{eq:Hx_div_est_3} over $A'$, noting that
$\#A' = O(n)$, to obtain
\begin{align}
  \sum_{e'\in A'}|\Gd(e-\xvec')| |\Lap \bar{y}^n_j(e')|&\lesssim
  \|\bar{y}_j\|_\mathcal{X}\log(n)\bg[\sum_{\substack{e'\in A'\\ e'\notin
  \bigcup_l B_{l,\eps}}}n^{-3}\eps^{\delta-2}+
  n^{-1-\delta}\hspace{-3mm}\sum_{\substack{e'\in A'\\ e'\in
  \bigcup_l B_{l,\eps}}}\hspace{-3mm}\dist(e',nc_l)^{\delta-2}\bg],\notag\\
  &\lesssim \|\bar{y}_j\|_\mathcal{X} \b[\eps^{\delta-2}
  \log(n)n^{-2}+n^{-1-\delta}\log(n)\b]\label{eq:Hx_est2}.
\end{align}
Putting \eqref{eq:Hx_est1} and \eqref{eq:Hx_est2} together, and
applying similar arguments to that made for the other cases above, 
we deduce that \eqref{eq:dv_est} also holds in the case where the 
dual lattice is $\Hx$.

\subsubsection*{Conclusion.}
Combining \eqref{eq:grad_asymptotics_1}, 
\eqref{eq:grad_asymptotics_2}, \eqref{eq:du_est} and 
\eqref{eq:dv_est} and noting that
\begin{equation*}
  \nabla\bar{y}_j(x)
  +\frac{1}{\V\pi}\frac{x-\xvec_j}
  {|x-\xvec_j|^2} = \nabla \Gc_{\xvec_j}(x),
\end{equation*}
we have proved Theorem~\ref{th:interior_asymptotics}.
\medskip

\noindent
Theorem~\ref{th:interior_asymptotics} implies the following
corollary.

\begin{corollary}
\label{th:uniform_deriv_bound}
Given $\eps>0$ and a convex lattice polygon $\Dom\subset\R^2$,
for all $n$ sufficiently large,
\begin{equation*}
  \sup_{e\in \Dom^*_{n,1}}|\df^* G_{\mu^*}(e)|<\smfrac12
  \quad\text{for any }\mu\in \Posen.
\end{equation*}
\end{corollary}

\begin{proof}
Let $e^*\in \Dom^*_{n,1}$, and let
$e_i^* \in \argmin\b\{\dist(e^*,e_i^*)\bsep e_i^*\in\supp\{\mu\}\b\}$.
Applying Theorem~\ref{th:interior_asymptotics}, and splitting
$\Gc_{\xvec_i}(x)=\smfrac1{\V\pi}\log(|x-\xvec_i|)
+\bar{y}_i(x)$, we obtain the estimate
\begin{equation*}
  |\df G_{\mu^*}(e)| \leq \sup_{e\in\L^*_1}|\df^* \Gdst(e)|+
  n^{-1}\B[\smfrac{(m-1)\eps}{\V\pi} +\sum_{i=1}^m\|\bar{y}_i\|_\mathcal{X}\B]
  +O\b(n^{-1-\delta}\log(n)\b),
\end{equation*}
where we recall the definition of the norm
$\|\,.\,\|_\mathcal{X}$ from \eqref{eq:ybar_norm}.
Further, from \S\ref{sec:ybar_regularity} we have that
$\|\bar{y}_i\|_\mathcal{X}$ is uniformly bounded
for $x_i\in\b\{x\in\Dom\bsep \dist(x,\partial\Dom)\geq 
\eps\b\}$, and so applying statement (2) of
Theorem~\ref{th:full_latt_Gfunc}, we have the result.
\end{proof}

\subsection{Proof of Theorem~\ref{th:equivalence}}
We now complete the proof of Theorem~\ref{th:equivalence} using the
results above.
Our first step is to verify the necessity of the equilibrium conditions 
given in \eqref{eq:equilibrium_conds}.

Let $u$ be a locally stable equilibrium containing the 
dislocation configuration $\mu\in\Posen$.
By inspecting the proof of Lemma~5.1 in \cite{ADLGP14}, it
follows that if $\df u(e)\in\smfrac12+\Z$ for some
$e\in \Dom_{n,1}$, then there exist lower energy states 
arbitrarily close to $u$, and so any
$\alpha\in[\df u]$ has $\|\alpha\|_\infty<\smfrac12$. 
By definition, we have that $\df \alpha =\mu$.
Finally, let $v\in\Wsc(\Dom_{n,0})$; then for $t$ 
sufficiently small, $\|\alpha+t\df v\|_\infty<\smfrac12$,
hence
\begin{equation*}
  E_n(u+tv;u) = \int_{\Dom_{n,1}} \psi(\alpha+t\df v)-\psi(\alpha)
  =\int_{\Dom_{n,1}} \lambda t\,\alpha\,\df v+ \smfrac12\lambda
  t^2|\df v|^2.
\end{equation*}
It follows that $(\alpha,\df v) = 0$ for any 
$v\in\Wsc(\Dom_{n,0})$, hence $\codf \alpha=0$.

Next, we show that if $\alpha$ satisfies the equilibrium 
conditions \eqref{eq:equilibrium_conds}, then it is unique.
Suppose that $\alpha$ and $\alpha'$ satisfy 
\eqref{eq:equilibrium_conds}. We define $\beta =
\alpha-\alpha'$, and note that $\beta^*\in\Wsc_0(\Dom_{n,1}^*)$
satisfies $\df^*\beta^*=0$ and
$\codf^*\beta^*=0$. Since $n\Dom^*$ is simply connected, the former
condition implies that $\beta^*=\df^* w$ for some
$w\in\Wsc_0(\Dom^*_{n,0})$, which must satisfy $\Lap^*w=0$: by 
the uniqueness of the solution proved in
Lemma~\ref{th:Pois_BVP_existence}, it follows that $w=0$, hence
$\beta=0$, and thus $\alpha=\alpha'$.

Since $*$ is a bijection between $\Wsc(\Dom_{n,1})$ and 
$\Wsc_0(\Dom^*_{n,1})$, there exists $\alpha\in\Wsc(\Dom_{n,1})$
such that $\alpha^* = \df^* G_{\mu^*}$. Furthermore, by using
\eqref{eq:duality_Dom}, we have that
\begin{gather*}
  \df \alpha(e) = \codf^* \df^* G_{\mu^*}(e^*) = \mu^*(e^*)
  =\mu(e),\quad\text{for }e\in \Dom_2,\\
  \text{and}\quad\codf \alpha(e) =(\df^*)^2 G_{\mu^*}(e^*) = 0
  \quad\text{for }e\in \Dom_{n,0}.
\end{gather*}
Finally, we note that
$\|\alpha\|_\infty = \|\df G_{\mu^*}\|_\infty$, hence applying
Corollary~\ref{th:uniform_deriv_bound}, it follows that 
$\alpha$ satisfies \eqref{eq:equilibrium_conds} if $n$ is sufficiently
large.

To demonstrate that $\alpha\in[\df u_\mu]$ for some 
$u_\mu\in\Wsc(\Dom_{n,0})$, fix $e'\in \Dom_{n,0}$, and define
$u_\mu(e')=0$.
Using the fact that $\Dom_n$ is path--connected, let $\gamma^e$ be
the path such that $\partial \gamma^e =e'\cup -e$, and define
$u_\mu(e):=\int_{\gamma^e} \alpha$.
Letting $b=[e_0,e_1]\in \Dom_{n,1}$, we find that
\begin{equation*}
  \df u(b) = \int_{\gamma^{e_1}}\alpha - \int_{\gamma^{e_0}}\alpha,
   = \int_{\gamma^{e_1}\cup -\gamma^{e_0} \cup-b}\hspace{-3mm}\alpha + \int_b\alpha.
\end{equation*}
Noting that $\partial(\gamma^{e_1}\cup -\gamma^{e_0}\cup-b) = \emptyset$, we
apply the fact that $\Dom_n$ is simply connected to assert that
$\gamma^{e_1}-\gamma^{e_0}-b = \partial A$, for some $A\in \Dom_2$, hence
\begin{equation*}
  \df u(b) = \alpha(b)+\int_{\partial A}\alpha = \alpha(b)+
  \int_A \mu \in \alpha(b)+\Z.
\end{equation*}
It follows that $\alpha\in[\df u_\mu]$. To prove that $u_\mu$ is 
unique up to the equivalence \eqref{eq:equivalence_defn}, we note 
that if $\alpha\in[\df u]$ and $\alpha\in[\df v]$, then by 
the definition of a bond--length 1--form (see 
\S\ref{sec:disl_configs}), it follows that
\begin{equation*}
  \df u(e) = \df v(e)+Z(e)\quad\text{for all }e\in \Dom_{n,1},
  \quad\text{with}\quad Z:\Dom_{n,1}\to\Z.
\end{equation*}
Moreover, $\df Z =0$, so $Z=\df z$, and it is straightforward to
check that $z:\Dom_{n,0}\to H+\Z$ for some $H\in\R$, completing
the proof of Theorem~\ref{th:equivalence}.

\section{Proof of Theorem~\ref{th:energy_barriers}}
\label{sec:barrier_proofs}
This section is devoted to the proof of
Theorem~\ref{th:energy_barriers}, and we proceed in several steps.
We first demonstrate that there exists $u$ which `solves' the
min--max problem used to define $\Tran_n(\mu\to\nu)$ via a compactness
method. We then identify necessary conditions for such a solution, and
show that these necessary conditions identify a pair of bond--length
1--forms. The required bond--length 1--forms are then constructed
via duality using an interpolation of dual Green's
functions, and we verify that the necessary conditions are satisfied
to conclude.

\subsection{The min--max problem}
\alth{To establish existence of a solution, we transform the problem via
taking the quotient of the space of deformations with respect to the
equivalence relation defined in
\eqref{eq:equivalence_defn}; in other words, we identify deformations
`up to lattice symmetries'. This space turns out to be compact, hence
the existence of a critical point follows directly by a compactness
argument.}

\subsubsection{Quotient space}
Recall from \eqref{eq:equivalence_defn} that $\sim$ is the
equivalence relation on $u\in\Wsc(\Dom_{n,0})$
\begin{equation*}
  u\sim v \quad\text{whenever}\quad u = v+z+C\quad \text{for 
  some }z:\Dom_{n,0}\to\Z\text{ and some }C\in\R.
\end{equation*}
Define the quotient space $\Quot:= \Wsc(\Dom_{n,0})/_\sim$ of 
equivalence classes $\llb u\rrb$; we claim that this is a metric
space when endowed with the metric
\begin{equation*}
  d_{\Quot}(\llb u\rrb,\llb v\rrb)=\|\alpha\|_2,
  \qquad\text{where
  }\alpha\in[\df u-\df v],\quad\text{for any }u\in \llb u\rrb
  \text{ and }v\in \llb v\rrb.
\end{equation*}
If $u\sim v$, then $\df u \in \df v+\Z$, and hence
$[\df u]=[\df v]$. Symmetry is
immediate, and $0\in[\df u-\df v]$ implies that $u-v\sim 0$, 
hence $d_\Quot(u,v) = 0$ implies that $u\sim v$. Finally,
for the triangle inequality, by checking cases it may be shown
that
\begin{equation*}
  \beta\in[\df u],\,\beta'\in[\df v]\text{ and }
  \alpha\in[\df u+\df v]\quad\text{imply that}\quad
  |\alpha(e)|\leq |\beta(e)|+|\beta'(e)|\text{ for all }e\in 
  \Dom_{n,1}.
\end{equation*}
The triangle inequality follows, and hence the metric is well--defined.
Moreover, the space is complete and totally bounded, so the
Heine--Borel theorem applies, and $\Quot$ is compact.
We recall that the mapping $u\mapsto\llb u\rrb$ is the natural
embedding of $\Wsc(\Dom_{n,0})$ in $\Quot$.

\subsubsection{Redefining the energy}
As noted in \S\ref{sec:equivalence}, for any $u,u',v\in\Wsc(\Dom_{n,0})$
such that $u\sim u'$, $E_n(u,v)=E_n(u',v)$. It follows that the mapping 
$\widetilde{E}_n:\Quot\to\R$,
\begin{equation*}
  \widetilde{E}_n(\llb u\rrb):=E_n(u,v)\quad\text{for some } u\in\llb 
  u\rrb
\end{equation*}
is well--defined.
Suppose that $u\in\llb u\rrb$, and $u'\in\llb u'\rrb$, and
that $\alpha\in[\df u-\df u']$. Then
\begin{equation*}
  \b|\widetilde{E}_n(\llb u\rrb)-\widetilde{E}_n(\llb u'\rrb)\b|
  =\bg|\int_{\Dom_{n,1}}\psi(\df u'+\alpha)-\psi(\df u')\bg|
  \lesssim C\int_{\Dom_{n,1}} |\alpha|\lesssim \|\alpha\|_2=
  d_\Quot\b(\llb u\rrb,\llb u'\rrb\b),
\end{equation*}
where we use the fact that $\psi$ is uniformly Lipschitz, and
then apply the Cauchy--Schwarz inequality. It follows that
$\widetilde{E}_n$ is uniformly Lipschitz on $\Quot$.

\subsubsection{Space of continuous paths}
Define the metric space $\CC([0,1];\Quot)$ of 
continuous functions from $[0,1]$ to $\Quot$, with the
usual metric
\begin{equation*}
  d^\infty_\Quot(\gamma,\gamma'):= \sup_{t\in[0,1]} 
  d_\Quot\b(\gamma(t),\gamma'(t)\b).
\end{equation*}
The mapping $\gamma\mapsto\max_{t\in[0,1]} \widetilde{E}_n(\gamma(t))$ is continuous with respect to this metric, since
$\widetilde{E}_n$ is uniformly continuous on $\Quot$.

We suppose that $n$ is large enough such that the conclusion
of Theorem~\ref{th:equivalence} holds, and write
$\llb u_\mu\rrb$ to mean the equivalence class containing
$u_\mu$, which is the set of all locally stable equilibria
corresponding to the
dislocation positions $\mu\in\Posen$. Define the sets of paths
\begin{equation*}
  \widetilde{\Gamma}_n(\mu\to\nu):=\b\{ \gamma\in\CC([0,1];
  \Quot )\bsep \gamma(0) = \llb u_\mu\rrb, \gamma(1) = 
  \llb u^{\nu}\rrb, \alpha\in[\df \gamma(t)]\text{ has }
  \df\alpha \in\{\mu,\nu\},\forall t\in[0,1]\b\};
\end{equation*}
this should be thought of as the set of paths through phase 
space which move dislocations from $\mu$ to $\nu$ without
visiting any intermediate states.

\subsubsection{Existence}
We recall that the energy barrier was defined to be
\begin{equation*}
  \Tran_n(\mu\to\nu) = \inf_{\gamma\in\Gamma_n(\mu,\nu)}
  \sup_{t\in[0,1]} E_n(\gamma(t);u_\mu).
\end{equation*}
The following lemma now demonstrates the existence of a transition
state.

\begin{lemma}
\label{th:existence_cp}
If $n$ is sufficiently large, for any $\mu,
\nu\in\Posen$ such that $\Gamma_n(\mu\to\nu)$
is non--empty, there exists $u_\uparrow\in\Wsc(\Dom_{n,0})$ such that
\begin{equation*}
  E_n(u_\uparrow;u_\mu) = \Tran_n(\mu\to\nu).
\end{equation*}
\end{lemma}

We will call $u_\uparrow$ a \emph{transition state} for the 
transition from $\mu$ to $\nu$.

\begin{proof}
We first note that since
$\Quot$ is compact, $\CC([0,1];\Quot)$ is compact.
By assumption, $\Gamma_n(\mu\to\nu)$ is non--empty, and so
the space $\widetilde{\Gamma}_n(\mu\to\nu)$ is non--empty
by applying the natural embedding
$\gamma(t)\mapsto\llb\gamma(t)\rrb$. Moreover, we have that
\begin{equation*}
  \max_{t\in[0,1]} E_n(\gamma(t);u_\mu) = 
  \max_{t\in[0,1]}\widetilde{E}_n\b(\llb \gamma(t)\rrb\b).
\end{equation*}
Since $\tilde{\gamma}\mapsto\max_{t\in[0,1]}\widetilde{E}_n
(\tilde{\gamma}(t))$ is continuous, there exists a minimiser
\begin{equation*}
  \tilde{\gamma}\in\argmin\B\{ \max_{t\in[0,1]}\widetilde{E}_n(\tilde{\gamma}(t)) \Bsep 
  \tilde{\gamma}\in\overline{\widetilde{\Gamma}_n(\mu\to\nu)}\,\B\},
\end{equation*}
where $\overline{\widetilde{\Gamma}_n(\mu\to\nu)}$ denotes the
closure of 
$\widetilde\Gamma_n(\mu\to\nu)$ in $\Quot$. As
$t\mapsto \widetilde{E}_n\b(\tilde{\gamma}(t)\b)$ is also continuous,
it follows that there exists $u_\uparrow\in\llb u_\uparrow\rrb$ with
$\llb u_\uparrow\rrb =\gamma(t^*)\in\overline{\Gamma_n(\mu\to\nu)}$ for
some $t^*\in[0,1]$, which satisfies
\begin{equation*}
  E_n(u_\uparrow;u_\mu) = \min_{\tilde{\gamma}\in\widetilde\Gamma_n(\mu\to\nu)}
  \max_{t\in[0,1]}\widetilde{E}_n\b(\tilde{\gamma}(t)\b)
  =\min_{\gamma\in\Gamma_n(\mu\to\nu)}
  \max_{t\in[0,1]}E\b(\gamma(t);u_\mu\b).\qedhere
\end{equation*}
\end{proof}

\subsection{Necessary conditions}
We now identify necessary conditions on the transition states
identified in Lemma~\ref{th:existence_cp}. We remark that the proof of
the following lemma relies crucially on the particular choice of
potential $\psi$.

\begin{lemma}
\label{th:necessary_conds_cp}
Suppose that $u_\uparrow\in\Wsc(\Dom_{n,0})$ is a transition state for the
transition from $\mu$ to $\nu$, where
$\nu-\mu = b_i[\chr_q-\chr_p]$ and $q^*=p^*+\avec$ for some nearest--neighbour direction in the dual lattice, $\avec$.
Then $u_\uparrow\in\b\{ u\in\Wsc(\Dom_{n,0})\bsep 
\alpha\in[\df u]$ has $\alpha(l)=\pm\smfrac12\b\}$, where 
$l^*=[p^*,q^*]$, and moreover there exist exactly two 
$\alpha_\uparrow,\alpha_\downarrow\in[\df u_\uparrow]$, satisfying
\begin{enumerate}
  \item $\df \alpha_\uparrow=\mu$, $\df\alpha_\downarrow=\nu$,
  \item $\codf\alpha_\uparrow(a)=\codf\alpha_\downarrow(a)=0$ 
    for all $a\pm\notin\partial l$,
  \item $\codf\alpha_\uparrow(e_0)+\codf\alpha_\uparrow(e_1)=0$
   and $\codf\alpha_\downarrow(e_0)+\codf\alpha_\downarrow(e_1)=0$
    for $e_0$ and $e_1$ such that $l=[e_0,e_1]$, and
  \item $-\alpha_\uparrow(l)=\alpha_\downarrow(l)
    =\smfrac12b_i$.
\end{enumerate}
\end{lemma}

\begin{proof}
We begin by proving that all transition states lie in the
set
\begin{equation*}
  B:=\b\{ u\in\Wsc(\Dom_{n,0})\bsep \alpha\in[\df u]\text{ has }
  \alpha(e)=\pm\smfrac12\text{ for some }e\in \Dom_{n,1}\b\}.
\end{equation*}
We remark that any $\gamma\in\Gamma_n(\mu\to\nu)$ must pass 
through $B$, since it is only on this set that we may have
$\alpha,\alpha'\in[\df \gamma(t)]$ with
\begin{equation*}
  \df\alpha(p) = b_i,\quad \df\alpha(q) = 0,
  \quad\text{and}\quad
  \df\alpha'(p) = 0, \quad\df\alpha'(q) = b_i.
\end{equation*}
Suppose that $\gamma\in\overline{\Gamma_n(\mu\to\nu)}$ solves
the minimisation problem \eqref{eq:energy_barrier_defn}, and 
attains a transition state $u_\uparrow=\gamma(t^*)$ at $t=t^*$. Suppose
further that $u_\uparrow\notin B$.

Taking an interval with $t^*\in[t_1,t_2]$ 
such that $\gamma(t)\notin B$ for all $t\in[t_1,t_2]$, and
$\gamma(t_1),\gamma(t_2)\neq \gamma(t^*)$, we define
\begin{equation*}
  \beta(t):=\cases{\gamma(t) &t\notin[t_1,t_2],\\
  \smfrac{t_2-t}{t_2-t_1}\gamma(t_1)
  +\smfrac{t-t_1}{t_2-t_1}\gamma(t_2) &t\in[t_1,t_2].
  }
\end{equation*}
This is a valid competitor for the minimum problem, and 
moreover by using strict convexity of $\psi(x)$ for $x\in[n-\smfrac12+,n+\smfrac12]$ for any $n\in\Z$, we obtain
\begin{equation*}
  E_n(\gamma(t^*);u_\mu)\leq 
  \sup_{t\in(t_1,t_2)} E_n(\beta(t);u_\mu)< 
  \max\b\{ E_n(\gamma(t_1);u_\mu),E_n(\gamma(t_2);
  u_\mu)\b\} \leq E_n(\gamma(t^*);u_\mu),
\end{equation*}
which is a contradiction.

Suppose once more that $\gamma$ is a minimal path, and 
$\max_{t\in[0,1]}E[\gamma(t);u_\mu]$ attaining a transition state
at $t=t^*$. Suppose also that $\alpha\in[\df 
\gamma(t^*)]$ has $\alpha(e)=\pm\smfrac12$ for some $e\neq\pm l$.
Let $a\in\partial e$ such that $e\notin \pm\partial l$. Then
by considering $\gamma(t^*)+s\chr_a$, and following
the strategy of proof of Lemma~5.1 in \cite{ADLGP14}, it may be
checked that there exists $\delta>0$ such that for all $s\in[0,
\delta)$ or for all $s\in(-\delta,0]$,
\begin{enumerate}
  \item $\alpha\in[\df(\gamma(t^*)+s\chr_a)]$ satisfies
    $\df\alpha\in\{\mu,\nu\}$, and
  \item $E[\gamma(t^*)+s\chr_a;\gamma(t^*)]<0$ if
    $s\neq 0$.
\end{enumerate}
By redefining $\gamma$ to pass through 
$\gamma(t^*)+s\chr_a$ in a neighbourhood of $t^*$, it 
follows that $\gamma(t^*)$ cannot be a transition state, and hence 
if $u$ is a transition state with $\alpha\in[\df u]$,
$\alpha(e)=\pm\smfrac12$ if and only if $e=\pm l$.

By considering paths $\beta$ which have $\beta(t^*) = 
\gamma(t^*)+s\chr_a$ with $a\notin\partial l$, we obtain that
\begin{equation*}
  E_n(\gamma(t^*);u_\mu)\leq E_n(\gamma(t^*)+s\chr_a;u_\mu)
\end{equation*}
for all $s$ sufficiently small. If $\alpha\in[\df \gamma(t^*)]$, we 
have that
\begin{equation*}
   \int_{\Dom_{n,1}}  \alpha\,\df \chr_a=\codf\alpha(a)=0.
\end{equation*}
By considering $\gamma(t^*)+s[\chr_{e_0}+\chr_{e_1}]$, where
$l = [e_0,e_1]$, we obtain that
\begin{equation*}
  \codf \alpha(e_0)+\codf\alpha(e_1) = 0,
\end{equation*}
hence we have proved that a transition state must satisfy
conditions (1)--(4).

Next, we prove that $\df \alpha_\uparrow=\mu$, 
$\alpha_\uparrow(l)=\smfrac12$ and conditions (2) and 
(3) define a unique 1--form, which is an elastic strain at the
transition state.
Suppose that $\alpha_\uparrow$ and $\alpha'_\uparrow$ satisfy these conditions.
Defining $\theta:=\alpha_\uparrow-\alpha'_\uparrow$, we have
that $\df\theta=0$, hence $\theta=\df v$ for some 0--form $v$
since $\Dom_n$ is simply connected.
Furthermore, $\df v(l)=\theta(l) =0$, $\Lap v(b)=s$ and $\Lap v(c)=-s$
for some $s\in\R$. Then we have
\begin{equation*}
  (\theta,\theta) = (\df v,\df v) = (\Lap v,v) = s[v(b)-v(c)]
  s\,\df v(l) = 0,
\end{equation*}
implying that $\theta=0$, and hence $\alpha_\uparrow$ is unique.
It may be similarly verified that $\alpha_\downarrow$ exists
and is unique, completing the proof.
\end{proof}

\subsection{Construction of the transition state}
In Theorem~\ref{th:equivalence}, we found that the bond
length 1--forms corresponding to local equilibria containing
dislocations are related to dual Green's functions.
By considering this relationship, it is natural to consider
strains dual to interpolations of these Green's functions as 
possible candidates for the transition state $u$. We therefore
define $G^t:=(1-t)G_{\mu^*}+t G_{\nu^*}$, where $t\in[0,1]$, 
$G_{\mu^*}$ and $G_{\nu^*}$ solve \eqref{eq:latt_grn_func_Dom}. 
We note that 
\alth{for any $e^*\in \Dom_{n,0}^*$,
\begin{gather*}
  \Lap^*G^t(e^*)=(1-t)\Lap^* G_{\mu^*}(e^*)+t\Lap^*G_{\nu^*}(e^*)
  =(1-t)\mu^*(e^*)+t\nu^*(e^*);\\
  \text{and in particular,}\quad \Lap^*G^t(p^*)= b_j(1-t)\quad\text{and}
  \quad \Lap^*G^t(q^*)=b_jt.
\end{gather*}}
As in Lemma~\ref{th:necessary_conds_cp}, set
$l\in \Dom_{n,1}$ with $l^* = [p^*,q^*]$.
Since Lemma~\ref{th:necessary_conds_cp} entails that the the transition
state must have $\alpha_\downarrow(l) = \smfrac12 b_i$, we choose
$t\in[0,1]$ such that
\begin{alignat}{3}
  \Lap^*G^t(p^*) +\df^* G^t(l^*)&=\smfrac12b_j,
  &\quad\Leftrightarrow\quad&&
   (1-t)b_j +\df^* G^t(l^*)&=\smfrac12b_j,\\
   \text{and}\quad\Lap^*G^t(q^*) -\df^* G^t(l^*)&=\smfrac12b_j,
   &\quad\Leftrightarrow\quad&& tb_j -\df^* G^t(l^*)
   &=\smfrac12b_j.\label{eq:gradGt(l)}
\end{alignat}
Solving, we find that
\begin{equation}
  t =\frac{\smfrac12b_j+\df^* G_{\mu^*}(l^*)}
  {b_j+\df^* G_{\mu^*}(l^*)-\df^* G_{\nu^*}(l^*)}.
  \label{eq:t_soln}
\end{equation}
Noting that $\df^* G_{\mu^*}(l^*) = \df^*G_{\L^*}([0,0+\avec])+o(1)=
\smfrac1{\Num^*} b_j+o(1)$ and similarly,
$\df^* G_{\nu^*}(l^*) = -\smfrac1{\Num^*} b_j+o(1)$, as 
$n\to\infty$ by applying Theorem~\ref{th:interior_asymptotics}
and statement (2) of Lemma~\ref{th:full_latt_Gfunc}, we see that
$t\in[0,1]$; indeed, $t\to\smfrac12$ as $n\to\infty$.

We now define $\alpha_\uparrow$ and $\alpha_\downarrow$ via
\begin{equation*}
  \alpha^*_\uparrow(e^*) := \cases{
    \df^* G^t(e^*) & e^*\neq \pm l^*,\\
    \mp\smfrac12b_j & e^*= \pm l^*,}
    \quad\text{and}\quad
    \alpha^*_\downarrow(e^*) := \cases{
    \df^* G^t(e^*) & e^*\neq \pm l^*,\\
    \mp\smfrac12b_j & e^*= \pm l^*,}
\end{equation*}
where $t$ is given by \eqref{eq:t_soln}. Letting $\alpha_\mu\in[\df u_\mu]$, for any $e\in \Dom_2$
with $e\neq p,q$, by duality we have
\begin{equation*}
  \df[\alpha_\uparrow-\alpha_\mu](e) =\Lap^* G^t(e^*)
  -\Lap^*G_{\mu^*}(e^*)= t \b[\Lap^* G_{\nu^*}(e^*)-\Lap^* G_{\mu^*}(e^*)\b]
  = 0.
\end{equation*}
Again, by duality we also have
\begin{align*}
  \df[\alpha_\uparrow-\alpha_\mu](p) &= \Lap^* 
  G^t(p^*)+\df^*G^t(l^*) +\smfrac12 b_j -\Lap^*G_{\mu^*}(p^*) = 0,\\
  \df[\alpha_\uparrow-\alpha_\mu](q) &= \Lap^* 
  G^t(q^*)-\df^*G^t(l^*) -\smfrac12 b_j  -\Lap^*G_{\mu^*}(p^*) = 0.
\end{align*}
Similarly, $\df[\alpha_\downarrow-\alpha_\nu]=0$.
It follows therefore that there exist $v_\uparrow$ and 
$v_\downarrow$ such that $\alpha_\uparrow\in[\df u_\mu+
\df v_\uparrow]$, and $\alpha_\uparrow\in[\df u_\mu+
\df v_\downarrow]$.

We also note that if $a\notin\pm\partial l$ and $l = [e_0,e_1]$, then
\begin{align*}
  \codf [\alpha_\uparrow-\alpha_\mu](a) &= (\df^*)^2(G^t-G_{\mu^*})
  (a^*) = 0,\\
  \codf [\alpha_\uparrow-\alpha_\mu](e_0)
  +\codf[\alpha_\uparrow-\alpha_\mu](e_1)
  &=(\df^*)^2[G^t-G_{\mu^*}](e_0^*)+(\df^*)^2[G^t-G_{\mu^*}](e_1^*) = 0.
\end{align*}
It follows that $\alpha_\uparrow$ and $\alpha_\downarrow$ 
satisfy conditions (1)--(4) of 
Lemma~\ref{th:necessary_conds_cp}, and hence we have 
constructed the bond--length one forms corresponding to the
transition state.

Finally, we define $\gamma\in\Gamma_n(\mu\to\nu)$ via
\begin{equation*}
  \gamma(t):=\cases{ u_\mu+2t v_\uparrow &t\in[0,\smfrac12],\\
    u_\mu+v_\uparrow+(2t-1)v_\downarrow &t\in(\smfrac12,1],}
\end{equation*}
which demonstrates that $\Gamma_n(\mu\to\nu)$ is non--empty,
and hence $\Tran_n(\mu\to\nu)$ exists.

\subsection{Proof of  Theorem~\ref{th:energy_barriers}}
We now use the dual representation of $\alpha_\uparrow$,
$\alpha_\downarrow$, $\alpha_\mu$ and $\alpha_\nu$ to give an asymptotic
expression for $\Tran_n(\mu\to\nu)$ as $n\to\infty$. We use duality to
compute
\alth{
\begin{align}
  \Tran_n(\mu\to\nu) &= E_n(u_\mu+u_\uparrow;u_\mu),\notag\\
  &=\smfrac12\lambda \b[(\alpha_\uparrow,\alpha_\uparrow)
  -(\alpha_\mu,\alpha_\mu)\b],\notag\\
  &=\smfrac12\lambda\b[ (\df^* G^t,\df^* G^t)-(\df^* G_{\mu^*},
    \df^* G_{\mu^*})-\df^* G^t(l^*)^2+\smfrac14\b],\notag\\
  &=\smfrac12\lambda\b[ 2t(\df^* G_{\nu^*}-\df^* G_{\mu^*},\df^* G_{\mu^*})
  +t^2(\df^* G_{\nu^*}-\df^* G_{\mu^*},\df^* G_{\nu^*}-
  \df^* G_{\mu^*})-\df^*G^t(l^*)^2+\smfrac14\b],\notag\\
  &=\smfrac12\lambda \b[2t(\Lap^* G_{\nu^*}-\Lap^* G_{\mu^*},G_{\mu^*})
  +t^2(\Lap^* G_{\nu^*}-\Lap^* G_{\mu^*},G_{\nu^*}-
  G_{\mu^*})-\df^*G^t(l^*)^2+\smfrac14\b],\notag\\
  & = \smfrac12\lambda\b[tb_j\df^*G_{\mu^*}(l^*)+
  tb_j\df^*G^t(l^*)
  -\df^*G^t(l^*)^2+\smfrac14\b]\notag\\
  &=\smfrac12\lambda \b[t b_j \df^* G_{\mu^*}(l^*)
    +\smfrac12b_j\df^*G^t(l^*)+\smfrac14\b]\notag\\
  &=\smfrac12\lambda \b[\smfrac12b_j\df^*G_{\mu^*}(l^*)+\smfrac12tb_j\b(\df^*G_{\mu^*}(l^*)+\df^*G_{\nu^*}(l^*)\b)+\smfrac14\b],\label{eq:trans}
\end{align}
where we use \eqref{eq:IBP} and the definition of $G_{\mu^*}$ and
$G_{\nu^*}$ as Green's
functions; to arrive at the penultimate line, we factorise and use
\eqref{eq:gradGt(l)}, and use the definition of $G^t$ to obtain the
final line.}
As a consequence of Theorem~\ref{th:interior_asymptotics}, we have
\begin{gather}
  b_j\df^* G_{\mu^*}(l^*) = b_j^2\df \Gdst([0,0+\avec])+n^{-1}\B[b_j^2
  \nabla\bar{y}_j(x_j)\cdot \avec+\sum_{i\neq j} b_jb_i\nabla 
  \Gc_{x_i}(x_j)\cdot \avec\B]
  +o(n^{-1}),\label{eq:dG_core_1}\\
  b_j\b[\df^* G_{\mu^*}(l^*)-\df^*G_{\nu^*}(l^*)\b]
  =2b_j^2\df \Gdst([0,0+\avec])+o(n^{-1}),\label{eq:dG_core_2}\\
  \text{and}\qquad b_j\b[\df^* G_{\mu^*}(l^*)+\df^*G_{\nu^*}(l^*)\b]
  =2n^{-1}\B[b_j^2
  \nabla\bar{y}_j(x_j)\cdot \avec+\sum_{i\neq j} b_jb_i\nabla 
  \Gc_{x_i}(x_j)\cdot \avec\B]+o(n^{-1})\label{eq:dG_core_3}
\end{gather}
Using \eqref{eq:dG_core_2}, it follows that
\begin{align}
  t &= \frac{\smfrac12 b_j +b_j^2\df \Gdst([0,0+\avec])+n^{-1}\B[b_j^2\nabla 
  \bar{y}_j(x_j)\cdot \avec+\sum_{i\neq j} b_ib_j\nabla 
  \Gc_{x_i}(x_j)\cdot \avec\B]+o(n^{-1})}
  {b_j+2b_j^2 \df \Gdst([0,0+\avec])+o(n^{-1})},\notag\\
  &=\frac12+n^{-1}\frac{b_j\nabla \bar{y}_j(x_j)\cdot \avec
  +\sum_{i\neq j}b_i\nabla \Gc_{x_i}(x_j)\cdot \avec}
  {1+2b_j\df \Gdst([0,0+\avec])}+o(n^{-1}).\label{eq:t_asym}
\end{align}
Substituting \eqref{eq:dG_core_1}, \eqref{eq:dG_core_3} and \eqref{eq:t_asym}
into \eqref{eq:trans}, we obtain
\begin{equation*}
  \Tran_n(\mu\to\nu) = \smfrac18\lambda+\smfrac14\lambda\df \Gdst([0,0+\avec])
  +\smfrac12\lambda n^{-1} \B[\nabla 
  \bar{y}_j(x_j)\cdot \avec+\sum_{i\neq j} b_jb_i\nabla 
  \Gc_{x_j}(x_i)\cdot \avec\B]
  +o(n^{-1}).
\end{equation*}
Finally, setting $c_0:=\smfrac18+\smfrac14\df \Gdst([0,0+\avec])$ 
completes the proof of Theorem~\ref{th:energy_barriers}.

\alth{
\section{Proofs of Large Deviations results}
\label{sec:LDP_proofs}
This section is devoted to the proofs of the Large Deviations
Principles. \S\ref{sec:LDP_synthesis_proof} verifies
Theorem~\ref{th:LDP_synthesis} by using the results of \cite{FK06}.
Theorem~\ref{th:LDP_main} is then split into the cases where $\L^*$ is
a Bravais lattice, i.e. $\L=\Hx$ or $\L=\Sq$, and where $\L^*$ is a multi--lattice, i.e. where $\L=\Tr$.
These separate cases are covered by
Lemma~\ref{th:LDP_Hx+Sq} and Lemma~\ref{th:LDP_Tr}, and the proofs of
these results constitute the remainder of the section.
}

\alth{
\subsection{Proof of Theorem~\ref{th:LDP_synthesis}}
\label{sec:LDP_synthesis_proof}
Where not otherwise stated, all references given in this section
are to results in \cite{FK06}.

Conditions (1)--(4) assumed in Theorem~\ref{th:LDP_synthesis}
are particular cases of assumptions of
Theorem~6.14. The only additional conditions we need to verify to
apply this theorem are first, that the equation
\begin{equation}
  F_\delta\b(x,f(x),\nabla f(x)\b):=f(x)-\delta\Ham\b(x,\nabla 
  f(x)\b)-h(x)=0,\label{eq:viscosity_eqn}
\end{equation}
where $F_\delta:E\times\R\times\R^N\to\R$,
satisfies a comparison principle for all $\delta>0$ sufficiently small,
and second, that the domain of $H$ is dense in $\CC(E;\R)$. The second
condition is immediate, since $H$ is defined on $\CC^1(E;\R)$.

We recall that a comparison principle
is the statement that viscosity sub-- and supersolutions of
\eqref{eq:viscosity_eqn} are globally ordered.
When $x$ lies on the boundary of
$M$, $\Ham$ vanishes, hence $F_\delta(x,r,p) = r-h(x)$ for all
$x\in\partial M$. Thus any subsolution $\overline{f}$ and supersolution 
$\underline{f}$ must satisfy
\begin{equation*}
  \overline{f}(x)\leq \underline{f}(x)\quad
  \text{for all }x\in\partial M.
\end{equation*}
Theorem~3.3 in \cite{CIL92} asserts that
$F_\delta$ satisfies a comparison principle on the interior of $M$ if
\begin{enumerate}
\item There exists $\gamma>0$ such that
  \begin{equation*}
    \gamma(r-s)\leq F_\delta(x,r,p)-F_\delta(x,s,p)
  \end{equation*}
  for all $x$ in the interior of $M$, $r,s\in\R$ and $p\in\R^N$; and
\item There exists a function $\omega:[0,+\infty)\to[0,+\infty)$ with
  $\lim_{t\to0}\omega(t) = 0$, such that
  \begin{equation*}
    F_\delta\b(x,r,\alpha(x-y)\b)-F_\delta\b(y,r,\alpha(x-y)\b)\leq
    \omega\b(\alpha|x-y|^2+|x-y|\b)
  \end{equation*}
  for all $x$ and $y$ in the interior of $M$, $\gamma\in\R$ and
  $r\in\R$.
\end{enumerate}
It is straightforward to verify that the former condition holds with
$\gamma=1$ for $F_\delta$ as defined in \eqref{eq:viscosity_eqn};
since we have assumed uniform
continuity and differentiability of $\Ham$ on the interior of
$M\times\R^N$,
the second condition is also straightforward to verify, since 
$M\times \{\alpha(x-y)\sep x,y\in M\}$ is compact in $M\times\R^N$.
Thus, a
comparison principle holds on the entirety of $M$, and it follows that
the conclusion of Theorem~6.14 holds, i.e. the sequence of Markov
processes satisfies a Large Deviations Principle.

To conclude that the rate function takes a variational form, we will
first apply Corrolary~8.29. This requires us to check the conditions of
Theorem~8.27. In the case considered here, the operators
$\mathbf{H}_\dagger=\mathbf{H}_\ddagger=H$, thus we need only check that
Conditions~8.9, 8.10 and 8.11 hold.

To verify Condition~8.9, we note the following, which demonstrate
that each of the subconditions (1)--(5) are satisfied.
\begin{enumerate}
\item In our case,
  \begin{equation}
    Af(x,u) = u\cdot \nabla f(x),
  \end{equation}
  which is well--defined on $\CC^1(M;\R)$; this set separates points,
  so Condition 8.9.1 is verified.
\item Here, $\Gamma:= M\times \R^N$: for any $x_0\in M$, define
  $x(t)=x_0$ for all $t\in[0,+\infty)$ and $\lambda(ds\times du)
  =\delta_0(du)\times ds$ to verify Condition 8.9.2.
\item Condition 8.9.3 is satisfied by assumption (5).
\item Condition 8.9.4 is trivially satisfied by taking $\hat{K}=M$,
  since $M$ is compact.
\item Condition 8.9.5 is satisfied due to our assumption that
  $\Lag$ satisfies the growth condition \eqref{eq:growth_condition}.
\end{enumerate}

Condition~8.10 is satisfied upon taking $\lambda(ds\times du) =
\delta_{\dot{x}(s)}(u)ds\times du$, where $x$ is the function
whose existence was asserted in \eqref{eq:gradient_flow_existence}.

Finally, to verify Condition~8.11, we follow the Legendre--Fenchel
transform approach described in \S8.3.6.2. Define
\begin{equation*}
  q_f(x):=\cases{\displaystyle
    \partial_p\Ham\b(x,\nabla f(x)\b)
    & x\in M\setminus\partial M,\\
  0 & x\in\partial M.}
\end{equation*}
This is well--defined, is continuous on the interior of $M$, and there
exists a solution to the ODE $\dot{x}=q_f(x)$ with $x(0)=x_0$ and
$x(t)\in M$ for all $t\in[0,+\infty)$, for any initial condition
$x_0\in M$. Therefore, Condition~8.11 holds upon choosing $x$ to be
this solution, and $\lambda(du\times ds)=\delta_{q_f(x(s))}(du)\times ds$.

We have thus verified the assumptions of Corollary~8.29, which allows
us to conclude that the rate functional has a variational
representation as a control problem, given in (8.18) in \cite{FK06}.
To conclude that the rate function takes the precise form we have here, 
where the solution to the minimisation problem over admissible controls 
is stated explicitly,
we may apply an identical proof to that given for Theorem~10.22,
noting that, under our assumptions, $I_0(x)$ is $0$ if $x=x_0$, and
$+\infty$ otherwise. We have therefore proved
Theorem~\ref{th:LDP_synthesis}.
}

\alth{
\subsection{Proof of Theorem~\ref{th:LDP_main}: the
cases $\L=\Hx$ and $\L=\Sq$}
\label{sec:LDP_Hx+Sq}
When the lattice is $\Hx$ or $\Sq$, the respective dual lattices are
isomorphic to $\Tr$ and $\Sq$, and hence the set of nearest neighbour
directions in the dual lattice is always the same; on the other hand,
since $\Tr^*$ is isomorphic to $\Hx$, which is a multi-lattice,
different techniques are required, and we therefore treat this case
separately in the following section.

Take $f\in\CC^1(\Posc;\R)$; as $\Posc$ is compact, there exists a
uniform 
modulus of continuity $\omega_f:[0,+\infty)\to[0,+\infty)$ with
$\lim_{r\to0}\omega_f(r)=0$, such that for $x=(x_1,\ldots,x_m)\in\Posc$ and
$y\in\Posc$. Thus, for all $x\in\Posc$, $n\in\N$ and $\svec\in\R^2$
such that $x_i+\frac1n\svec\in\Posc$, we have
\begin{equation*}
  \b|f(x_1,\ldots,x_j+\smfrac{1}{n}\svec,\ldots,x_m)-f(x) - 
  \smfrac1n\partial_jf(x)\cdot\svec\b|\leq
  \smfrac{|\svec|}{n}\omega_f\b(|x-y|\b).
\end{equation*}
As $n\to\infty$ in the parameter regime we prescribed in
\S\ref{sec:scaling}, with $x$ in the interior of $\Posc$, and a
sequence $x_n\in\Posen$ with $\dist\b(\iota_n(x_n),x\b)\to0$
as $n\to\infty$, we have
\begin{gather*}
  H_nf\circ\iota_n(x_n) = Hf(x)+O\b(\omega_f\b(\dist(\iota_n(x_n),x)\b)\b),
  \quad\text{setting}\\
  Hf(x):=\sum_{i=1}^m
  \sum_{j=1}^{\Num^*} A\B[\exp\B(\partial_j f(x)\cdot 
  \svec_j\B)-1\B]
  \exp\B[-B\B(\nabla\bar{y}_i(x_i) 
  +\sum_{k\neq j}b_ib_k\nabla\Gc_{x_j}(x_i)\B)
  \cdot \svec_j\B],
\end{gather*}
where $\svec_j$ are the nearest neighbour directions in $\L^*$.
If $x_n\in\partial \Posen$, then $H_nf\circ\iota_n(x_n)=0$, so we define
\begin{equation*}
  Hf(x):=0\quad\text{for }x\in\partial\Posc.
\end{equation*}

Recall the definition of $\bar{y}_i$ as the solution to
\begin{equation*}
  -\Delta \bar{y}_i=0\text{ in }\Dom,\qquad\text{and}\qquad
  \bar{y}_i(x) = \smfrac{1}{\V\pi}\log(|x-x_i|)
  \text{ on }\partial\Dom.
\end{equation*}
Following the approach of \cite{ADLGP14} and \cite{ADLGP16}, we
define the renormalised energy for $x\in\Posc$
to be
\begin{equation*}
  \E(x):= -\sum_{\substack{i,j=1\\i\neq j}}^m
  \frac{b_i b_j}{2\V\pi} \log\b(|x_i-x_j|\b)+ 
  \sum_{i,j=1}^mb_ib_j\bar{y}_i(x_j).
\end{equation*}
Recalling the definition of $\Gc_y$ from 
Theorem~\ref{th:energy_barriers}, we have that
\begin{equation*}
  \partial_i \E(x) = \nabla\bar{y}_i(x_i)
  +\sum_{j\neq i}b_ib_j\nabla\Gc_{x_j}(x_i).
\end{equation*}
For $x$ in the interior of $\Posc$, this allows us to write
\begin{equation*}
  Hf(x) = \sum_{i=1}^m \sum_{j=1}^{\Num^*}
  A \B(\cosh\b[\b(\partial_i f(x)-B\partial_i 
  \E(x)\b)\cdot \avec_j\b]-\cosh\b[-B\partial_i \E(x)\cdot \avec_j\b]
  \B).
\end{equation*}
We define the \emph{Hamiltonian}, $\Ham^\L_{A,B}:\Posc\times
\R^{2m}\to\R$, as
\begin{equation*}
  \Ham^\L_{A,B}(x,p):=\cases{\displaystyle
    \sum_{i=1}^m \sum_{j=1}^{\Num^*}
  A \B[\cosh\b(\b[p_i-B\partial_i 
  \E(x)\b]\cdot \svec_j\b)-\cosh\b(-B\partial_i 
  \E(x)\cdot \svec_j\b)\B]
  &x\in\Posc\setminus\partial\Posc,\\
  0&x\in\partial\Posc,
}
\end{equation*}
where $p = (p_1,\ldots,p_m)$ with $p_i\in\R^2$ for each $i$.
The \emph{Lagrangian} is the Legendre--Fenchel transform (for 
further details on this topic, see \S26 in \cite{Rockafellar})
of the Hamiltonian of $\Ham^\L_{A,B}$ with respect to its second
argument, i.e.
\begin{equation*}
  \Lag^\L_{A,B}(x,
  \xi):=\sup_{p\in\R^{2m}}\b\{\<\xi,p\>-\Ham^\L_{A,B}(x,p)\b\},
\end{equation*}
where $\<\cdot,\cdot\>$ is the inner product on $\R^{2m}$
given by $\<\xi,p\>:=\sum_{i=1}^m \xi_i\cdot p_i$. We now follow 
\cite{BP15} in defining $\Psi^\Hx_{A,B},\Psi^\Sq_{A,B}:
\R^{2m}\to\R$ via
\begin{align*}
  \Psi^\Hx_{A,B}(f)&:=\sum_{i=1}^m\sum_{j=1}^{6} \smfrac{A}{B}
  [\cosh( Bf_i\cdot \avec_j)-1],\\
  \Psi^\Sq_{A,B}(f)&:=\sum_{i=1}^m\sum_{j=1}^4 \smfrac{A}{B}
  [\cosh( Bf_i\cdot \evec_j)-1],
\end{align*}
which permits us to write
\begin{equation*}
  \Ham^\L_{A,B}(x,p) = B\b[\Psi^\L_{A,B}\b(\smfrac1Bp-
  \nabla\E(x)\b)-\Psi_{A,B}^\L
  \b(-\nabla \E(x)\b)\b].
\end{equation*}
$\Psi_{A,B}^\L$ is (strictly) convex, and hence has a convex 
dual, given by its Legendre--Fenchel transform,
denoted $\Phi^\L_{A,B}$. Moreover, by properties of the 
Legendre--Fenchel transform, we have that
\begin{align*}
  \nabla\Phi^\L_{A,B}(\xi) = \smfrac{1}{B}p-\nabla \E
  (x) 
  \quad&\Leftrightarrow\quad
  \xi = \nabla \Psi^\L_{A,B}\b(\smfrac{1}{B}p-\nabla\E
  (x)\b) \\
  &\Leftrightarrow\quad
  p\in\argmax_{p'}\{\<\xi,p'\>-\Ham^\L_{A,B}(x,p')\}.
\end{align*}
Using this fact, we have that
\begin{align*}
  \Lag^\L_{A,B}(x,\xi) &= \<B\nabla\Phi^\L_{A,B}(\xi)
  +B\nabla \E(x),\xi\>-\Ham^\L_{A,B}\b(x,B\nabla
  \Phi^\L_{A,B}(\xi)
  +B\nabla \E(x)\b)\\
   &=\<B\nabla\Phi^\L_{A,B}(\xi)+B\nabla \E(x),
   \xi\>-B\Psi_{A,B}^\L\b(\nabla\Phi^\L_{A,B}(\xi)\b)
   +B\Psi^\L_{A,B}\b(-\nabla\E(x)\b)
\end{align*}
Using the property that $\<u,v\> = \Psi^\L_{A,B}(u)+\Phi^\L_{A,B}(v)$,
we then have
\begin{equation*}
  \Lag^\L_{A,B}(x,\xi) = B\Phi^\L_{A,B}(\xi)+B\Psi^\L_{A,B}
  \b(-\nabla\E(x)\b)
  +B\<\nabla\E(x),\xi\>,
\end{equation*}
which leads us to define the rate functional
$\Act^\L_{A,B}:\DD([0,T];\Posc)\to \R$ with
\begin{gather*}
  \Act^\L_{A,B}(x):=\cases{\displaystyle
    \hspace{-1mm}\int_0^\infty \hspace{-2mm}
    \Lag^\L_{A,B}\b(x,\dot{x}\b)\dt
    &x\in \WW^{1,1}\b([0,+\infty);\Posc\b),\\
    +\infty &\text{otherwise,}
  }\\[1mm]
  \text{and}\quad\Lag^\L_{A,B}(x,q):=\cases{\displaystyle
    \Phi^\L_{A,B}(q)+\Psi^\L_{A,B}\b(-\nabla\E(x)\b)
    +\<\nabla\E(x),q\> & x\in \Posc\setminus \partial\Posc,\\
    0 & x\in\partial\Posc \text{ and }q=0,\\
    +\infty & x\in\partial\Posc \text{ and }q\neq0.}
\end{gather*}
We may now state the following result, asserting a Large
Deviation Principle for the sequence of processes in this case.}

\alth{
\begin{lemma}
\label{th:LDP_Hx+Sq}
Suppose that $\L=\Hx$ or $\L=\Sq$, and that $X^n_0=\iota_n(x^n)$ where $x^n\to 
x_0\in\Posen$ 
as $n\to\infty$ with $\delta>\eps$. Then the processes $X^n_t$
satisfy a Large Deviations principle with good rate function
$\Act^\L_{A,B}$.
\end{lemma}
\medskip

This result is very similar to those obtained in Chapter 5,
\S2 of \cite{FW12}, or \S10.3 of \cite{FK06}, the main difference
being that there is no diffusive part of the process.
We also refer the reader to \cite{BG99} for related results
concerning a discrete--time model on a lattice.

\begin{proof}
As stated, we wish to apply Theorem~\ref{th:LDP_synthesis}
to prove Lemma~\ref{th:LDP_Hx+Sq}. The main conditions
we are required to check are (3), (4) and (5), since conditions (1) and
(2) are straightforward to check when $M_n=\Posen$, $M=\Posc$, and
$\iota_n$ is as defined in \eqref{eq:iotan_defn}.

\subsubsection*{Verifying Condition (3)} $\Posen$ is a
finite state space and is endowed with a topology which is equivalent
to the discrete topology. Therefore, by the Lebesgue Decomposition
Theorem (see for example Theorem~3 of \S1.6 in \cite{EvansGariepy}) all 
measures $\mu$ on $\Posen$ may be expressed as
\begin{equation*}
  \mu(dx) = \sum_{y\in\Posen}f(y)\delta_y(dx),
\end{equation*}
where the density $f:\Posen\to\R$ is continuous (as are all real--valued
functions on $\Posen$). Next, fix a probability measure
$\mu_y(dx,0)=\delta_y(dx)$ on $\Posen$, and define
\begin{gather*}
  \mu_y(dx,t) = \sum_{z\in\Posen}f_y(z,t)\delta_z(dx),
\end{gather*}
where $f_y:\Posen\times[0,+\infty)\to\R$ solves $\partial_tf_y(x,t)=
\Omega_n^T f(x,t)$ with $\Omega_n^T$ being the adjoint of $\Omega_n$,
and  $f_y(x,0) = \delta_y(x)$. It is straightforward to see
that $f_y$ exists and is unique, since $\Omega_n^T$ is a bounded
linear operator, and therefore the ODE system
$\partial_t f_y=\Omega_n^T f_y$ has a unique solution for all time,
so the martingale problem is well--posed. Moreover, the mapping from
$y$ to $\mu_y$ is trivially measurable, since the topology on $M_n$ is
the discrete topology.

\subsubsection*{Verifying Condition (4)}
$\Ham^\L_{A,B}$ clearly satisfies the regularity
conditions required, since $\E$ is harmonic on the interior
of $M$, and $\cosh$ is smooth, and hence $\Ham^\L_{A,B}$ is smooth on the
interior of
$M\times\R^{2m}$. The third condition holds by definition, and since
$x\mapsto\cosh(Bx\cdot\avec)$ is a convex function on $\R^2$ for any
fixed $\avec\in\R^2$, $\Ham^\L_{A,B}$ is convex in $p$ for any
$x\in\Posc$.

We take $g=Hf$ for $f\in\CC^1(\Posc;\R)$ and
$f_n=f\circ\iota_n$, which trivially satisfies the required convergence
condition. Since $f\in\CC^1(\Posc;\R)$, it is straightforward
to check that $g_n=H_nf_n$ is uniformly bounded, since $\nabla f$ is
uniformly continuous, and $\Ham^\L_{A,B}$ is smooth, so the convergence
statement made at the beginning of \S\ref{sec:LDP_Hx+Sq} holds uniformly
for sequences $\iota_n(x_n)$ which approximate points in the interior
of $\Posc$. When $(x,p)$ is in the interior of $\Posc\times\R^{2m}$,
$\Ham^\L_{A,B}$ is continuous, so verification of
\eqref{eq:u+l_convergence} follows from the same arguments.
When $x\in\partial\Posc$, there are two possible limits, either $0$ or
the limiting value for $\lim_{y\to x}\Ham^\L_{A,B}(y,p)$ for sequences
of points $y$ lying in the interior of $\Posc$. Supposing
$\iota_n(x_n)\to x\in
\partial M$, we therefore have
\begin{equation*}
  \min\B\{\lim_{y\to x}\Ham^\L_{A,B}(y,p),0\B\}\leq
  \liminf_{n\to\infty} H_nf_n(x_n)\leq \limsup_{n\to\infty} H_nf_n(x_n)\leq
  \max\B\{\lim_{y\to x}\Ham^\L_{A,B}(y,p),0\B\},
\end{equation*}
which verifies the statement \eqref{eq:u+l_convergence}.

\subsubsection*{Verifying Condition (5)}
For $x\in\partial\Posc$, we have that
\begin{equation*}
  \Lag^\L_{A,B}(x,\xi) = \cases{0 &\xi=0,\\ +\infty &\xi\neq 0},
\end{equation*}
so \eqref{eq:growth_condition} is trivially satisfied. Next, 
using hyperbolic trigonometric identities and the fact that
$|\sinh(x)|\leq \cosh(x)$ for all $x\in\R$, we obtain
\begin{equation*}
  \cosh\b([p-\nabla\E(x)]\cdot\avec\b)
  -\cosh\b([-\nabla\E(x)]\cdot\avec\b)\leq
  2\cosh(|p|)\cosh(|\nabla\E(x)|).
\end{equation*}
Applying this estimate to the definition of $\Ham^\L_{A,B}$, we find that
for some $M$ sufficiently large,
\begin{equation*}
  \Ham^\L_{A,B}(x,p)\leq M\sum_{i=1}^m\cosh(|p_i|).
\end{equation*}
Define $\psi(y):=M\cosh(|y|)$ for $y\in\R^2$; it may be verified that
the Legendre--Fenchel transform of this function, $\psi^*$, is
\begin{equation*}
  \psi^*(y)=|y|\sinh^{-1}\B(\frac{|y|}{M}\B)-\sqrt{1+\frac{|y|^2}{M^2}}.
\end{equation*}
By the ordering properties of the Legendre--Fenchel transform, we
therefore have
\begin{equation*}
  \Lag^\L_{A,B}(x,\xi)\geq \sum_{i=1}^m\psi^*(\xi_i),
\end{equation*}
and since $\sinh^{-1}(r)\to\infty$ as $r\to\infty$, we have that
$\frac{\psi(\xi_i)}{|\xi_i|}\to\infty$ as $|\xi_i|\to\infty$, thus
$\Lag^\L_{A,B}$ satisfies \eqref{eq:growth_condition}.

Next, recalling that $\Psi^\L_{A,B}$ and $\Phi^\L_{A,B}$ are conjugate
functions, we have that for any $\alpha,\beta\in\R^{2m}$,
\begin{equation*}
  \Psi^\L_{A,B}(\alpha)+\Phi^\L_{A,B}(\beta)\geq\<\alpha,\beta\>,
\end{equation*}
where equality holds if and only if $\beta=\nabla\Psi^\L_{A,B}(\alpha)$.
This implies that $\Lag^\L_{A,B}(x,\dot{x})\geq0$ for all
$(x,\dot{x})\in\Posc\times\R^{2m}$, and $\Lag^\L_{A,B}(x,\dot{x})=0$ when $x$
lies in the interior of $\Posc$ if and only if
\begin{equation*}
  \dot{x} = \nabla\Psi^\L_{A,B}\b(-\nabla\E(x)\b).
\end{equation*}
Given that the function on the right--hand side is uniformly Lipschitz
for $x$ in the interior of $\Posc$, it follows that there exists a
solution $x\in\CC([0,T];M)$ to the ODE
\begin{equation*}
  \dot{x}(t)=\nabla\Psi^\L_{A,B}\b(-\nabla\E(x(t))\b),\quad x(0)=x_0,
\end{equation*}
where $T$ is chosen such that $x(T)\in\partial\Posc$, and $x_0$ lies in
the interior of $\Posc$. Then, setting
$x(t)=x(T)$ for all $t>T$, we have that $\dot{x}(t)=0$ for all $t>T$,
and thus
\begin{equation*}
  \int_0^\infty\Lag^\L_{A,B}\b(x(t),\dot{x}(t)\b)\dt=0;
\end{equation*}
we have therefore verified condition (5) of
Theorem~\ref{th:LDP_synthesis}, so applying its conclusion, we
have proved the result.
\end{proof}

To conclude the proof of Theorem~\ref{th:LDP_main}, we note that, by
properties of the Legendre--Fenchel transform, for
$x\in\Posc\setminus\partial\Posc$, $\Lag^\L_{A,B}(x,q)\geq0$,
\begin{equation*}
  \Lag^\L_{A,B}(x,q) = 0 \quad\text{if and only if}\quad
  q=\nabla\Psi^\L_{A,B}\b(-\nabla \E(x)\b),
\end{equation*}
and by definition, $\mathcal{M}^\L_{A,B}=\nabla\Psi^\L_{A,B}$.
}

\alth{
\subsection{Proof of Theorem~\ref{th:LDP_main}: the
case $\L=\Tr$}
The case where $\L=\Tr$ is more complicated than the cases treated
above, since $\Tr^*$ is isomorphic to $\Hx$, which is a multi--lattice
rather than a simple Bravais lattice.
The value of $\Omega_n f$ therefore oscillates
depending upon the specific sublattice on which each dislocation lies,
and so the verification of the convergence condition in
Theorem~\ref{th:LDP_synthesis} requires an additional step in this case.
The technique which we use to prove convergence bears significant
similarities to the use of a periodic `corrector' as used in the
theory of homogenisation for differential operators with rapidly
oscillating coefficients, and our approach may be viewed as the discrete
analogue of the strategy used in Example~1.10 in \cite{FK06}.

For clarity, we first fix some notation which we use throughout the
proof: recall from
\S\ref{sec:dual_examples} the definition of $\avec_i$, and the fact that
$\Tr^*$ is the union of 2 translated copies of
$\smfrac{\sqrt{3}}{3}\Tr$. It will therefore be convenient to define
\begin{equation*}
  \avec^*_i = \smfrac13(\avec_{2i}+\avec_{2i-1})
  \quad\text{for }i=1,2,3,\quad
  \Tr^*_+:=\Tr+\avec^*_1,\quad\text{and}\quad 
  \Tr^*_-:=\Tr-\avec^*_2.
\end{equation*}
By definition, we have that $\Tr^*=\Tr^*_+\cup\Tr^*_-$;
the subscripts refer to the fact that the nearest neighbour directions
in $\Tr^*$ are
\begin{gather*}
  \b\{\avec^*_1,\avec^*_2,\avec^*_3\b\}\quad\text{for }e^*\in\Tr^*_+
  \quad\text{and}\quad
\b\{-\avec^*_1,-\avec^*_2,-\avec^*_3\b\}\quad\text{for }e^*\in\Tr^*_-.
\end{gather*}
With this notation, if $\mu^*=(e^*_1,\ldots,e^*_m)$ with
$r_\infty(\mu^*,x)=O(n^{-1})$, we have
\begin{multline*}
  \Omega_n f(\mu) = \sum_{i|e^*_i\in \Tr^*_+}\sum_{j=1}^3
  nA\exp\b[-B\partial_i\E(x)\cdot\avec^*_j+o(1)\b]
  \b[f(e^*_1,\ldots,e^*_i+\avec_j^*,\ldots,e^*_m)-f(\mu)\b]\\
  +\sum_{i|e^*_i\in \Tr^*_-}\sum_{j=1}^3
  nA\exp\b[B\partial_i\E(x)\cdot\avec^*_j+o(1)\b]
  \b[f(e^*_1,\ldots,e^*_i-\avec_j^*,\ldots,e^*_m)-f(\mu)\b].
\end{multline*}
We see that the generator oscillates in value depending upon whether
each $e_i^*\in\Tr^*_+$ or $e_i^*\in\Tr^*_-$. To obtain a Large
Deviations Principle, we must show that the nonlinear generator
converges in the sense of condition~(4) in
Theorem~\ref{th:LDP_synthesis}. We suppose that $f\in\CC^1\b(\Posc;\R)$,
and define a sequence $f_n(\mu) = f\circ\iota_n(\mu)+\smfrac1n
h_f(\iota_n(\mu);\mu)$, where $h_f:\Posc\times(\Tr^*)^m\to\R$ will be
defined shortly. For convenience, we also define
$T^i_{\svec}\mu:=(e_1,\ldots,e_i+\svec,\ldots,e_m)$, and calculate
\begin{align*}
  H_nf(\mu)&= \smfrac1n\e^{-nf(\mu)}\Omega_n\e^{nf}(\mu)\\
  &=A\sum_{i|e^*_i\in \Tr^*_+}\sum_{j=1}^3
  \exp\b[-B\partial_i\E(x)\cdot\avec^*_j\b]
  \b[\exp\b(\partial_if(x)\cdot\avec^*_j+h_f(x,e_1,T^i_{\avec_j^*}\mu)
    -h_f(x,\mu)\b)-1\b]\\
  &\qquad\qquad+A\sum_{i|e^*_i\in \Tr^*_-}\sum_{j=1}^3
  \exp\b[B\partial_i\E(x)\cdot\avec^*_j\b]
  \b[\exp\b(-\partial_if(x)\cdot\avec^*_j+h_f(x,T^i_{-\avec^*_j}\mu)-h_f(x,\mu)\b)-1\b]
  \\
  &\qquad\qquad\qquad\qquad+o(n^{-1}).
\end{align*}
Our aim is now to define $h_f$ such that for some $g\in\CC(\Posc;\R)$,
\begin{equation*}
\sup_{\mu\in\Posen}\b|H_n(f\circ\iota_n+\smfrac1nh_f)(\mu)
-g\circ\iota_n(\mu)\b|\to0\quad\text{as }n\to\infty.
\end{equation*}
As long as $h_f(x,\mu)$ is uniformly bounded for
$(x,\mu)\in\Posc\times(\Tr^*)^m$, this will imply the convergence
condition required in Theorem~\ref{th:LDP_synthesis}.
We make the ansatz that 
\begin{equation*}
  h_f(x;e_1^*,\ldots,e_m^*)=\sum_{i=1}^mh_{f,i}(x;e^*_i),\quad\text{where}
  \quad h_{f,i}(x;e_i^*)=\cases{h_{f,i}^+(x) &e_i^*\in\Tr^*_+,\\[1mm]
  h_{f,i}^-(x) &e_i^*\in\Tr^*_-;}
\end{equation*}
thus, each $h_{f,i}:\Posc\times\Tr^*\to\R$ depends only on whether
$e^*_i\in\Tr^*_+$ or $e^*_i\in\Tr^*_-$. In order that
$H_n(f\circ\iota_n+\smfrac1nh_f)-g\circ\iota_n$ tends to zero
independently of the choice of sublattice for each $e_i^*$,
we then choose $h^\pm_{f,i}(x)$ to satisfy the `corrector problem'
\begin{align}
  g(x) &= 
  \sum_{i=1}^m\sum_{j=1}^{3}
  A \B[\exp\b(\b[\partial_if(x)-B\partial_i 
  \E(x)\b]\cdot \avec_j^*+h_{f,i}^-(x)-h_{f,i}^+(x)\b)
  -\exp\b(-B\partial_i\E(x)\cdot \avec_j^*\b)\B]\notag\\
  &=\sum_{i=1}^m\sum_{j=1}^3A \B[\exp\b(\b[B\partial_i 
  \E(x)-\partial_if(x)\b]\cdot \avec_j^*+h_{f,i}^+(x)
  -h_{f,i}^-(x)\b)-\exp\b(B\partial_i 
  \E(x)\cdot \avec_j^*\b)\B].\label{eq:Tr_corrector_eq}
\end{align}
Equating terms which contain $\pm[h_{f,i}^+(x)-h_{f,i}^-(x)]$ and
solving, we set
\begin{gather*}
h_{f,i}^\pm(x) = \pm\smfrac12\log\bg(\frac{(\gamma_i^+-\gamma^-_i)+
  \sqrt{(\gamma^+_i-\gamma^-_i)^2+4\delta^+_i\delta^-_i}}{2\delta^+_i}\bg),\\
  \text{and thus}\quad g(x) = A\sum_{i=1}^m
  \sqrt{\smfrac14(\gamma^+_i+\gamma^-_i)^2+\delta^+_i\delta^-_i
    -\gamma^+_i\gamma^-_i}-\smfrac12(\gamma^+_i+\gamma^-_i),\\
  \text{where}\quad
  \gamma_i^\pm=\sum_{j=1}^3\exp\b(\mp B\partial_i\E(x)
  \cdot\avec_j^*\b)\quad\text{and}\quad
  \delta_i^\pm=\sum_{j=1}^3\exp\b(\pm\b[\partial_if(x)
  -B\partial_i\E(x)\b]\cdot\avec_j^*\b).
\end{gather*}
By the convexity of the exponential function and the fact that
$\avec_1^*+\avec_2^*+\avec_3^*=0$, we have
$\gamma_i^\pm,\delta^\pm_i\geq 3$; in
addition, $\sqrt{(\gamma_i^+-\gamma_i^-)^2+4\delta^+_i\delta^-_i}
+\gamma_i^+-\gamma_i^-\geq0$, so $h^\pm_{f,i}(x)$ is
well--defined for all $x$. Since $\gamma_i^\pm$ and $\delta_i^\pm$ are
continuous functions of $x\in\Posc$, we also have that $h^\pm_{f,i}(x)$
depend continuously on $x$, and is thus uniformly bounded for
$x\in\Posc$.

By now expressing $g(x)$ in terms of hyperbolic trigonometric functions,
we define $Hf(x):=\Ham^\Tr_{A,B}\b(x,\nabla f(x)\b)$, where the limiting
Hamiltonian $\Ham^\Tr_{A,B}(x,p)$ is defined to be
\begin{gather*}
  \Ham^\Tr_{A,B}(x,p):=\sum_{i=1}^m\sqrt{
    \Upsilon_{A,B}\b[\partial_i\E(x)\b]^2
    +\Psi^\Tr_{A,B}\b[\smfrac1Bp_i-\partial_i\E(x)\b]
    -\Psi^\Tr_{A,B}\b[-\partial_i\E(x)\b]}
  -\Upsilon_{A,B}\b[\partial_i\E(x)\b],\\
  \text{for }x\in\Posc\setminus\partial\Posc,\quad\text{and}\quad
  \Ham^\Tr_{A,B}(x,p):=0\quad\text{for }x\in\partial\Posc,\\
  \text{where}\quad\Psi^\Tr_{A,B}[\xi]:=A^2\sum_{j=1}^6\cosh
  \b[B\xi_i\cdot\avec_j\b],
  \quad\text{and}\quad
  \Upsilon_{A,B}[\xi]:=A\sum_{j=1}^3\cosh[B\xi_i\cdot \avec_j^*].
\end{gather*}
We define the conjugate function $\Lag^\Tr_{A,B}:\Posc\times\R^{2m}\to\R
\cup\{+\infty\}$ to be 
\begin{equation*}
  \Lag^\Tr_{A,B}(x,\xi):=\sup_{p\in\R^{2m}}\b\{\xi\cdot p-
\Ham^\Tr_{A,B}(x,p)\b\},
\end{equation*}
and the corresponding rate functional $\Act:\DD([0,+\infty);\Posc)\to
\R\cup\{+\infty\}$ to be
\begin{equation*}
  \Act^\Tr_{A,B}(x):=\cases{\displaystyle
    \int_0^\infty\Lag^\Tr_{A,B}\b(x,\dot{x}\b)\dt
    &x\in\WW^{1,1}\b([0,+\infty);\R^{2m}\b),\\
    +\infty &\text{otherwise.}}
\end{equation*}
We now state the following theorem, which asserts the existence of a
Large Deviations Principle for the model for dislocation motion for the
case $\L=\Tr$.

\begin{lemma}
\label{th:LDP_Tr}
Suppose that $\L=\Tr$, and that $X^n_0=\iota_n(x^n)$ where
$x^n\to x_0\in\Posc$ as $n\to\infty$. Then the processes $X^n_t$
satisfy a Large Deviations Principle with good rate function
$\Act^\Tr_{A,B}$.
\end{lemma}
\medskip

Once more, we prove this result by checking the conditions of
Theorem~\ref{th:LDP_synthesis}.

\begin{proof}
As in the proof of Lemma~\ref{th:LDP_Hx+Sq}, conditions (1) and (2)
are straightforward to verify with $M_n=\Posen$ and $M=\Posc$, and
condition~(3) holds by an identical argument.

\subsubsection*{Verifying Condition (4)}
It is clear from the arguments of the previous section that
$\Ham^\Tr_{A,B}(x,p)$ satisfies the necessary regularity conditions, and
by definition the $\Ham^\Tr_{A,B}$ vanishes for $x\in\partial\Posc$;
the convexity condition is also evident for $x\in\partial\Posc$.
Next, let $x$ lie in the interior of $\Posc$: then the second derivative
of $\Ham^\Tr_{A,B}(x,p)$ with respect to $p_i$ is
\begin{multline*}
  \partial_{p_i}^2\Ham^\Tr_{A,B}(x,p) =
  \frac{\smfrac1{2B^2}\nabla^2
    \Psi^\Tr_{A,B}[\smfrac1Bp_i-\partial_i\E(x)]}
  {\b(\Upsilon_{A,B}[\partial_i\E(x)]^2+\Psi^\Tr_{A,B}[\smfrac1Bp_i-\partial_i\E(x)]-\Psi^\Tr_{A,B}[-\partial_i\E(x)]\b)^{1/2}}\\
  -\frac{\smfrac1{4B^2}\nabla\Psi^\Tr_{A,B}[\smfrac1Bp_i-\partial_i\E(x)]\otimes
  \nabla\Psi^\Tr_{A,B}[\smfrac1Bp_i-\partial_i\E(x)]}
  {\b(\Upsilon_{A,B}[\partial_i\E(x)]^2
    +\Psi^\Tr_{A,B}[\smfrac1Bp_i-\partial_i\E(x)]
    -\Psi^\Tr_{A,B}[-\partial_i\E(x)]\b)^{3/2}}.
\end{multline*}
To verify convexity of $\Ham^\Tr_{A,B}$, we check that this matrix is
positive definite. This reduces to verifying that, as symmetric
matrices,
\begin{equation*}
  \smfrac12\B(\Upsilon_{A,B}[\zeta]^2-\Psi^\Tr_{A,B}[\zeta]
  +\Psi^\Tr_{A,B}[\xi]\B)
  \nabla^2\Psi^\Tr_{A,B}[\xi]
  -\smfrac14\nabla\Psi^\Tr_{A,B}[\xi]\otimes\nabla\Psi^\Tr_{A,B}[\xi]\geq0
  \quad\text{for all }\xi,\zeta\in\R^2.
\end{equation*}
Pre--multiplying by $v^T$ and
post--multiplying the matrices in the above expression by $v$ for some
$v\in\R^2$, we have
\begin{gather*}
  \nabla^2\Psi^\Tr_{A,B}[\xi]:[v,v]
  =\sum_{j=1}^6 A^2B^2\cosh[B\xi\cdot\avec_j]
  (v\cdot\avec_j)^2,\\
  \text{and}\qquad
  \b(v\cdot\nabla\Psi^\Tr_{A,B}[\xi]\b)^2
  =\bg(\sum_{i=1}^6A^2B\sinh[B\xi\cdot\avec_j](v\cdot \avec_j)\bg)^2.
\end{gather*}
It is immediate that $\nabla^2\Psi^\Tr_{A,B}[\xi]:[v,v]\geq0$ for all
$v\in\R^2$, since $\cosh$ is bounded below by $1$, and the vectors
$\avec_j$ span $\R^2$. Next, we note that
\begin{multline*}
  \smfrac12\Psi^\Tr_{A,B}[\xi]\,\nabla^2\Psi^\Tr_{A,B}[\xi]:[v,v]
  -\smfrac14\b(v\cdot\nabla\Psi^\Tr_{A,B}[\xi]\b)^2\\
  =\frac12\sum_{j,k=1}^6A^4B^2\cosh[B\xi\cdot\avec_j]
  \cosh[B\xi\cdot\avec_k]
  (v\cdot \avec_j)^2\\
  -\frac14\sum_{j,k=1}^6A^4B^2\sinh[B\xi\cdot\avec_j]
  \sinh[B\xi\cdot \avec_j](v\cdot\avec_j)(v\cdot\avec_k).
\end{multline*}
Using the identity $(v\cdot\avec_j)(v\cdot\avec_k) = 
\smfrac12[v\cdot(\avec_j+\avec_k)]^2-\smfrac12(v\cdot\avec_j)^2
-\smfrac12(v\cdot\avec_k)^2$ and the symmetry of the vectors $\avec_j$,
we have
\begin{multline}
  \smfrac14\sum_{j,k=1}^6\sinh[B\xi\cdot\avec_j]
  \sinh[B\xi\cdot \avec_k](v\cdot\avec_j)(v\cdot\avec_k)\\
  =\smfrac18\sum_{j,k=1}^6\sinh[B\xi\cdot\avec_j]
  \sinh[B\xi\cdot \avec_k][v\cdot(\avec_j+\avec_k)]^2-\smfrac14
  \sum_{j,k=1}^6\sinh[B\xi\cdot\avec_j]
  \sinh[B\xi\cdot \avec_k](v\cdot\avec_j)^2.\label{eq:TriHam_convexity1}
\end{multline}
The latter sum vanishes, since $\sinh$ is an odd function and
$v\cdot\avec_k=-v\cdot\avec_{k+3}$ for $k=1, 2$ or $3$. Splitting the
sum, interchanging indices $j$ and $k$ and then using convexity and the fact that $\cosh$ is postive, we have
\begin{align}
  \smfrac12\sum_{j,k=1}^6\cosh[B\xi\cdot\avec_j]\cosh[B\xi\cdot\avec_k]
  (v\cdot \avec_j)^2&=\smfrac12\sum_{j,k=1}^6\cosh[B\xi\cdot\avec_j]
  \cosh[B\xi\cdot\avec_k]
  \b(\smfrac12(v\cdot \avec_j)^2+\smfrac12(v\cdot \avec_k)^2\b)\notag\\
  &\geq \smfrac18\sum_{j,k=1}^6\cosh[B\xi\cdot\avec_j]
  \cosh[B\xi\cdot\avec_k]
  \b[v\cdot (\avec_j+\avec_k)\b]^2.\label{eq:TriHam_convexity2}
\end{align}
Combining \eqref{eq:TriHam_convexity1} and \eqref{eq:TriHam_convexity2},
and using the addition formula for hyperbolic cosine, then bounding
$\cosh$ below by $1$ and dropping all terms except for those where
$j=k$, we find that
\begin{multline}
  \smfrac12\Psi^\Tr_{A,B}[\xi]\,\nabla^2\Psi^\Tr_{A,B}[\xi]:[v,v]
  -\smfrac14\b(v\cdot\nabla\Psi[\xi]\b)^2\\
  \geq \smfrac18
    A^4B^2\sum_{j,k=1}^6\cosh\b[B\xi\cdot(\avec_j+\avec_k)\b]
    [v\cdot(\avec_j+\avec_k)]^2,
  \\\geq \smfrac12A^4B^2\sum_{j=1}^6(v\cdot\avec_j)=\smfrac32A^4B^2|v|^2
      \label{eq:TriHam_convexity3}
\end{multline}
It remains to verify that $\b(\Upsilon_{A,B}[\zeta]^2-\Psi^\Tr_{A,B}
[\zeta]\b)\smfrac12\nabla^2\Psi^\Tr_{A,B}[\xi]\geq0$ for all
$\xi,\zeta\in\R^2$. This is immediate upon noting that
\begin{equation}
  \Upsilon_{A,B}[\zeta]^2-\Psi^\Tr_{A,B}[\zeta]
  =\smfrac14(\gamma_i^++\gamma_i^-)^2-
  \gamma_i^+\gamma_i^-=\smfrac14(\gamma_i^+-\gamma_i^-)^2=
  \bg(\sum_{j=1}^6 A\sinh[B\xi\cdot\avec_j]\bg)^2 \geq0,
  \label{eq:TriHam_convexity4}
\end{equation}
and using the positive--definiteness of $\nabla^2\Psi^\Tr_{A,B}[\xi]$.
Estimates \eqref{eq:TriHam_convexity3} and \eqref{eq:TriHam_convexity4}
entail that $\partial^2_{p_i}\Ham^\Tr_{A,B}(x,p)$ is strictly positive
definite
for all $x$ in the interior of $\Posc$, and therefore $\Ham^\Tr_{A,B}$
satisfies the convexity condition.

To verify that the convergence requirement of Condition~(4) is
satisfied, we define $h_n:\Posen\to\R$ to be
$h_n(\mu):=h\b(\iota_n(\mu),\mu)$. Then as $h_f(x,e_1^*,\ldots,e_m^*)$
is uniformly bounded for all $x\in\Posc$ and $e_i^*\in\Tr^*$, so
\begin{equation*}
  \b\|f\circ\iota_n+\smfrac1nh_n - f\circ\iota_n\b\| \leq cn^{-1}\to0
  \quad\text{as }n\to\infty.
\end{equation*}
Since $\nabla f$, $\partial_i\E$ and $x\mapsto h_f(x,\mu)$
are uniformly continuous on $\Posc$, and
$\Ham^\Tr_{A,B}$ is smooth and hence uniformly continuous on the
interior of $\Posc\times B_r(0)$ for any $r>0$, we have that
$x\mapsto\Ham^\Tr_{A,B}\b(x,\nabla f(x)\b)$ is uniformly continuous.
Using the fact that $h_f$ was chosen to satisfy
\eqref{eq:Tr_corrector_eq}, it is now straightforward to check that 
\begin{equation*}
  \b\|H_n\b(f\circ\iota_n+\smfrac1n h_n\b)(\mu)-Hf\circ \iota_n\b\|\to0\quad\text{as }n\to\infty,
\end{equation*}
and so convergence is verified.

\subsubsection*{Verifying Condition~(5)}
Given that $\Ham^\Tr_{A,B}$ is a significantly more complex function than
the Hamiltonians obtained in the previous cases, we do
not have as explicit an expression for $\Lag^\Tr_{A,B}$ as we obtained in
the cases where $\L=\Sq$ and $\L=\Hx$. We therefore
verify Condition~(5) indirectly using properties of the
Legendre--Fenchel transform.

First, we verify that $\Lag^\Tr_{A,B}(x,\xi)\geq0$. We note that since
$\Ham^\Tr_{A,B}(x,p)$ is smooth and strictly convex in $p$,
$\Lag^\Tr_{A,B}(x,\xi)$ is also smooth and strictly convex, and
$\Ham^\Tr_{A,B}(x,p)=\sup_{\xi\in\R^{2m}}\b\{p\cdot\xi
-\Lag^\Tr_{A,B}(x,\xi)\b\}$. It
follows that
\begin{equation*}
  0=\Ham^\Tr_{A,B}(x,0)=\sup_{\xi\in\R^{2m}}\b\{-\Lag^\Tr_{A,B}(x,\xi)\b\} =
  -\inf_{\xi\in\R^{2m}}\Lag^\Tr_{A,B}(x,\xi).
\end{equation*}

To verify the growth condition \eqref{eq:growth_condition}, we estimate
$\Ham^\Tr_{A,B}(x,p)$ above. Using the elementary inequality
$\sqrt{a+b}\leq\sqrt{a}+\sqrt{b}$ for 
any $a,b\geq0$, the AM--GM inequality, and the property that
$\gamma_i^\pm\geq3$, we find
\begin{align}
  \sqrt{\smfrac14(\gamma_i^+-\gamma_i^-)^2+\delta_i^+\delta_i^-}
  -\smfrac12(\gamma_i^++\gamma_i^-)&\leq
     \smfrac12|\gamma_i^+-\gamma_i^-|
  +\sqrt{\delta_i^+\delta_i^-}-\smfrac12(\gamma_i^++\gamma_i^-)\notag\\
  &\leq \smfrac12(\delta_i^++\delta_i^-)-\min\{\gamma_i^+,\gamma_i^-\}
  \leq \smfrac12(\delta^+_i+\delta^-_i).\label{eq:TriHam_convexity5}
\end{align}
Noting that $\cosh(v\cdot\avec_j)\leq \cosh(\smfrac{\sqrt{3}}{3}|v|)$,
formula \eqref{eq:TriHam_convexity5}, along with the definition of
$\Ham^\Tr_{A,B}$, the convexity of $\cosh$ and the fact that
$\partial_i\E(x)$ is uniformly bounded for all $x\in\Posc$, implies that
there exists a constant $C>0$ independent of $x$ such that
\begin{equation*}
  \Ham^\Tr_{A,B}(x,p)\leq \sum_{i=1}^m\sum_{j=1}^3A\cosh\b([p_i-B\partial_i\E(x)]\cdot\avec_j^*\b)\leq \sum_{i=1}^m\smfrac32A\cosh\b[\smfrac{\sqrt{3}}3|p_i|\b]+C.
\end{equation*}
A similar argument to that used in the proof of
Theorem~\ref{th:LDP_Hx+Sq} now allows us to conclude that
\eqref{eq:growth_condition} also holds in this case.

Next, we note that
\begin{equation*}
  0=\partial_\xi\Lag^\Tr_{A,B}(x,\xi) \quad\text{if and only if}
  \quad \xi=\partial_p\Ham^\Tr_{A,B}(x,0).
\end{equation*}
Computing $\partial_p\Ham^\Tr_{A,B}$, we find that if $x$ solves
\begin{equation}
  \dot{x}_i=\frac{\nabla\Psi^\Tr_{A,B}[-\partial_i\E(x)]}
  {2\Upsilon_{A,B}[\partial_i\E(x)]}\quad\text{with}
  \quad x(0)=x_0,\label{eq:Tr_ODE}
\end{equation}
where $x_0$ is in the interior of $\Posc$, then
\eqref{eq:gradient_flow_existence} is verified. As $\Psi^\Tr_{A,B}$ and
$\Upsilon_{A,B}$ are smooth, $\Upsilon_{A,B}$ is bounded below, and
$\partial\E(x)$ is bounded on $\Posc$, an identical argument
to that given in the proof of Theorem~\ref{th:LDP_Hx+Sq} entails that
this condition is satisfied.

Having now verified all conditions of Theorem~\ref{th:LDP_synthesis},
its application implies Lemma~\ref{th:LDP_Tr}. 
\end{proof}

Finally, upon noting
that $\Lag^\Tr_{A,B}$ is minimised when \eqref{eq:Tr_ODE} is satisfied,
and setting
\begin{equation*}
  \mathcal{M}^\Tr_{A,B}(\xi)=\frac{\nabla\Psi^\Tr_{A,B}[\xi]}
  {2\,\Upsilon_{A,B}[\xi]},
\end{equation*}
we have proved Theorem~\ref{th:LDP_main}.
}

\section*{Acknowledgements}
\noindent
{\bf Thanks:} The author would like to thank Giovanni Bonaschi and 
Giacomo Di~Ges\'u for informative discussions on Large Deviations 
Principles while carrying out this work, and the two anonymous
referees for helpful suggestions of a variety of improvements
to this paper.
\medskip

\noindent
{\bf Funding:} This study was funded by a public grant overseen by the 
French National Research Agency (ANR) as part of the 
``Investissements d'Avenir'' program (reference: ANR-10-LABX-0098).
\medskip

\noindent
{\bf Conflict of interest:} The author declares that there is no conflict 
of interest regarding this work.

\bibliographystyle{plain}
\bibliography{dislKMC}

\end{document}